\documentclass[11p,a4paper]{article}
\usepackage{amsfonts}
\usepackage{amssymb}
\usepackage{amsthm}
\usepackage{amsmath}
\usepackage{graphicx}
\usepackage{mathrsfs}
\usepackage{graphics}
\usepackage{abstract}
\usepackage{color}
\usepackage{cite}
\usepackage{physics}
\usepackage[shortlabels]{enumitem}
\usepackage{empheq}
\usepackage{appendix}
\usepackage{mathrsfs}
\usepackage{bm} 
\usepackage[colorlinks=true,citecolor=red]{hyperref}
\hypersetup{linkcolor=blue}
\graphicspath{ {./result/} }
\usepackage{titling}
\usepackage{authblk}
\usepackage{geometry}
\usepackage{fancyhdr}
\usepackage{lineno}
\usepackage{titlesec}
\titleformat{\section}[block]{\centering\Large\bfseries}{\thesection}{1em}{}
\titleformat{\subsection}[block]{\centering\large\bfseries}{\thesubsection}{1em}{}
\titleformat{\subsubsection}[block]{\centering\normalsize\bfseries}{\thesubsubsection}{1em}{}
\usepackage{ulem}

\newtheorem{theorem}{\textbf{Theorem}}[section]
\newtheorem{lemma}{\textbf{Lemma}}[section]
\newtheorem{proposition}{\textbf{Proposition}}[section]
\newtheorem{corollary}{\textbf{Corollary}}[section]
\newtheorem{remark}{\textbf{Remark}}[section]
\newtheorem{definition}{\textbf{Definition}}[section]

\providecommand{\keywords}[1]{\textbf{\textit{Keywords---}} #1}
\allowdisplaybreaks[4]

\numberwithin{equation}{section}
\def\be{\begin{equation}}
\def\ee{\end{equation}}
\def\bea{\begin{eqnarray}}
\def\eea{\end{eqnarray}}
\def\bt{\begin{theorem}}
\def\et{\end{theorem}}
\def\bl{\begin{lemma}}
\def\el{\end{lemma}}
\def\br{\begin{remark}}
\def\er{\end{remark}}
\def\bp{\begin{proposition}}
\def\ep{\end{proposition}}
\def\bc{\begin{corollary}}
\def\ec{\end{corollary}}
\def\bd{\begin{definition}}
\def\ed{\end{definition}}

\def\Pi{\mathbf{\psi}}

\def\R{{\mathbb R}}

\def\N{{\mathbb N}}
\def\T{{\mathbb  T}}
\def\P{{\mathbb  P}}
\def\G{{\mathbb G}}
\def\Im{{\mathcal I}}
\usepackage{tikz}
\usetikzlibrary{calc,arrows,intersections}
\title{\bf{Stability and bifurcation of Navier-Stokes equations in an annular
domain with mixed boundary conditions}}

\author{
Zhibo Hou$^{a}$
\thanks{houzhibo@mail.xhu.edu.cn}
\quad\quad
Liang Li$^{b}$
\thanks{liliang187@gxu.edu.cn}
\quad\quad
Quan Wang$^{c}$
\thanks{Corresponding author:xihujunzi@scu.edu.cn }
\\ \footnotesize $^{a}$
School of Science, Xihua University, Chengdu, Sichuan, 610039, P.R. China
 \\ \footnotesize $^{b}$
 School of Mathematics and Information Science, Guangxi University,
\\ \footnotesize 
Nanning, Guangxi, 530004, P.R.China
  \\ \footnotesize $^{c}$ College of Mathematics, Sichuan University,
  \footnotesize
 Chengdu, Sichuan, 610064,  P.R. China
}
\medskip

\begin{document}
%\linenumbers
%%%%%%%%%%%%%%%%%%%%%%%%%%%%%%%%%%%%%%%
\maketitle
\begin{abstract}

We study the existence and stability of non-trivial steady-state solutions to the two-dimensional incompressible Navier–Stokes equations posed in the annulus $\Omega = B(0,b) \setminus \overline{B(0, a)}$ with radii $b>a>0$. 
The outer boundary $\partial B(0, b)$ is subject to the free condition
while the inner boundary $\partial B(0, a)$
 obeys a Navier-slip condition with effective slip length  $\alpha > 0$. 
 Our main results are fourfold. First, we establish global-in-time strong solutions and derive a sharp energy estimate that underpins the subsequent nonlinear instability analysis. Second, for $\alpha > 0$, we compute an explicit critical viscosity $\mu_c:=\mu_c(\alpha, a,b,\mu)$ that separates qualitatively different dynamical behaviours.  Third, we precisely characterize the stability properties of the trivial solution in three distinct regimes.
 The zero solution is globally asymptotically stable in $H^2$ if $\mu > \mu_c$.
If $\mu = \mu_c$, we prove an alternative theorem that completely describes the local dynamics near the trivial state. If $\mu < \mu_c$, the trivial solution is nonlinearly unstable in every $L^p$ ($p \geq 1$).
Finally, we demonstrate that for $\mu < \mu_c$, the system undergoes a pitchfork bifurcation that generates an infinite family of non-trivial steady states. For generic choices of $(\alpha,b)$, this bifurcation is supercritical; a measurable subset of parameter space yields subcritical transitions.  Notably, all bifurcating solutions share the same topological pattern—a single row of counter-rotating vortices—despite their mathematical non-uniqueness.\\
\keywords{Navier-Stokes equations; Global stability; Dynamic bifurcation.}
\end{abstract}

\newpage
\tableofcontents

\newpage
\section{Introduction}
Consider the following incompressible two-dimensional Navier–Stokes equations
   \begin{align}\label{main}
\begin{cases}
\frac{\partial \mathbf{u}}{\partial t}+ (\mathbf{u} \cdot \nabla )\mathbf{u} = \mu \Delta \mathbf{u}-
\nabla P+\mathbf{f}(\mathbf{x}),\quad \mathbf{x}\in \Omega,
\\ \nabla \cdot  \mathbf{u}=0, \quad \mathbf{x}\in \Omega,
\end{cases}
\end{align}
where $ \mathbf{u}=(u_1,u_2)$ denotes the velocity vector field, $P$ denotes the pressure, $\mu$ is the viscous coefficient, $\Omega$ is a subset of $\R^2$, and $\mathbf{f}(\mathbf{x})$ is a prescribed, time-independent body-force term. Any nontrivial steady-state solution of
\eqref{main} represents a time-independent motion determined by
   \begin{align}\label{main-s}
\begin{cases}
(\mathbf{u} \cdot \nabla )\mathbf{u} = \mu \Delta \mathbf{u}-
\nabla P+\mathbf{f}(\mathbf{x}),\quad \mathbf{x}=(x,y)\in \Omega,
\\ \nabla \cdot  \mathbf{u}=0, \quad \mathbf{x}=(x,y)\in \Omega.
\end{cases}
\end{align}

The existence of steady-state solutions and their stability of \eqref{main-s} are inextricably intertwined with a host of significant mathematical and scientific inquiries. 
For the steady-state solutions involved shear flows, such as
 Couette flow \cite{Masmoudi2022,Chen2020,Ding2020}, Poiseuille flow \cite{Zelati2020} and other types shear flows 
 \cite{Ibrahim2019, Bedrossian2018}, they are an important research topic in applied mathematics and fluid sciences. Steady Navier–Stokes solutions are connected to a broad spectrum of problems. Leray’s problem \cite{Leray1933, Korobkov2015}) addresses the existence of weak solutions in arbitrary bounded, multiply connected planar or axially symmetric domains under the sole necessary condition of zero net flux through the boundary. Liouville-type theorems investigate uniqueness or non-uniqueness for classical shear flows and beyond  \cite{Zhang2025, Bang2025}. Additional topics include the existence, uniqueness, and asymptotic structure of three-dimensional exterior-domain flows  \cite{Galdi2011}, bifurcation phenomena in the Taylor problem \cite{Temam1977}, and uniqueness questions in certain unbounded domains \cite{Amick1980}.
 
This network of interconnections highlights the profound significance of investigating non-trivial steady solutions of \eqref{main-s} as a central theme in contemporary applied mathematics.  By seamlessly integrating rigorous theory with real-world applications, this field deepens our understanding of complex fluid behaviours and the governing mathematical principles.  Readers seeking further insight are encouraged to consult the extensive literature, notably the penetrating analyses of Gajjar $\&$ Azzam \cite{Gajjar2004}, the comprehensive studies of Perazzo $\&$ Gratton \cite{Perazzo2004}, the innovative contributions of Bazant $\&$  Moffatt \cite{Bazant2005}, the recent advances by Wang $\&$  Wu \cite{Wang2023}, and the authoritative treatise of Galdi $\&$  Yamamoto \cite{Galdi2025}, together with the references therein.

If \(\Omega\subset\mathbb{R}^2\) is unbounded and \(\mathbf{f}(\mathbf{x})\equiv\mathbf{0}\), system \eqref{main-s} readily admits a continuum of non-trivial steady solutions.  For instance, on either \(\Omega=\mathbb{R}^2\) or \(\Omega=\mathbb{T}\times\mathbb{R}\), the pair \(\mathbf{u}=(U(y),0)\) solves \eqref{main-s} for every non-constant profile \(U(y)\) satisfying \(U''(y)=0\); this is the classical shear flow whose stability has been thoroughly investigated \cite{Masmoudi2024}.  When \(\mathbf{f}(\mathbf{x})\not\equiv\mathbf{0}\) and \(\Omega\) is bounded, system \eqref{main-s} can still support non-trivial steady solutions.  For example, let \(\Omega=\mathbb{T}\times(0,1)\) and take the body force \(\mathbf{f}(\mathbf{x})=(\cos\pi y,0)\).  Then the pair \(\mathbf{u}=(\mu^{-1}\pi^{-2}\cos\pi y,0)\) satisfies \eqref{main-s}; this is the celebrated Kolmogorov flow, whose stability has been extensively investigated \cite{Wei2020}.  When the fluid is confined to a bounded domain \(\Omega\) and is free of any external forcing (\(\mathbf{f}(\mathbf{x})\equiv\mathbf{0}\)), it becomes a fundamental question whether system \eqref{main-s} supports non-trivial steady states—stable or unstable—that differ from the classical shear or Kolmogorov flows \cite{Lu2020}, or Taylor vortices \cite{Ma2006}.  The existence of such solutions could uncover previously unknown pathways to turbulence and substantially enrich our understanding of nonlinear pattern formation in confined fluid systems.

Let $\Omega$ be a bounded domain with a smooth boundary $\partial \Omega$. When the system \eqref{main-s} is subjected to no-slip boundary conditions, specifically  $\mathbf{u}|_{\partial \Omega}=\mathbf{0}$, this represents the classical case that has long been a focal point for both physicists and mathematicians. Extensive research has been dedicated to exploring various critical aspects of the system \eqref{main} under these conditions. These aspects include the existence and uniqueness of solutions, their regularity properties, and the behavior of solutions in the limit of vanishing viscosity. This body of work has significantly advanced our understanding of the underlying fluid dynamics and the mathematical principles governing such systems. Let $\mathbf{u}$ be any classical steady-state solution to \eqref{main} subject to $\mathbf{u}|_{\partial \Omega}=\mathbf{0}$, 
testing $\eqref{main-s}_1$ by $\mathbf{u}$, integrating by part, we can get from $$\int_{\Omega}\left(\mu \abs{\nabla \mathbf{u}}^2+\mathbf{f}\cdot \mathbf{u}\right)\,dx\,dy=0,\quad \mathbf{u}|_{\partial \Omega}=\mathbf{0}$$
that $\mathbf{u}=\mathbf{0}$, provided $\mathbf{f}\equiv 0$. As a result,
system \eqref{main-s} on a bounded domain and free of any external forcing 
($\mathbf{f}\equiv 0$) permits only the zero steady solution.

It is important to note that the no-slip boundary condition
$\mathbf{u}|_{\partial \Omega}=\mathbf{0}$ can be regarded as a special case, specifically when $\alpha=-\infty$, of the more general Navier-slip boundary conditions  \cite{ Naiver1827}
  \begin{align}\label{cond-100}
    \begin{aligned}
&\mathbf{u} \cdot   \mathbf{n}|_{\partial\Omega}=0,~~~\left[\left(-p\mathbf{I} +\mu\left (\nabla \mathbf{u}+\left(\nabla \mathbf{u}\right)^{Tr} \right)\right)\cdot\mathbf{n}\right]
   \cdot   \mathbf{\tau}|_{\partial\Omega}=\alpha  \mathbf{u}  \cdot   \mathbf{\tau}|_{\partial\Omega}.
  \end{aligned}
    \end{align}
In this formulation, $\mathbf{n}$ denotes the outward normal vector to the boundary $\partial\Omega$, while 
  $\mathbf{\tau}$ represents the tangent vector. Within the context of the Navier-slip boundary condition \eqref{cond-100}, the parameter $\alpha$ holds significant physical meaning. 
  It can be a constant, a bounded function \cite{Kelliher2006}, or even a smooth matrix \cite{Gie2012}. For our current discussion, we focus on the scenario where $\alpha$ is a constant. When it comes to such problems, the case where $\alpha\leq 0$
  is of particular interest.
  This situation generally reflects the frictional interactions between the fluid and the boundary. It is the classical case and has garnered substantial attention from both physicists and mathematicians. We direct interested readers to the relevant literature, including Badur \cite{Badur2011}, Bailey \cite{Bailey2017}, Bleitner \cite{Bleitner2024}, Itoh \cite{itoh2008}, Cao \cite{Cao2023}, Ding et al. \cite{ding2018stability}, Quarisa \cite{Quarisa2021}, Tapia\cite{Acevedo2018, AcevedoTapia2021}, Li\cite{Liaa2025} and Goodair\cite{Goodair2025}, as well as the references cited therein, for further exploration of this important area of research.

 Let $\mathbf{u}$ be any classical steady-state solution to the system
\eqref{main} subject to Navier slip boundary conditions \eqref{cond-100} and
 \(\mathbf{f}(\mathbf{x})\equiv\mathbf{0}\). Denote
$\kappa$ as the curvature of $\partial\Omega$. By testing $\eqref{main-s}_1$ with $\mathbf{u}$, integrating by parts, and utilizing Navier slip boundary conditions, we obtain:
\begin{align}\label{nontrivial-condition}
\int_{\Omega}|\nabla \mathbf{u}|^2\,dx\,dy
=\int_{\partial \Omega}\frac{\alpha-\kappa}{\mu }\left(\mathbf{u}\cdot \mathbf{\tau}\right)^2\,ds.
\end{align}
This equation implies that $\mathbf{u}=\mathbf{0}$ provided
$\left(\kappa-\alpha\right)>0$.

 To investigate the existence of nontrivial steady-state solutions to the system \eqref{main-s} subject to Navier-slip boundary conditions and under the condition 
$\kappa-\alpha<0$ and  \(\mathbf{f}(\mathbf{x})\equiv\mathbf{0}\), we consider the case where $\Omega$ has a smooth boundary $\partial\Omega$ with 
 nonzero constant curvature.  
 Without loss of generality, let $\Omega$
be an annular domain in
$\R^2$ as depicted in \autoref{region},
and suppose that the velocity $\mathbf{u}$ is then subject to the following boundary conditions:
  \begin{align}\label{cond-1}
  &\mathbf{u} \cdot   \mathbf{n}|_{\abs{\mathbf{x} }=a}=0,\quad
  \mathbf{u} \cdot   \mathbf{n}|_{\abs{\mathbf{x} }=b}=\nabla \times   \mathbf{u}|_{\abs{\mathbf{x} }=b}=0,\\ \label{cond-2}
   &\left[\left(-p\mathbf{I} +\mu\left (\nabla \mathbf{u}+\left(\nabla \mathbf{u}\right)^{Tr} \right)\right)\cdot\mathbf{n}\right]
   \cdot   \mathbf{\tau}|_{\abs{\mathbf{x} }=a}
=\alpha  \mathbf{u}  \cdot   \mathbf{\tau}|_{\abs{\mathbf{x} }=a},
  \end{align}
  where $\mathbf{n}$ is the outward normal vector to the boundary $\partial\Omega$, 
  $\mathbf{\tau}$ is the tangent vector, and $\alpha $ is a positive constant.
 Please note that \eqref{cond-1} represents
 the stress-free boundary condition, while \eqref{cond-2} is the
 Navier-slip boundary condition, which is a chosen interpolation between stress-free and no-slip boundary conditions. The conditions \eqref{cond-1}-\eqref{cond-2} are physically valid when the boundary $\abs{\mathbf{x} }=a$ is rough and 
 the outer boundary $\abs{\mathbf{x} }=b$ is free from any applied loads or constraints.
 For $\Omega$ with zero constant curvature (flat domain), we note that
 Ding et al. \cite{ding2018stability} studied a similar problem by analyzing the stability and instability of the trivial solution $\mathbf{u}=\mathbf{0}$.
 \begin{figure}[tbh]
\centering \label{region}
        \begin{tikzpicture}[>=stealth',xscale=1,yscale=1,every node/.style={scale=1.5}]
\fill[domain = -2:360,gray][samples = 200] plot({2*cos(\x)}, {2*sin(\x)});
        \fill[domain = -2:360,white][samples = 200] plot({1/2*cos(\x)}, {1/2*sin(\x)});
\draw [thick,->] (0,-2.6) -- (0,2.6) ;
\draw [thick,->] (-2.8,0) -- (2.8,0) ;
%\draw [thick,->] (1.4,1.4) -- (0.5,0.5) ;
\node [below right] at (3.1,0) {$x$};
\node [below right] at (1/2,0) {$a$};
\node [below right] at (2,0) {$b$};
%\node [below left] at (0.5,0.5) {$\mathbf{g}$};
\node [left] at (0,3.1) {$y$};
\draw[domain = -2:360][samples = 200] plot({2*cos(\x)}, {2*sin(\x)});
\draw[domain = -2:360][samples = 200] plot({1/2*cos(\x)}, {1/2*sin(\x)});
\draw[fill,black] (0.5,0) circle [radius=2pt];
\draw[fill,black] (2,0) circle [radius=2pt];
\end{tikzpicture}
\caption{Schematic representation of the domain $\Omega$}
\end{figure}
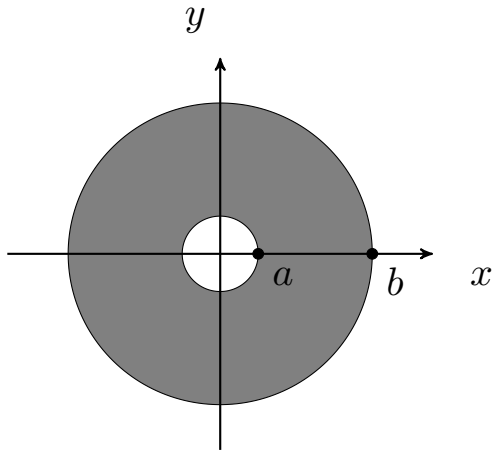

 While the Navier boundary value problem with $\alpha>0$ 
 has garnered relatively less attention, it is important to note that,
as Serrin \cite{Serrin1959} pointed out in 1959, the parameter $\alpha$ does not necessarily need to possess a defined sign. Indeed, this scenario is not only theoretically plausible but also finds practical relevance in fluid sciences.  For instance, in the context of flat hybrid gas-liquid surfaces, the eﬀective slip length $\alpha$ is invariably positive \cite{Haase2016}. 
Moreover, Navier-slip boundary condition \eqref{cond-2} with $\alpha>0$ is also generally used in other problems.  For example, in the study of boundary walls that accelerate the fluid, the Navier-slip boundary condition with a positive $\alpha$ is commonly used to account for the slip effect, which can significantly influence the fluid flow characteristics \cite{Magnaudet1995}.
For more examples, we refer readers to the problems of weather forecasts and hemodynamics \cite{Beavers1967, Chauhan1993}, as well as the moving contact line problem \cite{Ren2007, Guo2018} and other problems \cite{Arbon2025, Chang2025}. 

For the annular domain as depicted in \autoref{region}, under the condition 
$a^{-1}<\alpha$, intuitively, we see from \eqref{nontrivial-condition} 
and $\lim\limits_{\mu\to 0^{+}}\frac{\alpha-\kappa}{\mu }=+\infty$ that
there should exist viscosity threshold $\mu_c$ of $\mu$ above which 
$\mathbf{u}=\mathbf{0}$ is the only steady-state solution, but the system \eqref
{main-s} subject to the boundary conditions \eqref{cond-1}-\eqref{cond-3}
may permit other steady-state solutions. 
This analysis naturally leads to seven fundamental mathematical open questions:
    \begin{enumerate}
        \item What is the exact analytical expression for the critical viscosity threshold $\mu_c$?
        \item What are the precise stability characteristics of the trivial solution $\mathbf{u}=\mathbf{0}$ when $\mu > \mu_c$?
        \item How does the system \eqref{main} behave dynamically in the vicinity of $\mathbf{u}=\mathbf{0}$ when $\mu = \mu_c$?
        \item What is the well-posedness theory for the system \eqref{main} with the mixed boundary conditions \eqref{cond-1}–\eqref{cond-3} when 
 effective slip length $\alpha > 0$ on the inner boundary?
        \item What type of bifurcation occurs at the critical threshold $\mu = \mu_c$?
        \item What is the multiplicity of bifurcating solutions (i.e., their number and bifurcation structure)?
        \item What are the stability characteristics of these nontrivial bifurcated solutions?
    \end{enumerate}
In this article, we shall study the above seven questions.

   \subsection{Main results}
To consider the stability of the steady-state solution $(\mathbf{u},P)=(\mathbf{0},p_0)$. Denote the perturbation by
\[
\mathbf{v}=\mathbf{u}-\mathbf{0},\quad P=q-p_0.
\]
Then $(\mathbf{v},q)$ sastisfies the perturbed equations 
    \begin{align}\label{main=p1}
\begin{cases}
\frac{\partial \mathbf{v}}{\partial t}+ (\mathbf{v} \cdot \nabla )\mathbf{v} = \mu \Delta \mathbf{v}-
\nabla q\quad \mathbf{x}\in \Omega,
\\ \nabla \cdot  \mathbf{v}=0, \quad \mathbf{x}\in \Omega,\\
\mathbf{v}(0)=\mathbf{v}_0,
\end{cases}
\end{align}
 subject to the corresponding bondary conditions as follows
  \begin{align}\label{cond-3}
    \begin{aligned}
  &\mathbf{v} \cdot   \mathbf{n}|_{x^2+y^2=a^2}=0,\quad
  \mathbf{v} \cdot   \mathbf{n}|_{x^2+y^2=b^2}=\nabla \times   \mathbf{v}|_{x^2+y^2=b^2}=0,\\
   &\left[\left(-q\mathbf{I} +\mu\left (\nabla \mathbf{v}+\left(\nabla \mathbf{v}\right)^{Tr} \right)\right)\cdot\mathbf{n}\right]
   \cdot   \mathbf{\tau}|_{x^2+y^2=a^2}
  =\alpha  \mathbf{v}  \cdot   \mathbf{\tau}|_{x^2+y^2=a^2},
  \end{aligned}
    \end{align}
 and the following average condition
\begin{align}\label{average}
\int_{\partial B(0,r)}\mathbf{v}  \cdot   \mathbf{\tau}\,ds=0,\quad a\leq r\leq b,
\end{align}
  which guarantees the average perturbation of tangential velocity $v_{\theta}$ is zero.
  
   \begin{align}\label{uuu-1}
 \mu_c= \frac{a \alpha \left(1+3\sigma^4-4\sigma^2-
 4 \sigma^4 \log (\sigma)\right)}
 {2 \left(\sigma^4-1-4 \sigma^4 \log (\sigma)\right)}>0 ,\quad \sigma=b/a>1,
 \end{align}

   \begin{theorem}\label{linearstability}
  If $\mu>\mu_c$,  the state $(\mathbf{0},p_0)$ is linearly stable. While, it is linearly unstable if $\mu<\mu_c$.
  \end{theorem}
  The \autoref{linearstability} is deduced from \autoref{corrollary}.
 
 \begin{theorem}\label{existence}
For any given $T>0$ and initial data $\mathbf{v}_0 \in H^{2}\left(\Omega\right) $
satisfying $\nabla \cdot \mathbf{v}_0=0$ and the average condition \eqref{average}, there exists a unique strong solution
$(\mathbf{v},q)\in C([0,T];H^2(\Omega)\times H^1(\Omega))$ to
 the problem \eqref{main=p1} subject to boundary conditions \eqref{cond-3}.
 In addition, there exists $\delta_0 \in (0,1]$ such that
 \begin{align}\label{nonlinear0718}
\begin{aligned}
&\norm{\mathbf{v}}_{H^2}^2
+\norm{\nabla q}_{L^2}^2
+\norm{\mathbf{v}_t}_{L^2}^2
+\int_0^t\left(\norm{\nabla \mathbf{v}}_{L^2}^2
+
\norm{\mathbf{v}_s}_{H^1}^2
\right)\,ds
\\&\leq C_0\left(\norm{\mathbf{v}_{0}}_{H^2}^2
+
\int_0^t\norm{\mathbf{v}}_{L^2}^2\,ds
\right)
\end{aligned}
\end{align}
provided $\norm{\mathbf{v}}_{H^2}^2\leq \delta_0$ on the closed interval $[0,T]$,
where $ C_0$ is independent of $T$.
\end{theorem}
  The proof of \autoref{existence} is given in \autoref{section31}.

  \begin{theorem}\label{non-stability}
If $\mu>\mu_c$,  the steady-state solution $(\mathbf{0},p_0)$ is nonlinearly stable in
$H^2$-norm. More precisely, there exists a positive constant $\lambda_0$ such that
 the global strong solution
$(\mathbf{v}^{\delta},q^{\delta})\in C([0,+\infty);H^2(\Omega)\times H^1(\Omega))$ to the problem 
\eqref{main=p1}-\eqref{average} with the initial data $\mathbf{v}_0\in H^2(\Omega)$
obeying the incompressible condition $\nabla \cdot \mathbf{v}_0=0$
and the average condition \eqref{average}
satisfies
\begin{align}
\norm{\mathbf{v}}_{H^2}
+\norm{q}_{H^1}
\leq C\norm{\mathbf{v}_0}_{H^2}e^{-\lambda_0t}.
\end{align}
\end{theorem}
  The proof of \autoref{non-stability} is given in \autoref{section32}.

\begin{theorem}\label{criticalcase}
If $\mu=\mu_c$,  for the global strong solution
$(\mathbf{v},q)\in C([0,+\infty);H^2(\Omega)\times H^1(\Omega))$ to the problem 
\eqref{main=p1}-\eqref{average}, with the initial data $\mathbf{v}_0\in H^2(\Omega)$
obeying the incompressible condition $\nabla \cdot \mathbf{v}_0=0$
and the average condition \eqref{average}, one of the following results holds:
\begin{enumerate}
    \item[\rm{(1)}] There exists a sequence of invariant sets \(\{ \Gamma_n \} \subset E_0\) such that 
    \[
    0 \notin \Gamma_n, \quad \lim_{n \to \infty} \text{dist}( \Gamma_n, 0 ) = 0,
    \]
where $E_0=\text{Span}\{\mathbf{u}\}$ and $\mathbf{u}$ satisfies the following
system
\[
\begin{cases}
\mu \Delta \mathbf{u}-
\nabla p=0,
\\ \nabla \cdot  \mathbf{u}=0, 
\end{cases}
\]
subject to the boundary conditions \eqref{cond-3} and the average condition \eqref{average} .
 \item[\rm{(2)}] The state $(\mathbf{0},p_0)$ is a locally asymptotically stable equilibrium of  the problem 
\eqref{main=p1}-\eqref{average}  under the $L^2$-norm. Moreover, if the problem 
\eqref{main=p1}-\eqref{average}  has no other invariant sets in $E_0$ except the trivial equilibrium $(\mathbf{0},p_0)$, then $(\mathbf{0},p_0)$ is globally asymptotically stable.
\end{enumerate}
\end{theorem}
The proof of \autoref{criticalcase} is given in \autoref{section33}.

  \begin{theorem}\label{noninstability}
If $\mu<\mu_c$, for any $p\in [1,\infty]$
the steady-state solution $(\mathbf{0},p_0)$ is nonlinearly unstable in
$L^p$-norm. More precisely, there exist two positive constants $\epsilon, C^*$, and a function
$\mathbf{v}_0$ satisfying $\norm{\mathbf{v}_0}_{H^2}=1$ such that
for any $\delta \in (0,\epsilon)$, there exists a global strong solution
$(\mathbf{v}^{\delta},q^{\delta})\in C([0,T];H^2(\Omega)\times H^1(\Omega))$ to the problem 
\eqref{main=p1}-\eqref{average} with the initial data $\mathbf{v}^{\delta}(0)=\delta \mathbf{v}_0$
such that 
\begin{align}
\norm{\mathbf{v}^{\delta}(T_{\delta})}_{L^{p}}\geq \epsilon,
\end{align}
where $T_{\delta}$ is the escape time defined by 
$T_{\delta}=\lambda_1^{-1}\ln \frac{C^*\epsilon}{\delta}\in (0,T)$, in which 
$\lambda_1$ is given by \eqref{variation-one}.
\end{theorem}
The proof of \autoref{noninstability} is given in \autoref{section34}.

\begin{theorem}\label{criticalcase1}
For $\mu$ in the vicinity of $\mu_c$, the dynamics near the trivial solution 
$\mathbf{0}$ of 
 the global strong solution
$(\mathbf{v}^{\delta},q^{\delta})\in C([0,+\infty);H^2(\Omega)\times H^1(\Omega))$ to the problem 
\eqref{main=p1}-\eqref{average}, 
with the initial data $\mathbf{v}_0\in H^2(\Omega)$
obeying the incompressible condition $\nabla \cdot \mathbf{v}_0=0$
and the average condition \eqref{average}, is equivalent to that of 
the two dimensional system
	\begin{align}\label{req-m1-m}
\begin{cases}
	\frac{dz_1}{dt} = \lambda_{1}(\mu) z_1+ lz _1\left(z_1^2+z_2^2\right)+o\left(\abs{z}^{3}\right),\\
		\frac{dz_2}{dt} = \lambda_{1}(\mu) z_2+ l z _2\left(z_1^2+z_2^2\right)+o\left(\abs{z}^{3}\right),
\end{cases}
	\end{align}
where 
$\lambda_1$ is given by \eqref{variation-one} and $l$ is given by \eqref{Coefficient2}.
Furthermore, we have the following results:
\begin{itemize}
\item[ \rm{(1) }] If $l<0$, the bifurcation
at $\mu=\mu_c$  is supercritical.  In this case, the system of equations \eqref{main=p1}-\eqref{average} admits an infinite number of stable steady-state solutions 
$\mathbf{v}^s=(v_1^s,v_2^s)$
\begin{align}\label{flow-p-1}
\begin{aligned}
&v_1^s=\left(-r^{-1}\partial_{\theta}\psi_{s},\partial_r\psi_{s}\right)\cdot(\cos\theta,-\sin\theta),\quad 
\theta=\arctan y/x,\\
&v_2^s=\left(-r^{-1}\partial_{\theta}\psi_{s},\partial_r\psi_{s}\right)\cdot(\sin\theta,\cos\theta),\quad
r=\sqrt{x^2+y^2},
\end{aligned}
\end{align}
where $\psi_{s}$ is the corresponding streamfunction given by
\begin{align*}
\begin{aligned}
\psi_{s}(r,\theta)=&
s\Psi_1e^{i \theta}+\overline{s}\Psi_1e^{-i\theta}
+e^{2i\theta}G_{11} \abs{s}^2
\\&+e^{-2i\theta}\overline{G_{11}}\overline{s}^2+o(s^2),\quad \abs{s}=\sqrt{-\lambda_1/l}.
\end{aligned}
\end{align*}
 for $\mu<\mu_c$. Here, $\Psi_1$ satisfies the following eigenvalue problem
 \[
   \begin{cases}
  \mu\Delta_{1}^2\Psi_1=\lambda_1\Psi_1,\\
 \Psi_1=\Psi''_1+\frac{1}{r}\Psi'_1=0,\quad r=b,\\
  \Psi_1= \Psi''_1-
\left(\frac{1}{r}-\frac{\alpha}{\mu}\right) 
\Psi'_1=0,\quad r=a,
  \end{cases}
 \]
 and $G_{11}$ solves the non-homogeneous problem
 \[
   \begin{cases}
  \mu\Delta_{2}^2G_{11}-2\lambda_1\Delta_{2}G_{11}=-i\left(\frac{\Psi_1}{r}\frac{d}{dr}-\frac{1}{r}
\frac{d\Psi_1}{dr}\right)\Delta_{1}\Psi_{1},\\
G_{11}=G_{11}''+\frac{1}{r}G_{11}'=0,\quad r=b,\\
G_{11}= G_{11}''-
\left(\frac{1}{r}-\frac{\alpha}{\mu}\right) 
G_{11}'=0,\quad r=a,
  \end{cases}
 \]
 where $\Delta_{n}$ is a derivative operator given by
 $\Delta_{n}=\frac{d^2}{dr^2}+\frac{1}{r}\frac{d}{d r}-\frac{n^2}{r^2}$.
These stable steady-state solutions constitute a local ring attractor
 for $\mu<\mu_c$ and $\mu$ very close to $\mu_c$.
\item[  \rm{(2) }] If $l>0$,  the bifurcation at $\mu=\mu_c$  is subcritical. In this case, the system of equations \eqref{main=p1}-\eqref{average} admits an infinite number of stable steady-state solutions given by \eqref{flow-p-1} for $\mu>\mu_c$. These solutions constitute a local ring repeller $\mu>\mu_c$ and $\mu$ very close to $\mu_c$.
\end{itemize}
\end{theorem}
The proof of \autoref{criticalcase1} is given in \autoref{section4}.
\begin{remark}
Based on \autoref{criticalcase1}, one can observe that if $l>0$, then the first conclusion of \autoref{criticalcase} holds, while the second conclusion of \autoref{criticalcase} holds if $l<0$. Our numerical computations show that 
for generic choices
of $(a,b,\alpha)$, the second conclusion of \autoref{criticalcase} holds. However, the first conclusion of \autoref{criticalcase} occurs for certain special parameter regions of $(\alpha,b)$. 
\end{remark}

\subsection{Main technical difficulties behind the proof}

For the Navier-Stokes equations in an annular domain with mixed boundary conditions, as considered in this article, the explicit expression of the first eigenvalue $\lambda_1(\mu)$ remains unknown. To determine the viscosity threshold $\mu_c$, we introduced a nontrivial transformation. This transformation enabled us to convert the system \eqref{polar-l-eigen-3}, which originally has non-constant coefficients, into one with constant coefficients. By doing so, we were able to successfully obtain the viscosity threshold  $\mu_c$ by using the standard Euler method, bypassing the need to rely on $\lambda_1(\mu)$. 

The primary challenge in proving \autoref{existence} is that the left-hand side of \eqref{ddddd} (and similarly \eqref{ut-proof-11}) contains the term $\int_{\Omega}\abs{\nabla\mathbf{v}}^2\, dx\,dy$, which prevents us from directly applying the trace theorem to control the boundary integral
\[
\mu b^{-1}\int_{\partial B(0,b)}\left(\tau \cdot \mathbf{v}_t\right)^2 \,ds
+\int_{\partial B(0,a)}\left(\mu a^{-1}-\alpha\right)\left(\tau \cdot \mathbf{v}_t\right)^2 \,ds
\]
and obtain the necessary estimates \eqref{nonlinear-33}. To achieve \eqref{nonlinear-33}, we derive new estimates for the boundary integral, which are presented in \autoref{epsilon-1}. Additionally, we must extend the corresponding Stokes estimates under the boundary conditions \eqref{cond-1}-\eqref{cond-2}. The estimate for the boundary integral can be derived using \autoref{epsilon-1}, and the generalized Stokes estimates are provided in \autoref{Stokes0725}.

The first key estimate for the proof of \autoref{non-stability} is to control $\norm{\mathbf{v}}_{L^{\infty}\left((0,T)\times W^{1,p}(\Omega)\right)}$. Based on \autoref{lemma-grad} and \autoref{lemma-grad-1}, to estimate $\norm{\mathbf{v}}_{W^{1,p}}$, we only need to control the $L^p$ norm of $\omega=\nabla^{\perp}\cdot\mathbf{v}=(-\partial_y,\partial_x)\cdot\mathbf{v}$. However, the boundary condition \eqref{cond-2} means that $\omega|_{x^2+y^2=a^2}=-\left(2a^{-1}+\alpha \mu^{-1}\right)v_{\tau}$ is not vanish (vorticity production at the boundary). This makes it very difficult to bound $\norm{\omega}_{L^{p}}$. Our idea to overcome this difficulty is first to control
\[
\norm{\left(2a^{-1}+\alpha \mu^{-1}\right)v_{\tau}}_{L^{\infty}((0,T)\times \partial B(0,a))}
\]
for any fixed time $T$, then we bound the modified vorticity $\omega^{\pm}$ by replacing the boundary condition $\omega=-\left(2a^{-1}+\alpha \mu^{-1}\right)v_{\tau}|_{x^2+y^2=a^2}$ by $\pm\norm{\left(2a^{-1}+\alpha \mu^{-1}\right)v_{\tau}}_{L^{\infty}((0,T)\times \partial B(0,a))}$ in the equation of $\omega$. The maximum principle for parabolic equations can be used to derive an estimate of $\norm{\omega^{\pm}}_{L^{p}}$.

The second difficulty arising in the proof of \autoref{non-stability} is that $\mathbf{n}\cdot \nabla q|_{x^2+y^2=a^2}$ contains the term $
\left(2a^{-1}\mu+\alpha\right) 
\tau\cdot \nabla v_{\tau}$ due to the vorticity production at the boundary, leading to an integration
\[
\left(2a^{-1}\mu+\alpha\right)  \int_{\partial B(0,a)}q \tau \cdot\nabla v_{\tau}\,ds
\]
on the boundary. This makes it very difficult to directly get the following required estimate
\[
\left(2a^{-1}\mu+\alpha\right)  \int_{\partial B(0,a)}q \tau \cdot\nabla v_{\tau}\,ds\leq \norm{q}_{H^{1}}  \norm{\mathbf{v}}_{H^{1}}
\]
using the trace theorem. To circumvent this difficulty, we transfer the estimate of
\[
\left(2a^{-1}\mu+\alpha\right)  \int_{\partial B(0,a)}q \tau \cdot\nabla v_{\tau}\,ds
\]
into that of
\[
\left(2a^{-1}\mu+\alpha\right) \int_{\partial B(0,a)} \mathbf{n} \cdot q\nabla^{\perp}(\tau\cdot\mathbf{q}) \,ds
\]
by observing
\[
\left(2a^{-1}\mu+\alpha\right) \int_{\partial B(0,a)}\tau \cdot\nabla \left( v_{\tau}q\right)\,ds=0,
\]
where the boundary integral
\[
\left(2a^{-1}\mu+\alpha\right) \int_{\partial B(0,a)} \mathbf{n} \cdot q\nabla^{\perp}(\tau\cdot\mathbf{v}) \,ds
\]
can be transferred into the volume integral
\[
\int_{\Omega}\nabla \cdot (q\Phi \nabla^{\perp}(\tau\cdot\mathbf{v}) )\,dx
\]
by constructing a function $\Phi$ satisfying $\Phi|_{x^2+y^2=a^2}=2a^{-1}+\alpha \mu^{-1}$ and  $\Phi|_{x^2+y^2=b^2}=0$.

Regarding the nonlinear instability of the incompressible Navier-Stokes equations, several key results have been established in the following literatures: Yudovich \cite{Yudovich1989,Henry1993} demonstrated nonlinear instability in the function space $L^{p}$ with $p\geq n$ for 
$n$-dimensional spatial domains; Friedlander et al. \cite{Vishik1997} developed a fairly general abstract theorem, showing that nonlinear instability in 
$H^s$ (where $s>\frac{n}{2}+1$) can be derived when the linearized operator admits an unstable eigenvalue in 
$L^2$; Friedlander et al.\cite{Friedlander2006} employed a bootstrap argument to prove nonlinear instability in $L^{p}$ for all 
$p\in(1,+\infty)$, albeit restricted to bounded domains with Dirichlet boundary conditions or an n-dimensional torus $\T^n$. Notably, the findings in \cite{Yudovich1989, Henry1993, Vishik1997} are subject to inherent limitations, while the work in \cite{Friedlander2006}, despite its advancements, remains constrained by the applicability of Dirichlet boundary conditions or 
$\T^n$, limiting its scope to specific scenarios.
In this paper, we also utilize the bootstrap method to investigate nonlinear instability, but with a critical methodological innovation: instead of relying on semigroup properties as in \cite{Friedlander2006}, we adopt a variational approach to establish the existence of linearly growing solutions. This methodological shift enables broader applicability beyond restrictive boundary conditions, allowing us to prove nonlinear instability in 
$L^{p}$
 for all 
$p\in[1,+\infty]$. Furthermore, the adoption of Navier boundary conditions enhances the physical realism of our results, aligning more closely with practical scenarios.

For bifurcation, the primary challenge is that we cannot directly utilize the original system \eqref{main=p1}-\eqref{cond-3} to prove \autoref{criticalcase1}. To analyze the bifurcation types, we first derive an equivalent reformulation of \eqref{main=p1}-\eqref{cond-3}, given by \eqref{omega-2}. This equivalent system enables the application of separation of variables, which reduces the problem to a low-dimensional dynamical system. Through this reduced system, we can rigorously determine the bifurcation types at the critical threshold.

 \section{Eigenvalue problem and viscosity threshold }

  In this section, we delve into the eigenvalue problem associated with the linearized equation of the system \eqref{main} at the trivial steady-state solution  $(\mathbf{u},p)=(\mathbf{0},p_0)$. This eigenvalue problem serves a dual purpose: it is instrumental in determining a viscosity threshold for analyzing the transition from linear stability to linear instability of  $(\mathbf{u},p)=(\mathbf{0},p_0)$, and it also plays a pivotal role in the proof of our main theorems.
  To this end, linearizing \eqref{main} around the steady state $(\mathbf{u},p)=(\mathbf{0},p_0)$ yields the linearized equations
  (linear part of the perturbation system \eqref{main=p1})
      \begin{align}\label{main=p-l}
\begin{cases}
\frac{\partial \mathbf{v}}{\partial t}= \mu \Delta \mathbf{v}-
\nabla q\quad \mathbf{x}\in \Omega,
\\ \nabla \cdot  \mathbf{v}=0, \quad \mathbf{x}\in \Omega,
\end{cases}
\end{align}
subject to the mixed boundary condition \eqref{cond-3}.

For the sake of convenience, we will employ the following notations throughout this paper:
\begin{align}\label{funt-sapce-1}
&X_0(\Omega)=\left\{
\mathbf{v}\in L^2(\Omega)\,\middle|\,\nabla \cdot  \mathbf{v}=0,\int_{\partial B(0,r)}\mathbf{v}  \cdot   \mathbf{\tau}\,ds=0
\right\},\\
&X_1(\Omega)=\left\{
\mathbf{v}\in X_0\cap H^1(\Omega)
\,\middle|\,
\mathbf{v} \cdot   \mathbf{n}|_{\abs{\mathbf{x} }=a, b}=0
\right\},\\
&X_2(\Omega)=\left\{\mathbf{v}\in H^2(\Omega)
\,\middle|\,\mathbf{v}~\text{satisfies}~\eqref{cond-3}-\eqref{average}
\right\},\\
&(.)_{X_j}:=(.)_{H^j},\quad j=0,1,2.
\end{align}
These notations are utilized to reformulate the perturbation system \eqref{main=p1} in an abstract form. This abstract form is not only instrumental for studying the corresponding eigenvalue problem, but also highly convenient for the proof of \autoref{criticalcase}.

To reformulate the perturbation system \eqref{main=p1} in an abstract form,
we define the linear operator $\mathcal{A}_{\mu} :X_2(\Omega)\to X_0(\Omega)$ as follows 
\begin{align}
\mathcal{A}_{\mu}\mathbf{v}=
\P\left( -\Delta \mathbf{v}+
 \nabla q\right).
\end{align}
We also define the nonlinear operator $\mathcal{F}:X_2(\Omega)\times X_2(\Omega)
\to X_0(\Omega)$ as follows 
\begin{align}\label{nonlinear}
\mathcal{F}(\mathbf{u},\mathbf{v})=-
\P \left((\mathbf{u} \cdot \nabla )\mathbf{v}\right) .
\end{align}
We can rewrite the perturbation equations \eqref{main=p1} as the following abstract 
\begin{align}\label{abstract-1}
\begin{cases}
\frac{d\mathbf{v}}{dt}=-\mathcal{A}_{\mu}\mathbf{v}+\mathcal{F}(\mathbf{v},\mathbf{v}),\\
\mathbf{v}(0)=\mathbf{v}_0.
\end{cases}
\end{align}

 \subsection{Eigenvalue problems}
In the context of the linear problem \eqref{main=p-l}, we shall delve into the associated eigenvalue problem, which is formulated as follows:
\begin{align}\label{polar-l-eigen}
\begin{cases}
\mu \Delta \mathbf{v} - \nabla q = \lambda \mathbf{v}, \\
\nabla \cdot \mathbf{v} = 0,
\end{cases}
\end{align}
subject to the boundary conditions \eqref{cond-3}-\eqref{average}. This eigenvalue problem is also equivalent to the following representation:
\begin{align}\label{polar-l-eigen-a}
(-\mathcal{A}_{\mu}-\beta I)\mathbf{v}=(\lambda -\beta)\mathbf{v},\quad  \mathbf{v}\in X_0(\Omega).
\end{align}
Based on the preceding abstract eigenvalue equation, we present the following lemma:
\begin{lemma}\label{eigennumber}
The operator $\mathcal{A}_{\mu}$ is self-adjoint. Moreover, there exists a positive constant $\beta>0$ such that $-\mathcal{A}_{\mu}-\beta I$ is an isomorphism from $X_2(\Omega)$ to $X_0(\Omega)$. Additionally, the inverse operator $\left(-\mathcal{A}_{\mu}-\beta I\right)^{-1}:X_0(\Omega)\to X_0(\Omega)$ is a compact operator.
\end{lemma}
\begin{proof}
By performing some computations, we obtain
\begin{align}
\begin{aligned}
(\mathcal{A}_{\mu}\mathbf{u},\mathbf{v}) &= \mu \int_{\Omega}\nabla\mathbf{u}:\nabla\mathbf{v}\,dx\,dy + \left(a^{-1}\mu-\alpha\right)\int_{\partial B(0,a)}u_{\mathbf{\tau}} v_{\mathbf{\tau}}\,ds + \frac{\mu}{b}\int_{\partial B(0,b)}u_{\mathbf{\tau}} v_{\mathbf{\tau}} \,ds \\
&= (\mathbf{u},\mathcal{A}_{\mu}\mathbf{v}), \quad \forall~\mathbf{u},\mathbf{v} \in X_2(\Omega).
\end{aligned}
\end{align}
This implies that $\mathcal{A}_{\mu}$ is self-adjoint.
Note that
\begin{align}
\begin{aligned}
\left((\mathcal{A}_{\mu}+\beta I)\mathbf{u},\mathbf{u}\right) &= \mu \int_{\Omega}\nabla\mathbf{u}:\nabla\mathbf{u}\,dx\,dy + \beta \int_{\Omega}\abs{\mathbf{u}}^2\,dx\,dy \\&\quad+ \left(a^{-1}\mu-\alpha\right)\int_{\partial B(0,a)}|u_{\mathbf{\tau}}|^2 \,ds +\frac{\mu}{b}\int_{\partial B(0,b)}|u_{\mathbf{\tau}}|^2 \,ds \\
&\geq \mu \int_{\Omega}\abs{\nabla \mathbf{u}}^2\,dx\,dy + \beta \int_{\Omega}\abs{\mathbf{u}}^2\,dx\,dy \\&\quad- \frac{\mu}{2} \int_{\Omega}\abs{\nabla \mathbf{v}}^2\,dx\,dy - C \int_{\Omega}\abs{\mathbf{u}}^2\,dx\,dy \\
&\geq \frac{\mu}{2} \int_{\Omega}\abs{\nabla \mathbf{u}}^2\,dx\,dy + (\beta-C) \int_{\Omega}\abs{\mathbf{u}}^2\,dx\,dy \geq D\norm{\mathbf{u}}_{H^1}^2
\end{aligned}
\end{align}
for $\beta>2C$ and $D\leq \min\left\{ \frac{\mu}{2},C\right\}$. Therefore, by the Lax-Milgram theorem, the equation
\[
(\mathcal{A}_{\mu}+\beta I)\mathbf{v}=f \in X_0(\Omega)
\]
has a unique weak solution $\mathbf{v} \in X_1(\Omega)$. Based on the Stokes' estimates as in \autoref{Stokes0725}, we see that $\mathbf{v} \in X_2(\Omega)$. The compactness of $\left(\mathcal{A}_{\mu}+\beta I\right)^{-1}$ follows from the Rellich-Kondrachov theorem.
\end{proof}
 \subsection{First eigenvalue}
In fact, the eigenvalue problem \eqref{polar-l-eigen} admits a variational structure. This enables us to reformulate the first eigenvalue of \eqref{polar-l-eigen} from a variational perspective. To this end, we introduce three functionals $E_1(\mathbf{u})$, $E_2(\mathbf{u})$, and $E_3(\mathbf{u})$ as follows:
\begin{align}
\begin{aligned}
&E_1(\mathbf{u}):=\int_{\Omega}\abs{\nabla\mathbf{u}}^2\,dx\,dy+\frac{1}{b}\int_{\partial B(0,b)}\left(u_{\mathbf{\tau}}\right)^2 \,ds,\quad
E_2(\mathbf{u}):=\int_{\partial B(0,a)}\left(u_{\mathbf{\tau}}\right)^2 \,ds,\\
&E_3(\mathbf{u}):=\int_{\Omega}\abs{\mathbf{u}}^2\,dx\,dy,\quad E( \mu, \mathbf{u}):=\frac{ \mu E_1(\mathbf{u})-\left(\alpha-\mu a^{-1}\right)E_2(\mathbf{u})}{
E_3(\mathbf{u})}.
\end{aligned}
\end{align}
Then we consider the following variational problem
\begin{align}\label{variation-one}
-\lambda_1( \mu)=\inf_{\mathbf{u}\in X_1(\Omega)}E( \mu, \mathbf{u}).
\end{align}

\begin{lemma}\label{first-eigen}
There exists $\mathbf{u}_1\in C^{\infty}(\Omega)\cap X_2(\Omega)$ such that 
\begin{align}\label{variation-two}
-\lambda_1( \mu)=E( \mu, \mathbf{u}_1).
\end{align}
Furthermore, $(\mathbf{u},\lambda)=(\mathbf{u}_1,\lambda_1)$ solves the eigenvalue problem \eqref{polar-l-eigen}.
\end{lemma}
\begin{proof}
The existence of solution to the eigenvalue problem \eqref{polar-l-eigen} can be obtained from standard variational method.  By direct calculation, we obtain the following variational derivatives:
\begin{align}\label{func-energy-4}
\begin{aligned}
&DE_1(\mathbf{u}) \mathbf{w} = 2\int_{\Omega} \nabla \mathbf{u} : \nabla \mathbf{w} \,dx\,dy + \frac{2}{b}\int_{\partial B(0,b)} u_{\mathbf{\tau}} w_{\mathbf{\tau}} \,ds, \\
&DE_2(\mathbf{u}) \mathbf{w} = 2\int_{\partial B(0,a)} u_{\mathbf{\tau}} w_{\mathbf{\tau}} \,ds, \quad
DE_3(\mathbf{u}) \mathbf{w} = 2\int_{\Omega} \mathbf{u} \cdot \mathbf{w} \,dx\,dy,
\end{aligned}
\end{align}
for all $\mathbf{w} \in C_0^{\infty}(\Omega)$. At the minimum $\lambda_1(\mu)$ of the functional $E(\mu, \mathbf{u})$, we have
\[
\mu DE_1(\mathbf{u}_1) \mathbf{w} - \left(\alpha - \mu a^{-1}\right) DE_2(\mathbf{u}_1) \mathbf{w} = \lambda_1 DE_3(\mathbf{u}_1) \mathbf{w}, \quad \forall ~\mathbf{w} \in C_0^{\infty}(\Omega).
\]
From this, we see that $\mathbf{u}_1 \in X_1(\Omega)$ is a weak solution of the following equations:
\begin{align}\label{main-eq-4-linear-00}
\begin{cases}
\mu \Delta \mathbf{u}_1 - \nabla p_1 = \lambda_1 \mathbf{u}_1, \\
\nabla \cdot \mathbf{u}_1 = 0,
\end{cases}
\end{align}
subject to the boundary conditions \eqref{cond-2}. Based on the Stokes' estimates as in \autoref{Stokes0725} and the smoothness of $\partial \Omega$, we see that $\mathbf{u}_1 \in H^k(\Omega) \cap X_1(\Omega)$ for any positive integer $k$.
\end{proof}

 \subsection{Critical viscosity}
 
We define critical viscosity $\mu_c$ as the value of $\mu$
at which the first eigenvalue $\lambda_1(\mu)$ vanishes.
To get $\mu_c$, we consider the following zero eigenvalue problem
 \begin{align}\label{polar-l-eigen-a}
 \mathcal{A}_{\mu}\mathbf{v}=0,\quad  \mathbf{v}\in X_2(\Omega).
\end{align}
By direct calculation, we see that if $\mathbf{v}$ solves \eqref{polar-l-eigen-a}, one then gets
\[ \int_{\Omega}\abs{\nabla\mathbf{v}}^2\,dx\,dy+\frac{1}{b}\int_{\partial B\left(0,b\right)}|v_{\tau}|^2ds
=\left(\frac{\alpha}{\mu}-a^{-1}\right)\int_{\partial B(0,a)}\left(v_{\mathbf{\tau}}\right)^2 \,ds.
   \]
The above equation allows us to consider the following functional extremum
 \begin{align}
 \gamma=\inf_{\mathbf{u}\in X_1(\Omega)}
 \frac{E_1(\mathbf{u})}{E_2(\mathbf{u})}
 \end{align}
 by which we introduce a critical viscosity as follows
 \begin{align}\label{c-viscosity}
 \mu_c=\frac{a\alpha}{\gamma+a}.
 \end{align}
The value of $\mu_c$ given by \eqref{c-viscosity} is precisely the value of $\mu$ at which the first eigenvalue $\lambda_1(\mu)$ vanishes. This is described in the following lemma:

   \begin{lemma}\label{posi-eigen}
At $\mu=\mu_c$, the leading eigenvalue $\lambda_1=0$ 
and $\frac{d\lambda_1}{d\mu}\Big|_{\mu=\mu_c}<0$.
\end{lemma}
\begin{proof}
When the viscosity parameter is set to its critical value, namely, $\mu=\mu_c$, we can identify a specific vector field $\mathbf{u}_c\in X_2$ that successfully solves the eigenvalue problem \eqref{polar-l-eigen} with the eigenvalue $\lambda$ precisely equal to zero. This finding directly implies the existence of an eigenvalue, denoted as $\lambda(\mu)$, which satisfies the condition $\lambda(\mu_c)=0$. Our subsequent objective is to rigorously demonstrate that this particular eigenvalue $\lambda(\mu)$ must indeed correspond to the first eigenvalue, $\lambda_1(\mu)$. To proceed with a proof by contradiction, let us assume the contrary, that is, $\lambda_1(\mu_c)>\lambda(\mu_c)=0$. Given the continuous dependence of $\lambda_1(\mu)$ on the parameter $\mu$, we can infer that there must exist a value $\mu=\mu_0>\mu_c$ for which $\lambda_1(\mu_0)=0$. However, this conclusion is in direct conflict with the established definition of the critical value $\mu_c$. Finally, we aim to prove that the derivative of the first eigenvalue concerning the viscosity parameter, evaluated at the critical value, is negative.

Let us denote the fisrt eigenvector $\mathbf{u}_1$ at critical value $\mu=\mu_c$
as $\mathbf{u}_c$. By multiplying the equations of  $\mathbf{u}_1$ by $\mathbf{u}_c$ in $L^2$, and after performing integration by parts while utilizing the boundary conditions \eqref{cond-3}, we obtain the following result:
\begin{align}\label{func-energy-5}
\begin{aligned}
 &\mu \int_{\Omega} \nabla \mathbf{u}_1:
 \nabla \mathbf{u}_c
 \,dx\,dy
  +\frac{\mu}{b} \int_{\partial B(0,b)}u_{1,\mathbf{\tau}}
u_{c,\mathbf{\tau}}
\,ds
\\& -\left(\alpha-a^{-1}\mu\right)
 \int_{\partial B(0,a)}u_{1,\mathbf{\tau}}
u_{c,\mathbf{\tau}}
\,ds=-\lambda_1(\mu)
 \int_{\Omega}  \mathbf{u}_1\cdot
 \mathbf{u}_c
 \,dx\,dy.
 \end{aligned}
\end{align}
At the critical value $\mu=\mu_c$, we also have 
\begin{align}\label{func-energy-6}
\begin{aligned}
 &\mu_c\int_{\Omega} \nabla \mathbf{u}_c:
 \nabla \mathbf{u}_c
 \,dx\,dy
 +\frac{\mu_c}{b} \int_{\partial B(0,b)}u_{c,\mathbf{\tau}}
u_{c,\mathbf{\tau}}
\,ds
 \\&-\left(\alpha-a^{-1}\mu_c\right)
 \int_{\partial B(0,a)}u_{c,\mathbf{\tau}}
u_{c,\mathbf{\tau}}
\,ds=0.
 \end{aligned}
\end{align}
By setting  $\mu=\mu_c+\delta$, and $\mathbf{u}_1=\mathbf{u}_c+\mathbf{u}_{\delta}$ in \eqref{func-energy-5}, we obtain
\begin{align}\label{func-energy-7}
\begin{aligned}
 &\delta  \int_{\Omega} \nabla \mathbf{u}_c:
 \nabla \mathbf{u}_c
 \,dx\,dy
 +\frac{\delta }{b} \int_{\partial B(0,b)}u_{c,\mathbf{\tau}}
u_{c,\mathbf{\tau}}\,ds
 +
 a^{-1}\delta
 \int_{\partial B\left(0,a\right)}u_{c,\mathbf{\tau}}
u_{c,\mathbf{\tau}}
\,ds\\
& +\mu_c\int_{\Omega} \nabla \mathbf{u}_{\delta}:
 \nabla \mathbf{u}_c
 \,dx\,dy
 +\frac{\mu_c}{b} \int_{\partial B(0,b)}u_{\delta,\mathbf{\tau}}
u_{c,\mathbf{\tau}}\,ds
 \\&-\alpha 
  \int_{\partial B\left(0,a\right)}u_{\delta,\mathbf{\tau}}
u_{c,\mathbf{\tau}}\,ds
+a^{-1}\mu_c
 \int_{\partial B\left(0,a\right)}u_{\delta,\mathbf{\tau}}
u_{c,\mathbf{\tau}}ds
\\&+\delta 
\int_{\Omega} \nabla \mathbf{u}_{\delta}:
 \nabla \mathbf{u}_c
 \,dx\,dy
 +b^{-1}\delta
 \int_{\partial B(0,b)}u_{\delta,\mathbf{\tau}}
u_{c,\mathbf{\tau}}\,ds
\\& +
 a^{-1}\delta
 \int_{\partial B(0,a)}u_{\delta,\mathbf{\tau}}
u_{c,\mathbf{\tau}}\,ds
=-\lambda_1(\mu)
 \int_{\Omega}  \mathbf{u}_1\cdot
 \mathbf{u}_c
 \,dx\,dy,
 \end{aligned}
\end{align}
where \eqref{func-energy-6} is used.

From this, we deduce that:
\[
\frac{d\lambda_1}{d\mu}\Big|_{\mu=\mu_c}=-\frac{  \int_{\Omega} \nabla \mathbf{u}_c:
 \nabla \mathbf{u}_c
 \,dx\,dy
 + b^{-1}
 \int_{\partial B(0,b)}u_{c,\mathbf{\tau}}
u_{c,\mathbf{\tau}}
\,ds+
 a^{-1}
 \int_{\partial B(0,a)}u_{c,\mathbf{\tau}}
u_{c,\mathbf{\tau}}
\,ds}{ \int_{\Omega}  \mathbf{u}_1\cdot
 \mathbf{u}_c
 \,dx\,dy}<0.
\]
In this derivation, we have utilized the following relations:
\begin{align*}
\begin{aligned}
&\mu_c\int_{\Omega} \nabla \mathbf{u}_{\delta}:
 \nabla \mathbf{u}_c
 \,dx\,dy
 +b^{-1}\mu_c
 \int_{\partial B(0,b)}u_{\delta,\mathbf{\tau}}
u_{c,\mathbf{\tau}}\,ds
 \\&-\alpha 
  \int_{\partial B(0,a)}u_{\delta,\mathbf{\tau}}
u_{c,\mathbf{\tau}}\,ds
+a^{-1}\mu_c
 \int_{\partial B(0,a)}u_{\delta,\mathbf{\tau}}
u_{c,\mathbf{\tau}}\,ds=0.
 \end{aligned}
\end{align*}
and
\[
\lim_{\delta\to 0}\left(
\int_{\Omega} \nabla \mathbf{u}_{\delta}:
 \nabla \mathbf{u}_c
 \,dx\,dy
 +b^{-1}
 \int_{\partial B(0,b)}u_{\delta,\mathbf{\tau}}
u_{c,\mathbf{\tau}}\,ds
 +a^{-1}
 \int_{\partial B(0,a)}u_{\delta,\mathbf{\tau}}
u_{c,\mathbf{\tau}}\,ds\right)=0.
\]
\end{proof}
It is crucial to emphasize that \autoref{posi-eigen} confirms the validity of the principle of exchange of stability (PES) condition, a fundamental aspect in the analysis of bifurcations and stability transitions. Specifically, the PES condition is given by:
\begin{align}
\lambda_l(\mu)
\begin{cases}
<0,\quad \mu>\mu_c,\\
=0,\quad \mu=\mu_c,\\
>0,\quad \mu<\mu_c,
\end{cases}
\end{align}
where $\lambda_l(\mu)$ denotes the leading eigenvalue of \eqref{polar-l-eigen}. This condition indicates that as the viscosity $\mu$ crosses the critical threshold $\mu_c$, the stability of the steady-state solution $(\mathbf{u},p)=(\mathbf{0},p_0)$ changes. When $\mu$ is greater than $\mu_c$, it is linearly stable, as evidenced by the negative leading eigenvalue. Conversely, when $\mu$ is less than $\mu_c$, it becomes linearly unstable, with the leading eigenvalue becoming positive. At the critical value $\mu=\mu_c$, the leading eigenvalue is zero, marking the transition point between stability and instability. Moreover, it is noted that all other eigenvalues $\lambda$ remain negative. This PES condition thus provides a clear and precise criterion for identifying the onset of instability and the associated bifurcation behavior
at $(\mathbf{0},p_0, \mu_c)$. Based on this PES condition, we can first derive the following lemma:
   \begin{lemma}\label{corrollary}
When $\mu>\mu_c$, then the steady state $(\mathbf{u},p)=(\mathbf{0},p_0)$ is linearly stable while
the steady state $(\mathbf{u},p)=(\mathbf{0},p_0)$ becomes linearly unstable as $\mu<\mu_c$. 
\end{lemma}

 \subsection{Exact expression of $\mu_c$}
\autoref{posi-eigen} indicates that \(\mu_c\) is the value of \(\mu\) at which the first eigenvalue \(\lambda_1\) becomes critical. In other words, \(\mu_c\) is the largest value of \(\mu\) such that the abstract problem \eqref{polar-l-eigen-a} or the original problem \eqref{polar-l-eigen}, subject to the conditions \eqref{cond-3}--\eqref{average}, admits a nontrivial solution. It is important to note that the standard normal mode analysis cannot be directly applied to solve \eqref{polar-l-eigen} in the Cartesian coordinate system
    to derive the exact expression of $\mu_c$. 
   To make the standard normal mode analysis applicable, the geometric features of the domain
    $\Omega$ allow us to rewrite equation \eqref{polar-l-eigen} in the polar coordinate system.    
    In the polar coordinate system, the system \eqref{polar-l-eigen} with $\lambda=0$ reads
   \begin{align}\label{model0329}
  \begin{cases}
   0 =\mu(\Delta_{r} v_{r}-\frac{v_{r}}{r^2}-\frac{2}{r^2}\frac{\partial v_{\theta}}{\partial \theta})-
    \frac{\partial p}{\partial r}, 
    \\
0
    =\mu(\Delta_{r} v_{\theta}-\frac{v_{\theta}}{r^2}+\frac{2}{r^2}\frac{\partial v_{r}}{\partial \theta})-
    \frac{1}{r}\frac{\partial p}{\partial \theta},
    \\
    \frac{\partial (rv_{r})}{\partial r}+\frac{\partial v_{\theta}}{\partial \theta}=0,
  \end{cases}
\end{align}
where $\Delta_{n}$ is a derivative operator given by$
  \Delta_{r}=\frac{\partial^2}{\partial r^2}+\frac{1}{r}\frac{\partial}{\partial r}+\frac{1}{r^2}
  \frac{\partial^2}{\partial \theta^2}$.

It is crucial to adapt the boundary conditions and the domain accordingly. Specifically, the boundary conditions \eqref{cond-3}  and the average condition \eqref{average}, which are initially defined in Cartesian coordinates, need to be re-expressed in the new polar coordinate system. Through basic calculations, \eqref{cond-3}-\eqref{average} can be reformulated as follows:
\begin{align}\label{bianjie0329}
  \begin{aligned}
      &v_{r}|_{r=a,b}=
 \left(\partial_rv_{\theta}+\frac{v_{\theta}}{r}\right)\bigg|_{r=b}=0,
  ~\frac{\partial v_{\theta}}{\partial r}\bigg|_{r=a}
  =\left(\frac{1}{r}-\frac{\alpha}{\mu}\right) v_{\theta} \bigg|_{r=a},\quad
  \int_{0}^{2\pi}v_{\theta}\,d\theta=0.
    \end{aligned}
\end{align}

To eliminate the pressure terms and the divergence-free condition in \eqref{model0329}, we introduce a streamfunction \(\psi\) such that
\begin{align}
v_{r} = -\frac{1}{r}\frac{\partial \psi}{\partial \theta}, \quad
v_{\theta} = \frac{\partial \psi}{\partial r}.
\end{align}
This transformation allows us to reformulate the problem
 \eqref{model0329} into the following equivalent form:
\begin{align}\label{main=p}
\begin{cases}
0 = \mu \Delta_r^2 \psi, \\
\partial_\theta \psi = \partial_r^2 \psi - \left(\frac{1}{r} - \frac{\alpha}{\mu}\right) \partial_r \psi = 0, \quad r = a, \\
\partial_\theta \psi = \partial_r^2 \psi + \frac{\partial_r \psi}{r} = 0, \quad r = b, \\
\int_{0}^{2\pi} \psi \, d\theta = 0.
\end{cases}
\end{align}
Here, we have replaced the condition \(\int_{0}^{2\pi} \frac{\partial \psi}{\partial r} \, d\theta = 0\) with \(\int_{0}^{2\pi} \psi \, d\theta = 0\). Note that \(\int_{0}^{2\pi} \frac{\partial \psi}{\partial r} \, d\theta = 0\) implies that \(\int_{0}^{2\pi} \psi \, d\theta\) is a constant. Since \(\psi\) is defined up to a constant, we can choose \(\psi\) such that \(\int_{0}^{2\pi} \psi \, d\theta = 0\). This choice ensures that $\psi$ can not be a function only depending on $r$.
  
In polar coordinates \((r,\theta)\), the periodic property of \(\psi\) along the \(\theta\)-direction enables us to solve the zero eigenvalue problem \eqref{main-eq-4-linear-00} by employing the standard normal mode analysis. Specifically, we can express \(\psi\) in terms of a new unknown function as follows:
\begin{align}\label{normalmode}
\psi = \Psi(r) e^{in\theta}.
\end{align}

Substitute the preceding expression into \eqref{main=p},
by direct calculation, for each fixed $n\neq 0$, this leads to the following system of ODEs
   \begin{align}\label{polar-l-eigen-2}
  \begin{cases}
\Delta_{n}^2\Psi=0, \\
 \Psi= \Psi''-
\left(\frac{1}{r}-\frac{\alpha}{\mu}\right) 
\Psi'=0,\quad r=a,\\
 \Psi=\Psi''+\frac{\Psi'}{r}=0,\quad r=b,
  \end{cases}
\end{align}
where $
\Delta_{n}=\frac{d^2}{dr^2}+\frac{1}{r}\frac{d}{d r}-\frac{n^2}{r^2}$.
\eqref{polar-l-eigen-2} also has a variational structure,  which allows us to define
\begin{align}
\gamma_n=\inf_{
\substack{ \Phi \in H^2(a,b),~
\Phi(a)=0\\ \Phi(b)=0}}
\frac{
\int_a^br(\Delta_{n}\Phi)^2\,dr
}{\left(\Phi'(a)\right)^2}.
\end{align}
For the preceding variational problem, we obtain the folloiwng results:
\begin{lemma}\label{first-eigen}
There exists $\Psi\in C^{\infty}(a,b)\cap H^2(a,b)$ satisfying 
$\Psi(a)=\Psi(b)=0$
 such that 
\begin{align}\label{variation-two-2}
\gamma_n=\frac{\int_a^br(\Delta_{n}\Psi)^2\,dr
}{\left(\Psi'(a)\right)^2}=\inf_{
\substack{ \Phi \in H^2(a,b),~
\Phi(a)=0\\ \Phi(b)=0}}
\frac{
\int_a^br(\Delta_{n}\Phi)^2\,dr
}{\left(\Phi'(a)\right)^2}.
\end{align}
Furthermore, $\Psi$ solves the eigenvalue problem \eqref{polar-l-eigen-2}
 with 
 \begin{align}
 \mu= \mu_c^n=:\frac{a\alpha}{\gamma_n+2}.
 \end{align}
\end{lemma}
\begin{proof}
When \(\Psi = \Phi\), by performing integration by parts, we obtain
\begin{align}\label{WW}
\int_a^b r \Delta_{n} \Psi \Delta_{n} W \, dr = \left( \frac{a \alpha}{\mu_c^n} - 2 \right) \Psi'(a) W'(a), \quad \forall W \in H^{2}(a,b), \quad W(a) = W(b) = 0.
\end{align}
By selecting \(W\) to be compactly supported within \((a,b)\) in \eqref{WW}, we derive
\begin{align}\label{WW-0}
\int_a^b r \Delta_{n} \Psi \Delta_{n} W \, dr = 0, \quad \forall W \in C^{\infty}_0(a,b).
\end{align}
This implies that
\begin{align}\label{WW-1}
\mu_c \Delta_{n}^2 \Psi = 0
\end{align}
in a weak sense.
Utilizing standard bootstrapping arguments, we can demonstrate that \(\Psi\) is smooth. From \eqref{WW} and \eqref{WW-1}, we deduce that
   \begin{align}\label{polar-bd}
  \begin{cases}
 \Psi''-
\left(\frac{1}{r}-\frac{\alpha}{\mu}\right) 
\Psi'=0,\quad r=a,\\
\Psi''+\frac{\Psi'}{r}=0,\quad r=b,
  \end{cases}
\end{align}
Hence, $\Psi$ must solve the problem \eqref{polar-l-eigen-2}.
\end{proof}

The following proposition indicates that \(\gamma_{n+1} > \gamma_n\). 
\begin{lemma}\label{mingti0718}
    For the variational problem \eqref{variation-two-2}, $\gamma_{n+1} > \gamma_n$.
\end{lemma}
\begin{proof}
   From the \autoref{first-eigen}, we can conclude that there exist $\Psi_{n}$ and $\Psi_{n+1}\in C^{\infty}\left(a,b\right)\cap H^{2}\left(a,b\right)$ such that 
   \begin{align*}
       \begin{aligned}
       \gamma_n=\frac{\int_a^br(\Delta_{n}\Psi_{n})^2\,dr
}{\left(\Psi_{n}'(a)\right)^2},~
\gamma_{n+1}=\frac{\int_a^br(\Delta_{n+1}\Psi_{n+1})^2\,dr
}{\left(\Psi_{n+1}'(a)\right)^2}.
       \end{aligned}
   \end{align*}
Then, by the definition of $\gamma_{n}$, we have 
\begin{align*}
    \begin{aligned}
        \gamma_{n}\leq\frac{\int_a^br(\Delta_{n}\Psi_{n+1})^2\,dr
}{\left(\Psi_{n+1}'(a)\right)^2},
    \end{aligned}
\end{align*}
which yields that 
\begin{align*}
    \begin{aligned}
        \gamma_{n}-\gamma_{n+1}
        &\leq 
        \frac{\int_a^b r\left[(\Delta_{n}\Psi_{n+1})^2-(\Delta_{n+1}\Psi_{n+1})^2\right]\,dr
}{\left(\Psi_{n+1}'(a)\right)^2}
\\&<-(2n+1)\int_{a}^{b}\frac{2}{r}
\left(\frac{d\Psi_{n+1}}{dr}-\frac{\Psi_{n+1}}{r}\right)^2dr<0.
    \end{aligned}
\end{align*}
\end{proof}
Therefore, we can conclude that \(\gamma\), as defined by \eqref{variation-one}, is equal to \(\gamma_1\). Consequently, \(\mu_c\) is the value of \(\mu\) such that the following system has a nontrivial solution:
\begin{align}\label{polar-l-eigen-3}
\begin{cases}
\Delta_{1}^2 \Psi = 0, \\
\Psi = \Psi'' - \left(\frac{1}{r} - \frac{\alpha}{\mu_c}\right) \Psi' = 0, \quad r = a, \\
\Psi = \Psi'' + \frac{\Psi'}{r} = 0, \quad r = b.
\end{cases}
\end{align}

Note that \(\Delta_{1}^2 \Psi = 0\) is a fourth-order differential equation with non-constant coefficients. This characteristic implies that the standard Euler method cannot be directly applied to solve this fourth-order equation. To utilize the Euler method for solving this problem, an invertible transformation must be employed to convert the equation with non-constant coefficients into one with constant coefficients. Fortunately, employing the following transformation
\[
\Psi(r) := r^2 \Phi(\ln r), \quad s = \ln r,
\]
and performing some calculations, \(\Delta_{1}^2 \Psi = 0\) can be converted into the following system with constant coefficients:
\begin{align}\label{special}
\Phi^{(4)}(s) + 4\Phi^{(3)}(s) + 2\Phi''(s) - 4\Phi'(s) - 3\Phi(s) = 0,
\end{align}
whose characteristic polynomial reads
\[
y^4 + 4y^3 + 2y^2 - 4y - 3 = (y - 1)(y + 1)^2(y + 3) = 0.
\]
Hence, utilizing the standard Euler method, we obtain that \(e^{-s}\), \(se^{-s}\), \(e^{-3s}\), and \(e^s\) are four different special solutions of \eqref{special}. Correspondingly, the linear system \(\Delta_{1}^2 \Psi = 0\) has four different special solutions \(r^3\), \(r\), \(r \ln r\), and \(r^{-1}\). This means that the general solutions of \eqref{polar-l-eigen-3} are
\[
\Psi=a_1r+a_2r\ln r+\frac{a_3}{r}+a_4 r^3.
\]

Making use of the following four conditions on boundary 
\[
  \begin{cases}
 \Psi= \Psi''-
\left(\frac{1}{r}-\frac{\alpha}{\mu_c}\right) 
\Psi'=0,\quad r=a,\\
 \Psi=\Psi''+\frac{\Psi'}{r}=0,\quad r=b,
  \end{cases}
\]
we get that \eqref{polar-l-eigen-3} has nontrivial solutions if and only if
the three parameters $a, b$ and $\mu_c$ satisfy the following equation
\[
\begin{vmatrix}
a^2&a^2\ln a&1&a^4\\
b^2&b^2\ln b&1&b^4\\
b^2&b^2(2+\ln b)&1&9b^4\\
-a^2 \mu_c+\alpha a^3&
a^3 \alpha-a^2\mu_c \ln a+a^3\alpha \ln a
&3\mu_c
-a\alpha&
3a^5 \alpha+3 a^4\mu_c
\end{vmatrix}=0
\]
 from which one can get the explicit expression of viscosity threshold $\mu_c$ as follows
 \begin{align}\label{uuu}
 \mu_c= \frac{a \alpha \left(1+3\sigma^4-4\sigma^2-
 4 \sigma^4 \log (\sigma)\right)}
 {2 \left(\sigma^4-1-4 \sigma^4 \log (\sigma)\right)} ,\quad \sigma=b/a.
 \end{align}

\section{Proofs of main theorems}
 
 \subsection{Proof of \autoref{existence} }\label{section31}

The proof of local existence and uniqueness of strong solutions is standard and thus omitted here. To establish the global existence of strong solutions, it suffices to derive some global energy estimates. To this end, let \((\mathbf{v}, q)\) be a strong solution of the perturbed problem \eqref{main=p}. In the sequel, for simplicity, \(C\) will denote a generic positive constant, which may depend on $a, b, \mu$ and $\alpha$, and maby be different in different inequalities. 

To test \(\eqref{main=p1}_1\) by \(\mathbf{v}\), we integrate by parts over \(\Omega\) and use \autoref{lemma-grad-1-1}. This yields the following:
\begin{align}\label{ddddd}
\frac{1}{2}\frac{d}{dt}\int_{\Omega}\abs{\mathbf{v}}^2\,dx\,dy
+\mu\int_{\Omega}\abs{\nabla\mathbf{v}}^2\,dx\,dy
=-\frac{\mu}{b}\int_{\partial B(0,b)}\left(v_{\mathbf{\tau}}\right)^2 \,ds+
\left(\alpha-\mu a^{-1}\right)
\int_{\partial B(0,a)}\left(v_{\mathbf{\tau}}\right)^2 \,ds
\end{align}
by which and \autoref{epsilon-1}, we have
\begin{align}\label{l22}
\frac{1}{2}\frac{d}{dt}\int_{\Omega}\abs{\mathbf{v}}^2\,dx\,dy
+\frac{\mu}{2}\int_{\Omega}\abs{\nabla\mathbf{v}}^2\,dx\,dy
\leq C\int_{\Omega}\abs{\mathbf{v}}^2\,dx\,dy.
\end{align}
Gronwall inequality implies that for any fixed $T>0$, 
\begin{align}\label{e-1}
\sup_{0\leq t\leq T}\int_{\Omega}\abs{\mathbf{v}}^2\,dx\,dy
+\int_0^T\left(
\int_{\Omega}\abs{\nabla\mathbf{v}(s)}^2\,dx\,dy\right)\,ds
\leq e^{CT}
\int_{\Omega}\abs{\mathbf{v}(0)}^2\,dx\,dy.
\end{align}

By differentiating the corresponding equations and the conditions \eqref{cond-1}-
\eqref{cond-2} with respect to $t$, we observe that $\frac{\partial \mathbf{v}}{\partial t}$ satisfies the following system 
   \begin{align}\label{ut-0718}
\begin{cases}
\mathbf{v}_{tt}+ (\mathbf{v}_t \cdot \nabla )\mathbf{v}
+ (\mathbf{v}\cdot \nabla )\mathbf{v}_t
= \mu \Delta \mathbf{v}_t-\nabla q_t
\\ \nabla \cdot  \mathbf{v}_t=0.
\end{cases}
\end{align}
The boundary conditions for the system \eqref{ut-0718} are as follows:
  \begin{align}\label{cond-1-11}
  &\mathbf{v}_t \cdot   \mathbf{n}|_{x^2+y^2=a^2}=0,\quad
  \mathbf{v}_t \cdot   \mathbf{n}|_{x^2+y^2=b^2}=\nabla \times   \mathbf{v}_t|_{x^2+y^2=b^2}=0,\\ \label{cond-2-1}
   &\left[\left(-q_t\mathbf{I} +\mu\left (\nabla \mathbf{v}_t+\left(\nabla \mathbf{v}_t\right)^{Tr} \right)\right)\cdot\mathbf{n}\right]
   \cdot   \mathbf{\tau}|_{x^2+y^2=a^2}
  =\alpha  \mathbf{v}_t  \cdot   \mathbf{\tau}|_{x^2+y^2=a^2},\quad \alpha \geq 0.
  \end{align}

Taking the $L^2$-inner product of \eqref{ut-0718} with $\mathbf{v}_t $ and using
  \autoref{lemma-grad-1-1}, we infer that
  \begin{align}\label{ut-proof-11}
\begin{aligned}
&\frac{d}{dt}\int_{\Omega}\frac{\abs{\mathbf{v}_t}^2}{2}\,dx\,dy
+\mu\int_{\Omega}\abs{\nabla \mathbf{v}_t}^2\,dx\,dy
\\&
+\frac{\mu}{b}\int_{\partial B(0,b)}\left(\tau \cdot \mathbf{v}_t\right)^2 \,ds
+\int_{\partial B(0,a)}\left(\mu a^{-1}-\alpha\right)\left(\tau \cdot \mathbf{v}_t\right)^2 \,ds
\\  &=- \int_{\Omega}
\mathbf{v}_t\cdot (\mathbf{v}_t\cdot\nabla )\mathbf{v}
 \,dx\,dy.
\end{aligned}
\end{align}
By H\"older's, Ladyzhenskaya's and Young's inequalities, one deduces that
    \begin{align}\label{ut-proof-33}
&\begin{aligned}
 \int_{\Omega}\abs{
\mathbf{v}_t\cdot (\mathbf{v}_t\cdot\nabla )\mathbf{v}}
 \,dx\,dy\leq &\norm{
 \nabla \mathbf{v}}_{L^2}
 \norm{ \mathbf{v}_t}^2_{L^4}\leq
C \norm{
 \nabla \mathbf{v}}_{L^2}
 \left(
 \norm{ \nabla  \mathbf{v}_t}^{\frac{1}{2}}_{L^2}
 \norm{\mathbf{v}_t}^{\frac{1}{2}}_{L^2}
 +\norm{\mathbf{v}_t}_{L^2}
 \right)^2\\
 \leq &\epsilon
 \norm{\nabla \mathbf{v}_t}_{L^2}^2+C
  \norm{
 \nabla \mathbf{v}}^2_{L^2}
  \norm{
\mathbf{v}_t}^2_{L^2},
\end{aligned}
\end{align}

By choosing $\epsilon$ sufficiently small and using \autoref{epsilon-1}, we derive the following conclusion:
  \begin{align}\label{ut-proof-111}
\begin{aligned}
&\frac{d}{dt}\int_{\Omega}\frac{\abs{\mathbf{v}_t}^2}{2}\,dx\,dy
+\frac{\mu}{2}\int_{\Omega}\abs{\nabla \mathbf{v}_t}^2\,dx\,dy\leq
C\left(
  \norm{
 \nabla \mathbf{v}}^2_{L^2}
 +1\right)
  \norm{
\mathbf{v}_t}^2_{L^2},
\end{aligned}
\end{align}
Gronwall inequality with \eqref{e-1} implies that for any fixed $T>0$, 
\begin{align}\label{e-2}
\sup_{0\leq t\leq T}\int_{\Omega}\abs{\mathbf{v}_t}^2\,dx\,dy
+\int_0^T\left(
\int_{\Omega}\abs{\nabla\mathbf{v}_t(s)}^2\,dx\,dy\right)\,ds
\leq C(T)\norm{\mathbf{v}_t(0)}_{L^2}^2.
\end{align}

To bound $\norm{\mathbf{v}_t(0)}^2_{L^2}$, testing $\eqref{main=p}_1$ by $\mathbf{v}_t$, we have
 \begin{align*}
\begin{aligned}
&\int_{\Omega}\abs{\mathbf{v}_t}^2\,dx\,dy
-\mu\int_{\Omega}\Delta \mathbf{v}\cdot  \mathbf{v}_t\,dx\,dy
+\int_{\Omega} \mathbf{v}_t\cdot( \mathbf{v}\cdot \nabla )\mathbf{v}\,dx\,dy
=0.
\end{aligned}
\end{align*}
From this, it follows that
  \begin{align*}
\begin{aligned}
\norm{\mathbf{v}_t}^2_{L^2}
\leq \mu \norm{\mathbf{v}_t}_{L^2}
\norm{\Delta \mathbf{v}}_{L^2}
+ \norm{\mathbf{v}_t}_{L^2}
\norm{( \mathbf{v}\cdot \nabla )\mathbf{v}}_{L^2},
\end{aligned}
\end{align*}
which gives
  \begin{align}\label{initial-0}
\begin{aligned}
\norm{\mathbf{v}_t(0)}_{L^2}^2\leq C\norm{\mathbf{v}_0}_{H^2}^2.
\end{aligned}
\end{align}

Testing \(\eqref{main=p1}_1\) by \(\mathbf{v}_t\), integrating by parts over \(\Omega\), we obtain
 \begin{align}\label{keyes0718}
\begin{aligned}
&\int_{\Omega}\abs{\mathbf{v}_t}^2\,dx\,dy
+\frac{\mu}{2}\frac{d}{dt}\int_{\Omega}\abs{\nabla \mathbf{v}}^2\,dx\,dy
\\&
=-\frac{\mu}{b}\int_{\partial B(0,b)}\left(\tau \cdot \mathbf{v}_t\right)
\left(\tau \cdot \mathbf{v}\right)
 \,ds
-\int_{\partial B(0,a)}\left(\mu a^{-1}-\alpha\right)\left(\tau \cdot \mathbf{v}_t\right)
\left(\tau \cdot \mathbf{v}\right) \,ds
\\  &
\quad-\int_{\Omega} \mathbf{v}_t\cdot( \mathbf{v}\cdot \nabla )\mathbf{v}\,dx\,dy
=I_1+I_2.
\end{aligned}
\end{align}
where
 \begin{align*}
\begin{aligned}
&I_1\leq C\left(
\norm{\mathbf{v}_t}_{L^2}^2+
\norm{\nabla \mathbf{v}_t}_{L^2}^2
+\norm{ \mathbf{v}}_{H^1}^2
\right),\\
&I_2\leq C\left(\norm{\mathbf{v}}_{L^2}^2
\norm{\nabla \mathbf{v}_t}_{L^2}^2
+\norm{\mathbf{v}}_{H^1}^2\right).
\end{aligned}
\end{align*}
We then deduce that
\begin{align}
\sup_{0\leq t\leq T}\int_{\Omega}\abs{\nabla \mathbf{v}}^2\,dx\,dy
+\int_0^T\left(
\int_{\Omega}\abs{\mathbf{v}_t(s)}^2\,dx\,dy\right)\,ds
\leq C(T)\norm{\mathbf{v}_{0}}_{H^2}^2.
\end{align}

Finally, we get by using Stokes estimates as in \autoref{Stokes0725}
\begin{align}\label{e-3}
\sup_{0\leq t\leq T}\left(\int_{\Omega}\abs{\nabla^2 \mathbf{v}}^2\,dx\,dy
+\int_{\Omega}\abs{\nabla q}^2\,dx\,dy
\right)
\leq C(T)\norm{\mathbf{v}_{0}}_{H^2}^2.
\end{align}
Combining \eqref{e-1}, \eqref{e-2} and \eqref{e-3}, it yields 
\begin{align}\label{energyone}
\begin{aligned}
\sup_{0\leq t\leq T}\left(
\norm{\mathbf{v}}_{H^2}^2
+\norm{\nabla q}_{L^2}^2
+\norm{\mathbf{v}_t}_{L^2}^2
\right)+\int_0^T\left(\norm{\nabla \mathbf{v}}_{L^2}^2
+\norm{ \mathbf{v}_s}_{H^1}^2
\right)\,ds
\leq C(T)\norm{\mathbf{v}_{0}}_{H^2}^2.
\end{aligned}
\end{align}
This nonlinear energy estimates show that
the existence of global strong solutions
to the system \eqref{main=p} subject to the mixed boundary conditions \eqref{cond-3}.

One can get the local existence and uniqueness of strong solutions to 
the problem \eqref{main=p} subject to boundary conditions \eqref{cond-3}.
Then the global existence and uniqueness of strong solutions can be obtained by using 
the a priori energy estimates \eqref{energyone}. We then only need to show 
\eqref{nonlinear0718} under the condition $\norm{\mathbf{v}}_{H^2}^2\leq \delta_0$ on the closed interval $[0,T]$.
For $I_1$ and $I_2$ in \eqref{keyes0718}, we have
 \begin{align*}
\begin{aligned}
&I_1=-\frac{\mu}{b}\int_{\partial B(0,b)}\left(\tau \cdot \mathbf{v}_t\right)
\left(\tau \cdot \mathbf{v}\right)
 \,ds
-\int_{\partial B(0,a)}\left(\mu a^{-1}-\alpha\right)\left(\tau \cdot \mathbf{v}_t\right)
\left(\tau \cdot \mathbf{v}\right) \,ds
\\
&\leq
\frac{\mu\epsilon}{2b}
\int_{\partial B(0,b)}\left(\tau \cdot \mathbf{v}_t\right)^2
 \,ds
 +\frac{\abs{\left(\mu a^{-1}-\alpha\right)}\epsilon}{2}
\int_{\partial B(0,a)}
 \left(\tau \cdot \mathbf{v}_t\right)^2
 \,ds\\
 &\quad+\frac{\mu}{2b\epsilon}
\int_{\partial B(0,b)}\left(\tau \cdot \mathbf{v}\right)^2
 \,ds
 +\frac{\abs{\left(\mu a^{-1}-\alpha\right)}}{2\epsilon}
\int_{\partial B(0,a)}
 \left(\tau \cdot \mathbf{v}\right)^2
 \,ds
\\
&\leq \frac{1}{4}\norm{\mathbf{v}_t}_{L^2}^2+\epsilon \norm{\nabla\mathbf{v}_t}_{L^2}^2
+C_\epsilon \norm{\mathbf{v}}_{H^1}^2
,\\
&I_2\leq C\int_{\Omega}
\abs{\mathbf{v}}\abs{\nabla\mathbf{v}}\abs{\mathbf{v}_t}
\,dx\,dy\leq C\norm{\mathbf{v}_t}_{L^2}
\norm{\mathbf{v}}_{L^4}
\norm{\nabla \mathbf{v}}_{L^4}
\leq \frac{1}{4}
\norm{\mathbf{v}_t}_{L^2}^2+C
\norm{\mathbf{v}}_{H^2}^2
\norm{\mathbf{v}}_{L^2}^2
\end{aligned}
\end{align*}
where we have used \autoref{epsilon-1}. We then get that
\begin{align}\label{keyes}
\begin{aligned}
&\frac{1}{2}\int_{\Omega}\abs{\mathbf{v}_t}^2\,dx\,dy
+\frac{\mu}{2}\frac{d}{dt}\int_{\Omega}\abs{\nabla \mathbf{v}}^2\,dx\,dy
\leq 
\epsilon \norm{\nabla\mathbf{v}_t}_{L^2}^2
+C_\epsilon \norm{\mathbf{v}}_{H^1}^2
+C
\norm{\mathbf{v}}_{H^2}^2
\norm{\mathbf{v}}_{L^2}^2.
\end{aligned}
\end{align}
Adding $A\times \eqref{keyes}$, $B\times\eqref{l22}$ and \eqref{ut-proof-111}, we have
\begin{align}\label{keyes-adding}
\begin{aligned}
&\frac{d}{dt}\left(\frac{A\mu}{2}
\norm{\nabla\mathbf{v}}_{L^2}^2
+
\frac{B}{2}
\norm{\mathbf{v}}_{L^2}^2
+\frac{1}{2}
\norm{\mathbf{v}_t}_{L^2}^2
\right)
+\frac{A}{2}\norm{\mathbf{v}_t}_{L^2}^2
+\frac{B\mu}{2}
\norm{\nabla\mathbf{v}}_{L^2}^2
+\frac{\mu}{2}
\norm{\nabla\mathbf{v}_t}_{L^2}^2
\\&\leq 
A\epsilon  \norm{\nabla\mathbf{v}_t}_{L^2}^2
+AC_\epsilon \norm{\mathbf{v}}_{H^1}^2
+AC
\norm{\mathbf{v}}_{H^2}^2
\norm{\mathbf{v}}_{L^2}^2
\\&\quad+BC  \norm{
\mathbf{v}}^2_{L^2}+C
  \norm{
 \nabla \mathbf{v}}^2_{L^2}
   \norm{
\mathbf{v}_t}^2_{L^2}
 +C
  \norm{
\mathbf{v}_t}^2_{L^2}.
\end{aligned}
\end{align}
Taking $A=4C, B=\frac{16CC_{\epsilon}}{\mu}$ and $\epsilon=\frac{\mu}{16C}$ in the preceidng inequality, we get
\begin{align}\label{keyes-adding-2}
\begin{aligned}
&\frac{d}{dt}\left(2C\mu
\norm{\nabla\mathbf{v}}_{L^2}^2
+
\frac{8CC_{\epsilon}}{\mu}
\norm{\mathbf{v}}_{L^2}^2
+\frac{1}{2}
\norm{\mathbf{v}_t}_{L^2}^2
\right)
+C\norm{\mathbf{v}_t}_{L^2}^2
+4CC_{\epsilon}
\norm{\nabla\mathbf{v}}_{L^2}^2
+\frac{\mu}{4}
\norm{\nabla\mathbf{v}_t}_{L^2}^2
\\&\leq 4CC_\epsilon \norm{\nabla\mathbf{v}}_{L^2}^2
+BC  \norm{
\mathbf{v}}^2_{L^2}
+C
\norm{\mathbf{v}}_{H^2}^2
\left(4C
\norm{\mathbf{v}}_{L^2}^2+
   \norm{
\mathbf{v}_t}^2_{L^2}\right)
\\&\leq 4CC_\epsilon \norm{\nabla\mathbf{v}}_{L^2}^2
+BC  \norm{
\mathbf{v}}^2_{L^2}
+C\delta_0
\left(4C
\norm{\mathbf{v}}_{L^2}^2+
   \norm{
\mathbf{v}_t}^2_{L^2}\right).
\end{aligned}
\end{align}
Hence, for suitably small $\delta_0\in (0,1]$, we have
\begin{align}\label{keyes-adding-3}
\begin{aligned}
&\frac{d}{dt}\left(2C\mu
\norm{\nabla\mathbf{v}}_{L^2}^2
+
\frac{8CC_{\epsilon}}{\mu}
\norm{\mathbf{v}}_{L^2}^2
+\frac{1}{2}
\norm{\mathbf{v}_t}_{L^2}^2
\right)
\\&+C_0\norm{\mathbf{v}_t}_{L^2}^2
+4CC_{\epsilon}
\norm{\nabla\mathbf{v}}_{L^2}^2
+\frac{\mu}{4}
\norm{\nabla\mathbf{v}_t}_{L^2}^2\leq C'\norm{\mathbf{v}}_{L^2}^2.
\end{aligned}
\end{align}
It then yields from \eqref{keyes-adding-3} and \eqref{initial-0} that
there exists  $C_0>0$ independent of $T$ such that
 \begin{align}\label{nonlinear-33}
\begin{aligned}
&\norm{\mathbf{v}}_{H^1}^2
+\norm{\mathbf{v}_t}_{L^2}^2
+\int_0^t\left(\norm{\nabla \mathbf{v}}_{L^2}^2
+
\norm{\mathbf{v}_t}_{H^1}^2
\right)\,ds
\\&\leq C_0\left(\norm{\mathbf{v}_{0}}_{H^2}^2
+
\int_0^t\norm{\mathbf{v}}_{L^2}^2\,ds
\right).
\end{aligned}
\end{align}
Under the condition $\norm{\mathbf{v}}_{H^2}^2\leq \delta_0$, using 
Stokes estimates we have
\[
\norm{\mathbf{v}}_{H^2}^2
+\norm{\nabla q}_{L^2}^2\leq C\left(
\norm{\mathbf{v}}_{H^1}^2+\norm{\mathbf{v}_t}_{L^2}^2\right)
\]
by which and \eqref{nonlinear-33}, we finally get the estimate \eqref{nonlinear0718}.

 \subsection{Proof of \autoref{non-stability}}\label{section32}

\begin{lemma}\label{lemma-grad-1-2}
For the strong solutions \(\mathbf{v}\) of the problem \eqref{main=p1} subject to the boundary conditions \eqref{cond-3} and the average condition \eqref{average} with initial data \(\mathbf{v}_0 \in H^{2}\left(\Omega\right)\), we have
\begin{align}\label{houmian0721}
\norm{\mathbf{v}}_{L^2}^2 \leq \norm{\mathbf{v}_0}_{L^2}^2 e^{2\lambda_1(\mu)t},
\end{align}
where $\lambda_1(\mu)$
is the first eigenvalue of the problem \eqref{polar-l-eigen}, defined by \eqref{variation-one}.
\end{lemma}
\begin{proof}
Testing \(\eqref{main=p1}_1\) by \(\mathbf{v}\), we integrate by parts over \(\Omega\) and use \autoref{lemma-grad-1-1} to get:
\begin{align}
\begin{aligned}
\frac{1}{2}\frac{d\norm{\mathbf{v}}_{L^2}^2}{dt}=&
 -\mu \left(
 \int_{\Omega}\abs{\nabla\mathbf{v}}^2\,dx\,dy+\frac{1}{b}\int_{\partial B(0,b)}\left(v_{\mathbf{\tau}}\right)^2 \,ds\right)
 \\& -\left(\frac{\alpha}{\mu}- a^{-1}\right)\int_{\partial B(0,a)}(v_{\tau})^2\,ds
 \leq \lambda_1(\mu)\norm{\mathbf{v}}_{L^2}^2
 \end{aligned}
\end{align}
from which we obtain 
\begin{align}
\norm{\mathbf{v}}_{L^2}^2\leq \norm{\mathbf{v}_0}_{L^2}^2e^{2\lambda_1(\mu)t}.
\end{align}
\end{proof}

\begin{lemma}
Suppose that $p\geq 2$. For the solution $\mathbf{v}$ of the problem \eqref{main=p1}
subject to the boundary conditions \eqref{cond-3} and
with initial data $\mathbf{v}_0 \in W^{1,p}\left(\Omega\right) $,
 then there exists a positive constant $C$ only depending on
 $a,b, p$ and $\alpha$, such that for any fixed $T>0$ and arbitrary $\epsilon>0$, we have
\begin{align}\label{lama}
\norm{v_{\tau}}_{L^{\infty}((0,T)\times \partial B(0,a))}\leq \epsilon C
\norm{(
-\partial_y,\partial_x)\cdot\mathbf{v}}_{L^{p}((0,T)\times \Omega)}
+ C_{\epsilon}\norm{\mathbf{v}}_{L^{\infty}((0,T);L^2(\Omega))}.
\end{align}
\end{lemma}
\begin{proof}
We only need to estimate $v_{\tau}=\tau\cdot \mathbf{v}$.
Using Gagliardo-Nirenberg interpolation and Young's inequality, one obtains
\begin{align}
\begin{aligned}
&\norm{v_{\tau}}_{L^{\infty}((0,T)\times \partial B(0,a))}
\leq
\norm{\mathbf{v}}_{L^{\infty}((0,T)\times \Omega)}
\\&\leq
C\norm{\nabla\mathbf{v}}^{\frac{p}{2(p-1)}}
_{L^{\infty}((0,T);L^p(\Omega))}
\norm{\mathbf{v}}^{\frac{p-2}{2(p-1)}}_{L^{\infty}((0,T);L^2(\Omega))}
+C\norm{\mathbf{v}}_{L^{\infty}((0,T);L^2(\Omega))}\\ &\leq
\epsilon
\norm{\nabla\mathbf{v}}_{L^{\infty}((0,T);L^p(\Omega))}
+C_{\epsilon}\norm{\mathbf{v}}_{L^{\infty}((0,T);L^2(\Omega))}\\&\leq
\epsilon
\norm{(
-\partial_y,\partial_x)\cdot\mathbf{v}}_{L^{\infty}((0,T);L^p(\Omega))}
+ C_{\epsilon}\norm{\mathbf{v}}_{L^{\infty}((0,T);L^2(\Omega))}.
\end{aligned}
\end{align}

\end{proof}

\subsubsection{Decay of $\norm{\mathbf{v}}_{W^{1,p}(\Omega)}$}

\begin{lemma}\label{lemma-grad-1-2}
For the strong solutions \(\mathbf{v}\) of the problem \eqref{main=p1} subject to the boundary conditions \eqref{cond-3} and the average condition \eqref{average} with initial data \(\mathbf{v}_0 \in H^{2}\left(\Omega\right)\), we have
\begin{align}\label{3333}
\norm{\mathbf{v}}_{W^{1,p}}\leq 2
e^{-\frac{ 2C_a^b(p-1)\mu}{p^2}t}
\norm{\omega_0}_{L^p}+4\abs{\Omega}^{\frac{1}{p}}D_{\epsilon}
e^{\lambda_1(\mu)t}\left\|\mathbf{v}_{0}\right\|_{L^2},
\end{align}
where $\omega_{0}=( -\partial_y,\partial_x)\cdot\mathbf{v}_{0}$.
\end{lemma}
\begin{proof}
In view of \autoref{lemma-grad} and \autoref{lemma-grad-1},
to estimate $\left\|{\mathbf{v}}\right\|_{W^{1,p}}$, we only need to control
the $L^p$ norm of $\omega=\nabla^{\perp}\cdot\mathbf{v}=( -\partial_y,\partial_x)\cdot\mathbf{v}$.
After a straightforward calculation, it can be seen that $\omega$ solves
   \begin{align}\label{three-proof--3}
\begin{cases}
\frac{\partial \omega}{\partial t}+ (\mathbf{v} \cdot \nabla )\omega = \mu \Delta \omega ,\\
\omega=0,\quad x^2+y^2=b^2,\\
\omega=-\phi v_{\tau}. \quad x^2+y^2=a^2,\\
\omega|_{t=0}= \omega_0 ,
\end{cases}
\end{align}
where $\phi $ is given by
   \begin{align}\label{biaoji0722}
\phi:=2a^{-1}-\alpha \mu^{-1},
\end{align}
 and the second and third equations of \eqref{three-proof--3} are derived from \eqref{cond-1} and
\begin{align*}
\omega=\partial_xv_2-\partial_yv_1=
\frac{\partial v_{\tau}}{\partial \mathbf{n}}-\frac{v_{\tau}}{r}
+\frac{1}{r}\frac{\partial v_{\mathbf{n} }}{\partial \tau}
\end{align*}
with $\mathbf{n}$ and $\tau$ on $x^2+y^2=a^2$ being given by
\[
\mathbf{n}=-(\cos\theta,\sin\theta),\quad \tau
=(\sin\theta,-\cos\theta).
\]

The third equation of \eqref{three-proof--3} makes it difficult to directly estimate
$\norm{\omega}_{L^p}$. Fix an arbitrary $T > 0$, let us set $\lambda =\norm{\phi v_{\tau}}_{L^{\infty}((0,T)\times \partial B(0,a)}$, then
we consider the system:
   \begin{align*}
\begin{cases}
\frac{\partial \omega^{\pm}}{\partial t}+ (\mathbf{v} \cdot \nabla )\omega^{\pm} = \mu \Delta \omega^{\pm} ,\\
\omega^{\pm}=\pm\lambda,\quad x^2+y^2=b^2,\\
\omega^{\pm}=\pm\lambda. \quad x^2+y^2=a^2,\\
\omega^{\pm}|_{t=0}=\pm \abs{\omega_0}.
\end{cases}
\end{align*}
Let us denote $\eta^{\pm}=\omega-\omega^{\pm}$, then $\eta^{\pm}$ satisfies
   \begin{align*}
\begin{cases}
\frac{\partial \eta^{\pm}}{\partial t}= \mu \Delta  \eta^{\pm}-(\mathbf{v} \cdot \nabla ) \eta^{\pm},\quad (x.y)\in \Omega ,\\
\eta^{\pm}=\mp\lambda,\quad x^2+y^2=b^2,\\
\eta^{\pm}=\phi v_{\tau}
\mp\lambda. \quad x^2+y^2=a^2,\\
\eta^{\pm}|_{t=0}=\omega_0\mp \abs{\omega_0}.
\end{cases}
\end{align*}
Since the initial and boundary values of $\eta^{\pm}$
are sign definite, by the maximum principle, it gives
   \begin{align*}
\eta^{+}\leq 0,\quad \eta^{-}\geq 0
\end{align*}
which implies
   \begin{align*}
\omega^{-}\leq \omega \leq \omega^{+},\quad
\abs{\omega}\leq\max\{\abs{\omega^{+}},\abs{\omega^{-}}\}.
\end{align*}

Let us define $\sigma =\omega^{+}-\lambda $, then $\sigma$ solves
   \begin{align}\label{three-proof--8}
\begin{cases}
\frac{\partial \sigma}{\partial t}+ (\mathbf{v} \cdot \nabla )\sigma =
\mu \Delta\sigma,\\
\sigma=0,\quad x^2+y^2=b^2,\\
\sigma=0, \quad x^2+y^2=a^2,\\
\sigma|_{t=0}= \abs{\omega_0}-\lambda.
\end{cases}
\end{align}
Multiplying the first equation of \eqref{three-proof--8} by $\abs{\sigma}^{p-2}\sigma$ in $L^2$, one sees that
\begin{align}\label{three-proof--9}
\begin{aligned}
\frac{1}{p}\frac{d}{dt}\norm{\sigma}^p_{L^p}
=&-(p-1)\mu
\norm{\abs{\sigma}^{\frac{p-2}{2}}\abs{\nabla \sigma}}^2_{L^2}
\end{aligned}
\end{align}

 For $p=2$, we have
 \begin{align*}
\begin{aligned}
\frac{1}{2}\frac{d}{dt}\norm{\sigma}^2_{L^2}
\leq &-\frac{\mu}{2}C_a^b
\norm{\sigma}^2_{L^2}
\end{aligned}
\end{align*}
which gives
 \begin{align}\label{three-proof--9-2}
\begin{aligned}
\norm{\sigma}^2_{L^2}
\leq e^{-C_a^b\mu t}
\norm{\abs{\omega_0}-\lambda}^2_{L^2}.
\end{aligned}
\end{align}
Apply H\"{o}lder inequality and Poincar\'{e} inequality to \eqref{three-proof--9}, we have
\begin{align*}
\begin{aligned}
\frac{1}{p}\frac{d}{dt}\norm{\sigma}^p_{L^p}
+\frac{ 2C_a^b(p-1)\mu}{p^2}
\norm{\sigma}^p_{L^p}
\leq0,
\end{aligned}
\end{align*}
which, combined with Gronwall's inequality, results in
\begin{align}\label{three-proof--13}
\begin{aligned}
\norm{\sigma}_{L^p}=
\norm{\omega^{+}-\lambda}_{L^p}
\leq e^{-\frac{ 2C_a^b(p-1)\mu}{p^2}t}
\norm{\abs{\omega_0}-\lambda}_{L^p}.
\end{aligned}
\end{align}

Similarly, we have
\begin{align}\label{three-proof--14}
\begin{aligned}
\norm{\omega^{-}+\lambda}_{L^p}
\leq e^{-\frac{ 2C_a^b(p-1)\mu}{p^2}t}
\norm{-\abs{\omega_0}+\lambda}_{L^p}.
\end{aligned}
\end{align}
Thus, we infer from \eqref{three-proof--9-2},\eqref{three-proof--13}-\eqref{three-proof--14} and
$\norm{\omega}_{L^p}\leq \max\{
\norm{\omega^{-}+\lambda}_{L^p},
\norm{\omega^{+}-\lambda}_{L^p}
\}+\norm{\lambda}_{L^p}$ that
\begin{align}\label{three-proof--15}
\begin{aligned}
\norm{\omega}_{L^p}\leq
e^{-\frac{ 2C_a^b(p-1)\mu}{p^2}t}
\norm{\omega_0}_{L^p}+\abs{\Omega}^{\frac{1}{p}}\lambda
\left(e^{-\frac{ 2C_a^b(p-1)\mu}{p^2}t}
+1\right),\quad p\geq 2.
\end{aligned}
\end{align}

Taking $\epsilon$ in \eqref{lama} sufficient small, such that $2\epsilon C\abs{\Omega}^{\frac{1}{p}}\leq \frac{1}{2}$, one concludes
\begin{align}\label{three-proof--16}
\begin{aligned}
\norm{\omega}_{L^p}\leq 2
e^{-\frac{ 2C_a^b(p-1)\mu}{p^2}t}
\norm{\omega_0}_{L^p}+4\abs{\Omega}^{\frac{1}{p}}D_{\epsilon}
\norm{\mathbf{v}}_{L^2}.
\end{aligned}
\end{align}
Finally, \autoref{lemma-grad-1}, \eqref{houmian0721} and \eqref{three-proof--16} show that \eqref{3333} holds.
\end{proof}

\subsubsection{Decay of  $\norm{q}_{H^1(\Omega)}$}

\begin{lemma}\label{pres-es}
For the strong solutions \(\mathbf{v}\) of the problem \eqref{main=p1} subject to the conditions \eqref{cond-3}-\eqref{average} with initial data \(\mathbf{v}_0 \in H^{2}\left(\Omega\right)\),
suppose $q$ is average-free, we have
\begin{align}
\norm{q}_{H^1}^2\leq C\norm{ \mathbf{v}}^4_{W^{1,4}}\leq
 C\left(
2
e^{-\frac{ 3C_a^b\mu}{8}t}
\norm{\omega_0}_{L^4}+4\abs{\Omega}^{\frac{1}{4}}D_{\epsilon}
e^{\lambda_1(\mu)t}\left\|\mathbf{v}_{0}\right\|_{L^2}
\right)^4.
\end{align}
\end{lemma}
\begin{proof} It follows from \eqref{main=p1}  that the pressure $q$ solves the following equations:
  \begin{align}\label{213-3}
\begin{cases}
\Delta q=-\nabla \cdot \left((\mathbf{v} \cdot \nabla )
\mathbf{v}\right),
 \\
 \mathbf{n}\cdot \nabla q= \frac{\abs{\mathbf{v}}^2}{b}, \quad x^2+y^2=b^2,\\
  \mathbf{n}\cdot \nabla q=- \frac{\abs{\mathbf{v}}^2}{a}
  +\mu\phi 
  \tau \cdot \nabla v_{\tau},\quad x^2+y^2=a^2,
\end{cases}
\end{align}
where $\phi$ is \eqref{biaoji0722} and we have used
\begin{align*}
 &\mathbf{n}\cdot (\mathbf{v} \cdot \nabla )\mathbf{v} =
\frac{\abs{\mathbf{v}}^2}{a},\quad
 \mathbf{n}\cdot \Delta \mathbf{v}=
 \mathbf{n}\cdot \nabla^{\perp}\omega
 =- \tau \cdot \nabla \omega
 =\phi \tau \cdot \nabla \left(
v_{\tau}\right),\quad x^2+y^2=a^2,\\
 &\mathbf{n}\cdot (\mathbf{v} \cdot \nabla )\mathbf{v} =-
\frac{\abs{\mathbf{v}}^2}{b},\quad
 \mathbf{n}\cdot \Delta \mathbf{v}=
 \mathbf{n}\cdot \nabla^{\perp}\omega
 =- \tau \cdot \nabla \omega
 =0,\quad x^2+y^2=b^2.
\end{align*}

Taking the $L^2$-inner product of the first equations of \eqref{213-3} with $q$, we obtain
                 \begin{align}\label{p-1}
           \begin{aligned}
\int_{\Omega}q\Delta q\,dx\,dy&=
-\int_{\Omega}q\nabla \cdot \left((\mathbf{v} \cdot \nabla )
\mathbf{v}\right)\,dx\,dy.
   \end{aligned}
     \end{align}
By performing integration by parts on the left-hand side of \eqref{p-1} and utilizing the boundary conditions specified in \eqref{213-3}, we obtain
               \begin{align}\label{p-2}
           \begin{aligned}
\int_{\Omega}q\Delta q\,dx\,dy&=-
\int_{\Omega}\abs{\nabla q}^2\,dx\,dy
+\int_{\partial \Omega}q \mathbf{n}\cdot\nabla q\,ds\\
&=-\int_{\Omega}\abs{\nabla q}^2\,dx\,dy
 +\int_{\partial B(0,b)}q
 \frac{\abs{\mathbf{v}}^2}{b}\,ds \\&\quad
-\int_{\partial B(0,a)}q
 \frac{\abs{\mathbf{v}}^2}{a}\,ds
+\mu\phi \int_{\partial B(0,a)}q
   \tau \cdot\nabla v_{\tau}\,ds.
   \end{aligned}
     \end{align}
We then combine \eqref{p-1} with \eqref{p-2} to infer that
               \begin{align}\label{p-3}
               \begin{aligned}
\int_{\Omega}\abs{\nabla q}^2\,dx\,dy
&=\int_{\partial B(0,b)}q
 \frac{\abs{\mathbf{v}}^2}{b}\,ds-\int_{\partial B(0,a)}q
\frac{\abs{\mathbf{v}}^2}{a}\,ds\\&\quad+\mu
\phi \int_{\partial B(0,a)}q
   \tau \cdot\nabla v_{\tau}
\,ds
+\int_{\Omega}q\nabla \cdot \left((\mathbf{v} \cdot \nabla )
\mathbf{v}\right)\,dx\,dy\\&=J_0+J_1+J_2+J_3,
   \end{aligned}
     \end{align}
where $J_{0}$ and $J_1$ can be bounded as follows, using the trace theorem, H\"{o}lder's and Young's inequalities.
\begin{align}\label{p-31}
 \begin{aligned}
J_0+J_1=&\int_{\partial B(0,b)}q
 \frac{\abs{\mathbf{v}}^2}{b}\,ds-\int_{\partial B(0,a)}q
\frac{\abs{\mathbf{v}}^2}{a}\,ds\\&\leq
C\norm{q{\abs{\mathbf{v}}^2}}_{W^{1,1}}
\leq   C\norm{q}_{H^1}  \norm{\abs{\mathbf{v}}^2}_{H^1}
\\ & \leq C\norm{q}_{H^1}  \norm{\mathbf{v}}_{L^{4}} \norm{\mathbf{v}}_{W^{1,4}}
\leq \epsilon \norm{q}_{H^1}^2+
C_{\epsilon} \norm{\mathbf{v}}_{W^{1,4}}^4.
   \end{aligned}
     \end{align}

Regarding $J_2$, we carry out a straightforward computation to obtain
                     \begin{align*}
               \begin{aligned}
\mu\phi  \int_{\partial B(0,a)}q
   \tau \cdot\nabla v_{\tau}
\,ds&=\mu\phi \int_{\partial B(0,a)}
   \mathbf{n} \cdot q\nabla^{\perp}(\tau\cdot\mathbf{v}) \,ds
  \\& =-
\int_{\Omega}
    \nabla \cdot (q\Phi \nabla^{\perp}(\tau\cdot\mathbf{v}) )\,dx
   \end{aligned}
     \end{align*}
 which with $\nabla\cdot\nabla^{\perp}\left(\tau\cdot\mathbf{v}\right)=0$ implies
 \begin{align}\label{p-5}
\begin{aligned}
 J_2&\leq \mu\phi \abs{ \int_{\partial B(0,a)}q
   \tau \cdot\nabla
  v_{\tau}\,ds}
 \\& \leq C
  \norm{q}_{H^1}  \norm{\mathbf{v}}_{H^1}\leq  \epsilon   \norm{q}_{H^1}^2+C_{\epsilon}\norm{\mathbf{v}}_{H^1}^2,
\end{aligned}
 \end{align}
 where $\Phi \in H^{1}(\Omega)$, such that
\[
\begin{cases}
\Phi =\mathbf{0},\quad x^2+y^2=b^2,\\
\Phi =\phi,\quad x^2+y^2=a^2,\\
\norm{\Phi }_{W^{1,\infty}}\leq C\abs{\phi}.
\end{cases}
\]

 For $J_3$, we have the following estimates:
  \begin{align}\label{p-6}
 \begin{aligned}
  J_3&=-\int_{\Omega}q\nabla \cdot \left((\mathbf{v} \cdot \nabla )
\mathbf{v}\right)\,dx\,dy\leq
C\norm{q\abs{\nabla \mathbf{v}}^2}_{L^1}\\&\leq
C\norm{q}_{L^2}
\norm{\nabla \mathbf{v}}_{L^4}^2
\leq \epsilon \norm{q}_{H^1}^2+C_{\epsilon}
\norm{ \mathbf{v}}^4_{W^{1,4}}.
   \end{aligned}
 \end{align}
Since $q$ is only defined up to a constant, we choose it such that $q$ is average-free. Combining \eqref{p-31}, we ultimately establish the following inequality:
  \begin{align}\label{p-9}
 \begin{aligned}
\norm{q}_{H^1}^2\leq C_{\epsilon}\norm{ \mathbf{v}}^4_{W^{1,4}} <\infty ,
 \end{aligned}
 \end{align}
 where \autoref{lemma-grad-1-2} has been used.
\end{proof}

\subsubsection{Decay of $\norm{\mathbf{v}_t}_{L^{2}(\Omega)}$}

\begin{lemma}\label{utestimate}
For the strong solution \(\mathbf{v}\) of the problem \eqref{main=p1} subject to the conditions \eqref{cond-3}-\eqref{average} with initial data \(\mathbf{v}_0 \in H^{2}\left(\Omega\right)\),
there exists $\sigma_1>0$ such that
\begin{align}\label{tul2h1}
\begin{aligned}
\norm{\mathbf{v}_t}^2_{L^2}\leq &C\norm{\mathbf{v}_0}_{H^2}^2 e^{-\sigma_1 t}
\end{aligned}
\end{align}
\end{lemma}
\begin{proof} 
By differentiating the corresponding equations and the conditions \eqref{cond-3}-
\eqref{average} with respect to $t$, we observe that $\frac{\partial \mathbf{v}}{\partial t}$ satisfies the following system 
   \begin{align}\label{ut-1}
\begin{cases}
\mathbf{v}_{tt}+ (\mathbf{v}_t \cdot \nabla )\mathbf{v}
+ (\mathbf{v}\cdot \nabla )\mathbf{v}_t
= \mu \Delta \mathbf{v}_t-\nabla q_t,
\\ \nabla \cdot  \mathbf{v}_t=0.
\end{cases}
\end{align}
The boundary conditions for the system \eqref{ut-1} are as follows:
  \begin{align}\label{cond-1-1}
  \begin{aligned}
  &\mathbf{v}_t \cdot   \mathbf{n}|_{x^2+y^2=a^2}=0,\quad
  \mathbf{v}_t \cdot   \mathbf{n}|_{x^2+y^2=b^2}=\nabla \times   \mathbf{v}_t|_{x^2+y^2=b^2}=0,\\ 
   &\left[\left(-q_t\mathbf{I} +\mu\left (\nabla \mathbf{v}_t+\left(\nabla \mathbf{v}_t\right)^{Tr} \right)\right)\cdot\mathbf{n}\right]
   \cdot   \mathbf{\tau}|_{x^2+y^2=a^2}
  =\alpha  \mathbf{v}_t  \cdot   \mathbf{\tau}|_{x^2+y^2=a^2},\quad \alpha \geq 0.
  \end{aligned}
  \end{align}

Taking the $L^2$-inner product of \eqref{ut-1} with $\mathbf{v}_t $ and using
  \autoref{lemma-grad-1-1}, we infer that
  \begin{align}\label{ut-proof-1}
\begin{aligned}
&\frac{d}{dt}\int_{\Omega}\frac{\abs{\mathbf{v}_t}^2}{2}\,dx\,dy
+\mu\int_{\Omega}\abs{\nabla \mathbf{v}_t}^2\,dx\,dy
\\&
+\frac{\mu}{b}\int_{\partial B(0,b)}\left(\tau \cdot \mathbf{v}_t\right)^2 \,ds
-\left(\alpha-\mu a^{-1}\right)\int_{\partial B(0,a)}\left(\tau \cdot \mathbf{v}_t\right)^2 \,ds
\\  &\leq  \int_{\Omega}\abs{
\mathbf{v}_t\cdot (\mathbf{v}_t\cdot\nabla )\mathbf{v}}
 \,dx\,dy.
\end{aligned}
\end{align}
If $\left(\alpha-\mu a^{-1}\right)\leq0$, we have
  \begin{align}\label{ut-proof-1-new}
\begin{aligned}
&\frac{d}{dt}\int_{\Omega}\frac{\abs{\mathbf{v}_t}^2}{2}\,dx\,dy
+\mu\int_{\Omega}\abs{\nabla \mathbf{v}_t}^2\,dx\,dy
+\frac{\mu}{b}\int_{\partial B(0,b)}\left(\tau \cdot \mathbf{v}_t\right)^2 \,ds
\leq  \int_{\Omega}\abs{
\mathbf{v}_t\cdot (\mathbf{v}_t\cdot\nabla )\mathbf{v}}
 \,dx\,dy.
\end{aligned}
\end{align}
If $\left(\alpha-\mu a^{-1}\right)>0$, we have
    \begin{align}\label{ut-proof-30}
&\begin{aligned}
&-\left(\alpha-\mu a^{-1}\right)\int_{\partial B(0,a)}\left(\tau \cdot \mathbf{v}_t\right)^2 \,ds\\
&\geq
-\left(\alpha-\mu_c a^{-1}\right)\int_{\partial B(0,a)}\left(\tau \cdot \mathbf{v}_t\right)^2 \,ds
\\&\geq-\mu_c \int_{\Omega}\abs{\nabla \mathbf{v}_t}^2\,dx\,dy
-\frac{\mu_c}{b}\int_{\partial B(0,b)}\left(\tau \cdot \mathbf{v}_t\right)^2 \,ds
\end{aligned}
\end{align}
where the definition of $\mu_{c}$ is used. Then,
 from \eqref{ut-proof-1} we have
  \begin{align}\label{ut-proof-1-new2}
\begin{aligned}
&\frac{d}{dt}\int_{\Omega}\frac{\abs{\mathbf{v}_t}^2}{2}\,dx\,dy
+(\mu-\mu_c)\int_{\Omega}\abs{\nabla \mathbf{v}_t}^2\,dx\,dy
+\frac{\mu-\mu_c}{b}\int_{\partial B(0,b)}\left(\tau \cdot \mathbf{v}_t\right)^2 \,ds
\\&\leq  \int_{\Omega}\abs{
\mathbf{v}_t\cdot (\mathbf{v}_t\cdot\nabla )\mathbf{v}}
 \,dx\,dy.
\end{aligned}
\end{align}
Hence, for $\mu>\mu_c$, there exists $\sigma>0$ depending on $\mu$ such that 
  \begin{align}\label{ut-proof-1-new-3}
\begin{aligned}
&\frac{d}{dt}\int_{\Omega}\frac{\abs{\mathbf{v}_t}^2}{2}\,dx\,dy
+\sigma \int_{\Omega}\abs{\nabla \mathbf{v}_t}^2\,dx\,dy
+\sigma \int_{\partial B(0,b)}\left(\tau \cdot \mathbf{v}_t\right)^2 \,ds
\leq  \int_{\Omega}\abs{
\mathbf{v}_t\cdot (\mathbf{v}_t\cdot\nabla )\mathbf{v}}
 \,dx\,dy.
\end{aligned}
\end{align}

By H\"older's, Ladyzhenskaya's and Young's inequalities, one deduces that
    \begin{align}\label{ut-proof-3}
&\begin{aligned}
 \int_{\Omega}\abs{
\mathbf{v}_t\cdot (\mathbf{v}_t\cdot\nabla )\mathbf{v}}
 \,dx\,dy\leq &\norm{
 \nabla \mathbf{v}}_{L^2}
 \norm{ \mathbf{v}_t}^2_{L^4}\leq
C \norm{
 \nabla \mathbf{v}}_{L^2}
 \norm{ \nabla  \mathbf{v}_t}_{L^2}
\norm{\mathbf{v}_t}_{L^2}\\
\leq &\epsilon
\norm{\nabla \mathbf{v}_t}_{L^2}^2+C_{\epsilon}
  \norm{
 \nabla \mathbf{v}}^2_{L^2}
 \norm{
\mathbf{v}_t}^2_{L^2}.
\end{aligned}
\end{align}

By choosing $\epsilon$ sufficiently small, we combine the results from \eqref{ut-proof-3} to derive the following conclusion:
  \begin{align}\label{ut-proof-5}
\begin{aligned}
&\frac{d}{dt}\norm{\mathbf{v}_t}^2_{L^2}
+\frac{\sigma}{2}  \norm{\nabla\mathbf{v}_t}^2_{L^2
}
+\frac{\sigma}{2} \int_{\partial B(0,b)}\left(\tau \cdot \mathbf{v}_t\right)^2 \,ds
\leq C
\norm{ \mathbf{v}}^2_{H^1}
  \norm{
\mathbf{v}_t}^2_{L^2}.
\end{aligned}
\end{align}
The \autoref{lemma-grad-1-2} and \autoref{lemma-grad}
mean that for large enough time there exists $\sigma_1>0$ such that
  \begin{align}\label{ut-proof-5-new}
\begin{aligned}
&\frac{d}{dt}\norm{\mathbf{v}_t}^2_{L^2}
+ \sigma_1 \norm{
\mathbf{v}_t}^2_{L^2}\leq 0,
\end{aligned}
\end{align}
which, combined with Gronwall's inequality, yields
  \begin{align}\label{ut-proof-6}
\begin{aligned}
\norm{\mathbf{v}_t}^2_{L^2}
\leq &C\norm{\mathbf{v}_t(0)}^2_{L^2}e^{-\sigma_1 t}.
\end{aligned}
\end{align}

To bound $\norm{\mathbf{v}_t(0)}^2_{L^2}$, we multiply the equations \eqref{ut-1} with $\mathbf{v}_t$ in $L^2$ to get
\begin{align*}
\begin{aligned}
&\int_{\Omega}\abs{\mathbf{v}_t}^2\,dx\,dy
-\mu\int_{\Omega}\Delta \mathbf{v}\cdot  \mathbf{v}_t\,dx\,dy
+\int_{\Omega} \mathbf{v}_t\cdot( \mathbf{v}\cdot \nabla )\mathbf{v}\,dx\,dy
=0.
\end{aligned}
\end{align*}
From this, it follows that
  \begin{align*}
\begin{aligned}
\norm{\mathbf{v}_t}^2_{L^2}
\leq \nu \norm{\mathbf{v}_t}_{L^2}
\norm{\Delta \mathbf{v}}_{L^2}
+ \norm{\mathbf{v}_t}_{L^2}
\norm{( \mathbf{v}\cdot \nabla )\mathbf{v}}_{L^2},
\end{aligned}
\end{align*}
which gives
  \begin{align*}
\begin{aligned}
\norm{\mathbf{v}_t(0)}_{L^2}^2\leq
C\norm{\mathbf{v}_0}_{H^2}^2
\end{aligned}
\end{align*}
which together with \eqref{ut-proof-6} yields  \eqref{tul2h1}.
\end{proof}

\subsubsection{Decay of $\norm{\mathbf{v}}_{H^{2}(\Omega))}$}

\begin{lemma}\label{21336}
For the solution $\mathbf{v}$ of the problem \eqref{main=p1}
subject to the boundary conditions \eqref{cond-3} and initial data $\mathbf{v}_0 \in H^{2}\left(\Omega\right)$,
there exists $\sigma_1>0$ such that
\begin{align}\label{tul2h13}
\begin{aligned}
\norm{\mathbf{v}}^2_{H^2}\leq &C\norm{\mathbf{v}_0}_{H^2}^2 e^{-\sigma_1 t}
\end{aligned}
\end{align}
\end{lemma}
\begin{proof}
From $\eqref{main=p1}_{1}$ and Stokes' estimate, we have
\begin{align}
\begin{aligned}
\mu \norm{\nabla^2 \mathbf{v}}_{L^2}&\leq
\norm{\mathbf{v}_t+ (\mathbf{v} \cdot \nabla )\mathbf{v} +
 \nabla q}_{L^2}\\
&\leq C\left(\norm{\mathbf{v}_t}_{L^2}+
\norm{\mathbf{v}  }_{H^{1}}^2
+\norm{\nabla q }_{L^2}\right),
\end{aligned}
\end{align}
which together with \autoref{lemma-grad-1-2}- \autoref{utestimate} yields \eqref{tul2h13}.
\end{proof}
\subsection{Proof of \autoref{criticalcase}}\label{section33}

Based on \autoref{eigennumber}, we observe that the eigenvalue problem \eqref{polar-l-eigen}, subject to the conditions \eqref{cond-3} and \eqref{average}, possesses a countably infinite number of eigenvalues. 
We denote the solutions of \eqref{polar-l-eigen}, subject to the conditions \eqref{cond-3} and \eqref{average} by $\{(\lambda_{m}, \mathbf{v}_{m})\}_{m\in \N}$, where the eigenvalues, counting multiplicity, are ordered by decreasing real part, that is,
\begin{align*}
-\infty\leftarrow \cdots  \leq\lambda_{3}\leq\lambda_{2}\leq \lambda_1.
\end{align*}

It is not hard to see that
the first eigenvalue of \eqref{polar-l-eigen} has multiplicity two, and $\{\mathbf{v}_{m}\}_{m\in \N}$ forms an orthogonal basis of $X_0(\Omega)$. Hence, all eigenvalues 
$\{\lambda_{m}\}_{m\in \N}$ satisfy 
\begin{align}\label{zero}
\begin{aligned}
&\lambda_{1}=\lambda_{2}=0,\quad \mu=\mu_c,\\
&\lambda_{m}<0,\quad \mu=\mu_c.
\end{aligned}
\end{align}
Let us denote two subsets $E_0$ and $E_0^{\perp}$ as follows
\begin{align}\label{e01}
\begin{aligned}
&E_0=\text{Span}\{\mathbf{v}_{1},\mathbf{v}_{2}\},\\
&E_0^{\perp}=\left\{\mathbf{w}\in X_0(\Omega) |(\mathbf{w},\mathbf{u})_{L^2}=0,~ \forall
\mathbf{u}\in E_0
 \right\}.
 \end{aligned}
\end{align}

For the strong solution $\mathbf{v}$ of \eqref{main=p1}-\eqref{average}, taking the following expansion 
\begin{align}
\mathbf{v}=\sum_{m=1}^{\infty}\alpha_m\mathbf{v}_{m}
\end{align}
testing the abstract form 
$\eqref{abstract-1}_1$ of \eqref{main=p1}-\eqref{average}
by $\mathbf{v}$, integrating by part over $\Omega$, we have
\begin{align*}
\begin{aligned}
\frac{1}{2}\frac{d}{dt}\int_{\Omega}\abs{\mathbf{v}}^2\,dx\,dy
&=\left(-
\mathcal{A}_{\mu}\mathbf{v}+\mathcal{F}(\mathbf{v},\mathbf{v}),
\mathbf{v}\right)
=\left(-
\mathcal{A}_{\mu}\mathbf{v},
\mathbf{v}\right)\\&=
\sum_{m=3}^{\infty}\lambda_{m}\alpha_m^2
\int_{\Omega}\abs{\mathbf{v}_m}^2\,dx\,dy\leq 0.
\end{aligned}
\end{align*}
Integrate the preceding inequality from $0$ to $t$, we get 
\begin{align}\label{deady-1}
\begin{aligned}
\norm{\mathbf{v}}_{L^2}^2\leq 
\norm{\mathbf{v}_0}_{L^2}^2-2\abs{\lambda_{3}}
\sum_{m=3}^{\infty}\alpha_m^2
\int_0^t\norm{\mathbf{v}_m(\tau)}_{L^2}^2\,d\tau
=\norm{\mathbf{v}_0}_{L^2}^2
-2\abs{\lambda_{3}}
\int_0^t\norm{\tilde{\mathbf{v}}(\tau)}_{L^2}^2\,d\tau
\end{aligned}
\end{align}
where $\mathbf{v}=\mathbf{u}+\tilde{\mathbf{v}}\in X_0(\Omega)=E_0\oplus E_0^{\perp}$ in which 
\begin{align*}
\begin{aligned}
&\mathbf{u}=\alpha_1\mathbf{v}_{1}+\alpha_2\mathbf{v}_{2},\\
&\tilde{\mathbf{v}}=\sum_{m=3}^{\infty}\alpha_m\mathbf{v}_{m}.
\end{aligned}
\end{align*}

It is straightforward to observe that for any $\mathbf{v}_0\in X_2(\Omega)$, the solution $\mathbf{v}(t, \mathbf{v}_0)$ of the problem \eqref{main=p1}-\eqref{average} (with $\mathbf{v}(0, \mathbf{v}_0) = \mathbf{v}_0$) is non-increasing in $t$, i.e.,
\[
\norm{\mathbf{v}(t_2, \mathbf{v}_0)}_{L^2} \leq \norm{\mathbf{v}(t_1, \mathbf{v}_0)}_{L^2}, \quad \forall t_1 < t_2, \quad \mathbf{v}_0 \in X_2(\Omega).
\]
Consequently, for any $\mathbf{v}_0\in X_2(\Omega)$, the limit 
\[
\lim_{t \to \infty} \norm{\mathbf{v}(t, \mathbf{v}_0)}_{L^2} \text{ exists.}
\]

For any $\mathbf{v}_0\in X_2(\Omega)$, we have
\[
\lim_{t \to \infty} \norm{\mathbf{v}(t, \mathbf{v}_0)}_{L^2} =
 \lim_{t \to \infty} \norm{\mathbf{u}(t, \mathbf{v}_0)
 + \tilde{\mathbf{v}}(t, \mathbf{v}_0)}_{L^2} = \delta \leq \norm{\mathbf{v}_0}_{L^2}.
\]
The $\omega$-limit set of $\mathbf{v}_0$, which is an invariant set, satisfies
\[
\omega(\mathbf{v}_0) \subset S_\delta = \left\{ \mathbf{w}
 \in X_0(\Omega) \mid \norm{\mathbf{w}}_{L^2} = \delta \right\}.
\]
Since $\omega(\mathbf{v}_0) $ is invariant, for any $\mathbf{w}\in \omega(\mathbf{v}_0) $, we have
\begin{align}\label{decay-2}
\mathbf{v}(t,\mathbf{w}) \subset \omega(\mathbf{v}_0) \subset S_\delta, \quad \forall t \geq 0. 
\end{align}
If $\mathbf{w}=\mathbf{w}_0 + \mathbf{w}_1$, where $\mathbf{w}_0 \in E_0$, $\mathbf{w}_1 \in E_0^\perp$, and $\mathbf{w}_1\neq 0$, then \eqref{deady-1} implies
\[
\norm{\mathbf{v}(t,\mathbf{w})}_{L^2} < \norm{\mathbf{w}}_{L^2}= \delta, \quad \forall t > 0.
\]
This contradicts \eqref{decay-2}. Therefore, we conclude
\[
 \omega(\mathbf{v}_0)  \subset E_0, \quad \forall~ \mathbf{v}_0 \in X_2(\Omega).
\]

If conclusion (2) of \autoref{criticalcase} does not hold, 
then there exists a sequence $\mathbf{v}_0^n \in X_2(\Omega)$ such that $\mathbf{v}_0^n \to 0$ in the $L^2$-norm as $n \to \infty$, with $0 \notin \omega(\mathbf{v}_0^n) \subset E_0$, and
\[
\lim_{n \to \infty} \text{dist}(\omega(\mathbf{v}_0^n), 0) = 0,
\]
which means conclusion (1) of \autoref{criticalcase} holds.

If conclusion (1) of \autoref{criticalcase} does not hold, then there must exist a neighborhood $U \subset X_2(\Omega)$ of $\mathbf{0}$ such that for all $\mathbf{v}_0 \in U$,
\[
\omega(\mathbf{v}_0) = \{0\}.
\]

\subsection{Proof of \autoref{noninstability}}\label{section34}
For $\mu<\mu_c$, it follows from \autoref{posi-eigen} that the linear problem 
\eqref{main=p-l}, subject to conditions \eqref{cond-3}-\eqref{average}, has a solution
as follows $
\mathbf{v}=\mathbf{u}_1e^{\lambda_1 t}$,
where $\mathbf{u}_1\in C^{\infty}(\Omega)\cap X_2(\Omega)$ is the first eigefunction given in \autoref{first-eigen}, $\lambda_{1}>0$, and it satisfies $\norm{\mathbf{u}_1}_{H^2}=1$, 
$\norm{\mathbf{u}_1}_{L^2}=C_1$ and $\norm{\mathbf{u}_1}_{L^1}=C_2$.
Choosing $\mathbf{v}_0:=\mathbf{v}(0)=\delta \mathbf{u}_1$,
for $\delta\in (0,\delta_0)$, by \autoref{existence},
the problem \eqref{main=p1}-\eqref{average} has a global solution
$(\mathbf{v}^{\delta},q^{\delta})\in  C([0,T];H^2(\Omega)\times H^1(\Omega))$,
its initial data $\mathbf{v}^{\delta}(0)=\delta \mathbf{u}_1$ satisfies
$\norm{\mathbf{v}^{\delta}(0)}_{H^2}=\delta$.

For any $\delta \in (0,\delta_0)$ satisfying $\delta<\epsilon_0$, let us define 
\begin{align}
T_{\delta}=\lambda_1^{-1}\ln \frac{\epsilon_0}{\delta},i.e., \epsilon_0=\delta e^{\lambda_1 T_{\delta}}
\end{align}
where $\epsilon_0$ independent of $\delta$ is a small positive constant to be determined,
and define 
\begin{align*}
&T_{*}=\sup\{t\in (0,\infty)|\norm{\mathbf{v}^{\delta}}_{H^2}\leq \delta_0\},\\
&T_{**}=\sup\{t\in (0,\infty)|\norm{\mathbf{v}^{\delta}}_{L^2}\leq 2C_1\delta e^{\lambda_1t}\}.
\end{align*}
By the definitions of $T_{*}$ and $T_{**}$, we have $T_{*}>0$, $T_{**}>0$, and 
\begin{align*}
&\norm{\mathbf{v}^{\delta}(T_{*})}_{H^2}=\delta_0,\quad \text{if}\quad T_{*}<\infty,\\
&\norm{\mathbf{v}^{\delta}(T_{**})}_{L^2}=2C_1\delta e^{\lambda_1T_{**}},\quad \text{if}\quad T_{**}<\infty.
\end{align*}
Note that for $t\leq \min\{T_{\delta},T_{*},T_{**}\}$, by \autoref{existence}, we have
 \begin{align}\label{nonlinear-delta}
\begin{aligned}
&\norm{\mathbf{v}^{\delta}}_{H^2}^2
+\norm{\mathbf{v}^{\delta}_t}_{L^2}^2
+\int_0^t\left(\norm{\nabla \mathbf{v}^{\delta}}_{L^2}^2
+
\norm{\mathbf{v}^{\delta}_t}_{H^1}^2
\right)\,ds
\\&\leq C_0\left(\delta^2
+4C_1^2\delta^2
\int_0^te^{2\lambda_1t}\,ds
\right)\leq 
 C_0\delta^2
 +4C_1^2C_0\delta^2\frac{e^{2\lambda_1t}}{2\lambda_1}
 \leq C_3\delta^2e^{2\lambda_1t}
\end{aligned}
\end{align}
where $C_3$ is independent of $\delta$.

Let $\mathbf{u}^{\delta}$ be the solution of linear problem 
\eqref{cond-3}-\eqref{main=p-l}
 with initial data $\delta \mathbf{u}_1$. Consider $\mathbf{v}_{e}^{\delta}=\mathbf{v}^{\delta}-\mathbf{u}^{\delta}$, it satisfies the following equations
     \begin{align}\label{main=p-ln}
\begin{cases}
\frac{\partial \mathbf{v}_{e}^{\delta}}{\partial t}+
 (\mathbf{v}^{\delta}\cdot \nabla )\mathbf{v}^{\delta} = \mu \Delta \mathbf{v}_{e}^{\delta}-
\nabla q^{\delta}\quad \mathbf{x}\in \Omega,
\\ \nabla \cdot  \mathbf{v}_{e}^{\delta}=0, \quad \mathbf{x}\in \Omega,\\
\mathbf{v}_{e}^{\delta}(0)=0,
\end{cases}
\end{align}
subject to the conditions \eqref{cond-3}-\eqref{average}.
Taking the $L^2$-inner product of \eqref{main=p-ln} with $ \mathbf{v}_{e}^{\delta}$, we have
  \begin{align}\label{ut-proof-1-ln}
\begin{aligned}
\frac{1}{2}\frac{d}{dt}\norm{\mathbf{v}_{e}^{\delta}}_{L^2}^2
&\leq 
-\mu\int_{\Omega}\abs{\nabla \mathbf{v}_{e}^{\delta}}^2\,dx\,dy
-\frac{\mu}{b}\int_{\partial B(0,b)}\left(\tau \cdot \mathbf{v}_{e}^{\delta}\right)^2 \,ds
\\  &\quad +\left(\alpha-\mu a^{-1}\right)\int_{\partial B(0,a)}\left(\tau \cdot \mathbf{v}_{e}^{\delta}\right)^2 \,ds
+
 \int_{\Omega}\abs{
\mathbf{v}_{e}^{\delta}\cdot (\mathbf{v}^{\delta}\cdot\nabla )\mathbf{v}^{\delta}}
 \,dx\,dy\\
 &\leq \lambda_1\norm{\mathbf{v}_{e}^{\delta}}_{L^2}^2
 +\norm{\mathbf{v}_{e}^{\delta}}_{L^2}
 \norm{\mathbf{v}^{\delta}}_{L^4}
 \norm{\nabla\mathbf{v}^{\delta}}_{L^4}\\
 &\leq 
  \lambda_1\norm{\mathbf{v}_{e}^{\delta}}_{L^2}^2
  +C\norm{\mathbf{v}_{e}^{\delta}}_{L^2}
 \norm{\mathbf{v}^{\delta}}_{L^4}^2
 +C\norm{\mathbf{v}_{e}^{\delta}}_{L^2}
 \norm{\nabla\mathbf{v}^{\delta}}_{L^4}^2\\
 &\leq 
   \lambda_1\norm{\mathbf{v}_{e}^{\delta}}_{L^2}^2
   +C\norm{\mathbf{v}_{e}^{\delta}}_{L^2}
 \norm{\mathbf{v}^{\delta}}_{H^2}^2
\end{aligned}
\end{align}
which yields that 
  \begin{align}\label{ut-proof-1-ln-2}
\begin{aligned}
\frac{d}{dt}\norm{\mathbf{v}_{e}^{\delta}}_{L^2}
 \leq 
   \lambda_1\norm{\mathbf{v}_{e}^{\delta}}_{L^2} +C
 \norm{\mathbf{v}^{\delta}}_{H^2}^2.
\end{aligned}
\end{align}
Gronwall inequality and \eqref{nonlinear-delta} give that
  \begin{align}\label{ut-proof-1-ln-3}
\begin{aligned}
\norm{\mathbf{v}_{e}^{\delta}}_{L^2}
 \leq C e^{ \lambda_1t}
 \int_0^t
 \norm{\mathbf{v}^{\delta}}_{H^2}^2\,ds
 \leq  CC_3\delta^2e^{ \lambda_1t}
  \int_0^te^{ 2\lambda_1s}\,ds
  \leq C_4\delta^2e^{ 2\lambda_1t}.
\end{aligned}
\end{align}

Let us choose $\epsilon_0$ such that $\epsilon_0<
\min\{\frac{\delta_0}{\sqrt{C_3}},\frac{C_1}{C_4},\frac{C_2}{2\abs{\Omega}^{\frac{1}{2}}C_4}
\}$. Then we aim to show $T_{\delta}=\min\{T_{\delta},T_{*},T_{**}\}$.
If $T_{*}=\min\{T_{\delta},T_{*},T_{**}\}$, $T_{\delta}>T_{*}$ and $T_{*}<\infty$, we have
\[
\norm{\mathbf{v}^{\delta}(T_{*})}_{H^2}\leq 
\sqrt{C_3}\delta e^{\lambda_1T_{*}}<
\sqrt{C_3}\epsilon_0\leq \delta_0,
\]
which is contradictory to the definition of $T_{*}$.
If $T_{**}=\min\{T_{\delta},T_{*},T_{**}\}$, $T_{\delta}>T_{**}$ and $T_{**}<\infty$, we have
\begin{align*}
\norm{\mathbf{v}^{\delta}(T_{**})}_{L^2}&\leq 
\norm{\mathbf{v}_e^{\delta}(T_{**})}_{L^2}
+\norm{\mathbf{u}^{\delta}(T_{**})}_{L^2}
\leq C_1\delta 
e^{\lambda_1T_{**}}+C_4\delta^2e^{2\lambda_1T_{**}}\\
&\leq 
C_1\delta 
e^{\lambda_1T_{**}}\left(
1+
\frac{C_4}{C_1}\delta e^{\lambda_1T_{**}}
\right)<
C_1\delta 
e^{\lambda_1T_{**}}\left(
1+
\frac{C_4}{C_1}\epsilon_0
\right)\leq 2
C_1\delta 
e^{\lambda_1T_{**}}.
\end{align*}
which is contradictory to the definition of $T_{**}$. By 
$T_{\delta}=\min\{T_{\delta},T_{*},T_{**}\}$, we finally get
\begin{align*}
&\norm{\mathbf{v}^{\delta}(T_{\delta})}_{L^1}\geq 
\norm{\mathbf{u}^{\delta}(T_{\delta})}_{L^1}
-\norm{\mathbf{v}_e^{\delta}(T_{\delta})}_{L^1}
\\&\geq C_2\delta e^{\lambda_1 T_{\delta}}-\abs{\Omega}^{\frac{1}{2}}
\norm{\mathbf{v}_e^{\delta}(T_{\delta})}_{L^2}
\\&\geq
C_2\epsilon_0-\abs{\Omega}^{\frac{1}{2}}
C_4\epsilon_0^2>\frac{C_2\epsilon_0}{2}.
\end{align*}
 Hence, for $p\in [1,+\infty]$, we get that 
 \begin{align*}
   \norm{\mathbf{v}^{\delta}(T_{\delta})}_{L^p}\geq  C\norm{\mathbf{v}^{\delta}(T_{\delta})}_{L^1}
     \geq C\frac{C_2\epsilon_0}{2}=\epsilon.
 \end{align*}

 \section{Proof of \autoref{criticalcase1}}\label{section4}
 
 The annular nature of the region under consideration makes it particularly advantageous to switch to polar coordinates when dealing with the system of equations \eqref{main=p1}-\eqref{average}. By doing so, we can more effectively capture the radial and angular dependencies inherent in the problem. 

The curl in polar coordinates is expressed as
\begin{align}
\omega=\partial_xv_2-\partial_yv_1=
\partial_rv_{\theta}+\frac{v_{\theta}}{r}
-\frac{\partial_{\theta}v_r}{r}.
\end{align}
This formula captures the rotation of the fluid in the annular domain and is essential for understanding the vorticity dynamics.
By eliminating the pressure term $q$ from the Navier-Stokes equations \eqref{main=p1}, we obtain the vorticity equation:
  \begin{align}\label{omega1}
 \frac{\partial \omega}{\partial t}+(\mathbf{v}\cdot\nabla_{r})\omega
    =\mu\Delta_{r}\omega,\quad
     \mathbf{v}\cdot\nabla_{r}=v_{r}\frac{\partial}{\partial r}+\frac{v_{\theta}}{r}\frac{\partial}{\partial \theta},~\Delta_{r}\omega=\left(\frac{\partial^2}{\partial r^2}+\frac{1}{r}\frac{\partial}{\partial r}+\frac{\partial^2}{\partial \theta^2}\right)\omega.
\end{align}
This equation not only describes how the vorticity $\omega$ evolves over time under the influence of the fluid velocity $\mathbf{u} $ and the kinematic viscosity $\mu$,but also can be used for simplifying the system \eqref{main=p1} into a system containing single equation.
In fact,  introducing a streamfunction $\psi$ such that
\begin{align}
v_{r}=-\frac{1}{r}\frac{\partial \psi}{\partial \theta},\quad
v_{\theta}=\frac{\partial \psi}{\partial r},
\end{align}
then the divergence-free condition $
\frac{\partial (rv_{r})}{\partial r}+\frac{\partial v_{\theta}}{\partial \theta}=0$
 is automatically satisfied. This allows us to rewrite the system of equations \eqref{main=p1}-\eqref{average} in terms of the streamfunction $\psi$ as follows
  \begin{align}\label{omega-2}
    \begin{cases}
 \frac{\partial \Delta_r\psi}{\partial t}
    =\mu\Delta_{r}^2\psi+
     \left(\frac{1}{r}\frac{\partial \psi}{\partial \theta}\frac{\partial}{\partial r}-\frac{1}{r}
\frac{\partial \psi}{\partial r}
\frac{\partial}{\partial \theta}\right)
 \Delta_r\psi,\\
\partial_{\theta}\psi=\partial_r^2 \psi-
\left(\frac{1}{r}-\frac{\alpha}{\mu}\right) 
\partial_r \psi=0,\quad r=a,\\
\partial_{\theta}\psi=\partial_r^2 \psi+
\frac{1}{r}\partial_r \psi=0,
 \quad r=b,\\
 \int_{0}^{2\pi}\psi \,d\theta=0,\\
 \psi_{t=0}=\psi_0.
 \end{cases}
\end{align}
We then employ  \eqref{omega-2}
rather than  \eqref{main=p1}-\eqref{average} 
to prove \autoref{criticalcase1}.

In polar coordinates, the annular domain $\Omega$ is then replaced by the domain 
$\Omega_r$ defined as follows
 \begin{align}\label{bianjie0329}
  \begin{aligned}
\Omega_r :=(a,b)\times\T=\left\{\left(r,\theta\right)|a< r< b, 0\leq\theta\leq 2\pi\right\}.
  \end{aligned}
\end{align}
To use the center unstable manifold reduction in \cite{Han2021}, we first introduce some 
functional spaces:
\begin{align*}
&Y(\Omega_r)=\left\{\psi \in L^2(\Omega_r)| \int_{0}^{2\pi}\psi \,d\theta=0
\right\},\\
&Y_0(\Omega_r)=\left\{\psi \in Y(\Omega_r)\cap H^1(\Omega_r)
\right\},\\
&Y_1(\Omega_r)=\left\{\psi \in Y_0(\Omega_r)\cap H^2(\Omega_r)
| \partial_{\theta}\psi|_{r=a,b}=0\right\},\\
&\begin{aligned}
Y_2(\Omega_r)=\bigg\{&\psi \in H^4(\Omega_r)
\cap Y_1(\Omega_r)\bigg|
\partial_r^2 \psi|_{r=b}+
\frac{1}{r}\partial_r \psi|_{r=b}=0,\\
&\partial_r^2 \psi|_{r=a}-
\left(\frac{1}{a}-\frac{\alpha}{\mu}\right) 
\partial_r \psi|_{r=a}=0
\bigg\}.
\end{aligned}
\end{align*}

	Note that $\lambda_1 (\mu_c)=0$ and $\frac{d\lambda_1}{d\mu}\big|_{\mu=\mu_c}<0$ the
PES condition hold at $\mu=\mu_c$, and  $(\mu, \psi)=(\mu_c,0)$ is a bifurcation point of the system 
\eqref{omega-2} \cite{Ma2019}. The authors \cite{Han2021} established a general framework for 
stability and bifurcation of a fluid system whose governing equations
have a generic fourth-second order structure. To employ the general framework \cite{Han2021} for the problem \eqref{omega-2}, we define two linear operator $L_{\mu} :Y_2(\Omega_{r})\to Y(\Omega_{r})$ 
and $A:Y_1(\Omega_{r})\to Y(\Omega_{r})$:
\begin{align}
L_{\mu}\psi=\Delta_r^2 \psi,\quad 
A\psi=\Delta_r \psi.
\end{align}
We also define the nonlinear operator $\G:Y_2(\Omega_{r})\times Y_2(\Omega_{r})
\to Y(\Omega_{r})$ as follows 
\begin{align}\label{nonlinear}
\G(\psi,\phi):=
  \left(\frac{1}{r}\frac{\partial \psi}{\partial \theta}\frac{\partial}{\partial r}-\frac{1}{r}
\frac{\partial \psi}{\partial r}
\frac{\partial}{\partial \theta}\right)
 \Delta_r\phi.
\end{align}
With the help of these operators, \eqref{omega-2} can be cast in the abstract framework as
\begin{align}\label{abstract-2}
\begin{cases}
\frac{d\psi}{dt}=\mathcal{L}_{\mu}\psi+\mathcal{G}(\psi,\psi),\\
\psi(0)=\psi_0\in Y_1(\Omega_{r}),
\end{cases}
\end{align}
where operator $\mathcal{L}_{\mu}$ and $\mathcal{G}$ are defined by $
\mathcal{L}_{\mu}=A^{-1}L_{\mu},\quad \mathcal{G}=A^{-1}\G$. Please note that the abstract system \eqref{abstract-2} is an equivalent system of \eqref{abstract-1}.

Note that $A$ is an isomorphism between $Y_1(\Omega_{r})$ and  $A(Y_1(\Omega_{r}))\subset Y(\Omega_{r})$, and $L_{\mu}(Y_2(\Omega_{r}))\subset A(Y_1(\Omega_{r}))$,
hence the operator $\mathcal{L}_{\mu} := A^{-1}\circ L_{\mu}$ is bounded from $Y_2(\Omega_{r})$ into $Y_1(\Omega_{r})$. Furthermore, due to the classical Sobolev embeddings, the inclusion $Y_2\left(\Omega_{r}\right)\hookrightarrow Y_1(\Omega_{r})$ is dense and compact, and thus $\mathcal{L}_{\mu} :\: D(\mathcal{L}_{\mu} )=Y_2\left(\Omega_{r}\right)\subset Y_1\left(\Omega_{r}\right) \to Y_1(\Omega_{r})$ has a compact resolvent. Thus, the problem \eqref{abstract-1} satisfies all the necessary conditions for the dynamic bifurcation theory of nonlinear evolution equations, as considered in the literature \cite{Ma2019, Han2021}.

Denote now the center-unstable space 
associated with PES condition $\lambda_1 (\mu_c)=0$ and $\frac{d\lambda_1}{d\mu}\big|_{\mu=\mu_c}<0$
by $H_{c}$. We get from the literature \cite{Ma2019, Han2021} that $H_c$ is the linear span of all first eigenfunctions with the corresponding first eigenvalue changing its sign at critical $\mu=\mu_c$.	
There are two different first eigenvectors 
$\Psi_1(r)\sin \theta $  and $\Psi_1(r)\cos \theta$, where $\Psi_1(r)$ satisfies
   \begin{align}\label{polar-l-eigen-4}
  \begin{cases}
  \mu\Delta_{1}^2\Psi_1=\lambda_1\Psi_1,\\
 \Psi_1= \Psi''_1-
\left(\frac{1}{r}-\frac{\alpha}{\mu}\right) 
\Psi'_1=0,\quad r=a,\\
 \Psi_1=\Psi''_1+\frac{1}{r}\Psi'_1=0,\quad r=b.
  \end{cases}
\end{align}

Note that any linear combination $x_1\Psi_1\cos \theta+x_2\Psi_1\sin \theta$ can be rewritten as
   \begin{align*}
&x_2\Psi_1\sin \theta+x_1\Psi_1\cos \theta=
z\psi_1+\overline{z\psi_1},\\&
	\psi_1=\Psi_1e^{i\theta},\quad z=\frac{x_1}{2}-\frac{x_2}{2}i.
\end{align*}
The use of complex expressions can bring convenience in calculations, we then employ $z\psi_1+\overline{z\psi_1}$ instead of $x_2\Psi_1\sin \theta+x_1\Psi_1\cos \theta$
in the subsequent analysis.

 Let us denote the two-dimensional center-unstable space by
\begin{equation} \label{Hc}
  H_c=\left\{ z\psi_1+\overline{z\psi_1} \, \mid \, z\in \mathbb{C}\right\},
\end{equation}
The center manifold function $h$, by definition, is defined locally from $H_c$ to the stable space $H_s:=(H_c)^\perp$ and its graph is tangent to $H_c$ at the origin.

Let us decompose the solution $u$ into its center-unstable and stable parts
\begin{align}\label{decomp}
u=u_c+u_s,\quad u_c\in H_c, \quad u_s\in H_s,
\end{align}
where $u_c = P_c (u)$ and $u_s = P_s(u)$, with $P_c$ and $P_s$ denoting the projections of the phase space $H$ onto the center-unstable space $H_c$ and the stable space $H_s$ respectively.

By \eqref{Hc}, we can further write the center-unstable part of the solution as
\begin{align}\label{stable-space}
  u_c(t)  = z(t)\psi_1+\overline{z(t)\psi_1} .
\end{align}
We rewrite the abstract equation \eqref{abstract-2} as
\begin{align} \label{abstract equ1}
\frac{du}{dt}=A^{-1}Lu+A^{-1}\G(u).
\end{align}
and then project it to the center-unstable and stable parts to obtain
\begin{align}\label{equeee1}
\begin{aligned}
&\frac{du_c}{dt}=A^{-1}Lu_c+P_cA^{-1}\G(u_s+u_c,u_s+u_c),\\
&\frac{du_s}{dt}=A^{-1}Lu_s+P_sA^{-1}\G(u_s+u_c,u_s+u_c).
\end{aligned}
\end{align}

Taking the inner product of the first equation above with the eigenvector $\psi_1$, obtain
\begin{align} \label{abstract equ2}
\begin{aligned}
&\frac{dz}{dt}=\lambda_1z+\left(A^{-1}\G(u_s+u_c,u_s+u_c),\psi_1\right),\\
&\frac{du_s}{dt}=A^{-1}Lu_s+P_sA^{-1}\G(u_s+u_c,u_s+u_c).
\end{aligned}
\end{align}
By the center manifold theory, the dynamics of the whole system, near the criticality, is captured by the dynamics of the amplitudes $z$ of the modes spanning the center space, that is by the first two equations above. Now, our purpose is to obtain a closed form for the ODE of $z$ in \eqref{abstract equ2} by approximating the center manifold $h(u_c)$ and then making the substitution 
\begin{align}\label{decomp-2}
u_s = h(u_c)
\end{align}
 in the above equations. That is, we need to study the dynamics on the center manifold.

Note that
\begin{align} \label{manifold}
\frac{d h(u_c)}{dt}=\frac{dh(u_c)}{du_c}\frac{d u_c}{dt}
=\frac{dh(u_c)}{du_c}\left(
\frac{dz}{dt}\psi_1+\frac{d\overline{z}}{dt}\overline{\psi_1}
\right).
\end{align}
By using the expressions for $\dfrac{du_c}{dt}$ and $\dfrac{du_s}{dt}$ in \eqref{equeee1}, the above equation becomes
\begin{equation} \label{comp1}
  \begin{aligned}
  &A^{-1}Lh(u_c)+P_sA^{-1}\G(h(u_c)+u_c,h(u_c)+u_c)\\&
  =\frac{dh(u_c)}{du_c}\times\left(A^{-1}Lu_c+P_cA^{-1}\G(h(u_c)+u_c,h(u_c)+u_c)\right).
  \end{aligned}
\end{equation}

By the tangency of the center manifold function at the origin, we have the approximation
\begin{equation} \label{tangency}
  h(u_c)=h_2(u_c)+o(\abs{u_c}^2)
  =\sigma (u_c)^2+o(\abs{u_c}^2)
  =g_{11}z^2+g_{12}\abs{z}^2+g_{22}\overline{z}^2 +o(\abs{z}^2),
\end{equation}
where $h_2$ denotes the quadratic terms of the center manifold.
Due to this tangency requirement, to the lowest order, the equation \eqref{comp1} can be written as
\begin{align}
\begin{aligned}
A^{-1}Lh_2(u_c)+P_sA^{-1}\G(u_c,u_c)=
2\sigma \lambda_1(u_c)^2
\end{aligned}
\end{align}
from which one obtains the equations governing the coefficients of the center manifold function as
\begin{align}\label{hig-order1}
\begin{cases}
A^{-1}Lg_{11}-2g_{11}\lambda_1=
-P_sA^{-1}\G(\psi_1,\psi_1),\\
A^{-1}Lg_{12}-2g_{12}\lambda_1=-
P_sA^{-1}\left(\G(\psi_1,\overline{\psi_1})+\G(\overline{\psi_1},\psi_1)\right),\\
A^{-1}Lg_{22}-2g_{22}\lambda_1=-P_sA^{-1}\G(\overline{\psi_1},\overline{\psi_1})
\end{cases}
\end{align}
where $g_{11},g_{12}$ and $g_{22}$ are subject to the conditions \eqref{cond-3}.

Now our purpose is to obtain the exact expressions of the coefficients $g_{ij}$ of the center manifold function from these equations. For this reason, we need to obtain the exact expressions of the nonlinear interactions of the critical modes, which we list below
\begin{align}\label{nonlinear-gg}
\begin{aligned}
&\G(\psi_1,\psi_1)=e^{2i\theta}i
  \left(\frac{\Psi_1}{r}\frac{d}{dr}-\frac{1}{r}
\frac{d\Psi_1}{dr}\right)
 \Delta_1\Psi_1=\overline{\G(\overline{\psi_1},\overline{\psi_1})}\in H_s,\\
&\G(\psi_1,\overline{\psi_1})=
i
  \left(\frac{\Psi_1}{r}\frac{d}{dr}+\frac{1}{r}
\frac{d\Psi_1}{dr}\right)
 \Delta_1\Psi_1=\overline{\G(\overline{\psi_1},\psi_1)}.
 \end{aligned}
\end{align}

Making use of \eqref{nonlinear-gg}, we see that 
\[
A^{-1}Lg_{12}-2g_{12}\lambda_1=0
\]
which and $\lambda_{1}\left(\mu_{c}\right)=0$ yields that 
$g_{12}=0$. In addition, $g_{22}=\overline{g_{11}}$, and 
\[
g_{11}=e^{2i\theta}G_{11},
\]
where $G_{11}$ solves the following system 
 \begin{align}\label{polar-l-eigen-4-cond2}
  \begin{cases}
  \mu\Delta_{2}^2G_{11}-2\lambda_1\Delta_{2}G_{11}=-i\left(\frac{\Psi_1}{r}\frac{d}{dr}-\frac{1}{r}
\frac{d\Psi_1}{dr}\right)\Delta_{1}\Psi_{1},\\
G_{11}= G_{11}''-
\left(\frac{1}{r}-\frac{\alpha}{\mu}\right) 
G_{11}'=0,\quad r=a,\\
G_{11}=G_{11}''+\frac{1}{r}G_{11}'=0,\quad r=b.
  \end{cases}
\end{align}

Taking 
\[
 h(u_c)=h_2+o(\abs{z}^2)=e^{2i\theta}G_{11}z^2+e^{-2i\theta}\overline{G_{11}}\overline{z}^2+o(\abs{z}^2)
\]
and substituting it into $\eqref{abstract equ2}_1$, we get 
\begin{align} \label{abstract equ3}
&\begin{aligned}
\frac{dz}{dt}=&\lambda_1z+\left(A^{-1}\G(u_c,u_c),\psi_1\right)\\
&+\left(A^{-1}\G(u_c,h_2),\psi_1\right)
+\left(A^{-1}\G(h_2,u_c),\psi_1\right)
+o\left(\abs{z}^{3}\right).
\end{aligned}
\end{align}
Note that the quadratic terms in the reduced equations
\eqref{abstract equ3} are given by
\begin{align} \label{Coefficient}
\begin{aligned}
\left(A^{-1}\G(u_c,u_c),\psi_1\right)=0,
\end{aligned}
\end{align}
 thanks to the fact that
\[
P_cA^{-1}\G(u_c,u_c)=\mathbf{0}.
\]
Thus we need to derive the exact expressions of the cubic terms in \eqref{abstract equ3} which are given by
\begin{align*}% \label{Coefficient2}
\begin{aligned}
&\left(A^{-1}\G(u_c,h_2),\psi_1\right)
+\left(A^{-1}\G(h_2,u_c),\psi_1\right)=lz\abs{z}^2+o(\abs{z}^2)
\end{aligned}
\end{align*}
where we have used
\begin{align}\label{Coefficient2}
\begin{aligned}
l=&\left(A^{-1}\G\left(ze^{i\theta}\Psi_1+\overline{z}
e^{-i\theta}\Psi_1
,e^{2i\theta}G_{11}z^2+e^{-2i\theta}\overline{G_{11}}\overline{z}^2\right),
e^{i\theta}\Psi_1\right)
\\&+\left(A^{-1}\G\left(e^{2i\theta}G_{11}z^2+e^{-2i\theta}\overline{G_{11}}\overline{z}^2,ze^{i\theta}\Psi_1+\overline{z}
e^{-i\theta}\Psi_1\right),
e^{i\theta}\Psi_1\right)
\\=&
\left(A^{-1}\G\left(\overline{z}
e^{-i\theta}\Psi_1
,e^{2i\theta}G_{11}z^2\right),
e^{i\theta}\Psi_1\right)
\\&+\left(A^{-1}\G\left(e^{2i\theta}G_{11}z^2,\overline{z}
e^{-i\theta}\Psi_1\right),
e^{i\theta}\Psi_1\right)
\end{aligned}
\end{align}
and
\begin{align*}% \label{Coefficient2}
&\begin{aligned}
0=&\left(A^{-1}\G\left(ze^{i\theta}\Psi_1,e^{2i\theta}G_{11}z^2+e^{-2i\theta}\overline{G_{11}}\overline{z}^2\right),
e^{i\theta}\Psi_1\right)
\\&+\left(A^{-1}\G\left(e^{2i\theta}G_{11}z^2+e^{-2i\theta}\overline{G_{11}}\overline{z}^2,ze^{i\theta}\Psi_1\right),
e^{i\theta}\Psi_1\right)
\end{aligned}\\
&\begin{aligned}
0=&\left(A^{-1}\G\left(\overline{z}
e^{-i\theta}\Psi_1
,e^{-2i\theta}\overline{G_{11}}\overline{z}^2\right),
e^{i\theta}\Psi_1\right)
\\&+\left(A^{-1}\G\left(e^{-2i\theta}\overline{G_{11}}\overline{z}^2,\overline{z}
e^{-i\theta}\Psi_1\right),
e^{i\theta}\Psi_1\right).
\end{aligned}
\end{align*}

Let us further denote 
\[
z_1=\Re z,\quad z_2=\Im z,
\]
 we then transform the equations \eqref{abstract-1} or the original system \eqref{omega-2}
 on center manifold $\Phi$ into a lower-order approximate system
 of ordinary differential equations (ODEs) in the following form:
	\begin{align}\label{req-m1}
\begin{cases}
	\frac{dz_1}{dt} = \lambda_{1}(\mu) z_1+ lz _1\left(z_1^2+z_2^2\right)+o\left(\abs{z}^{3}\right),\\
		\frac{dz_2}{dt} = \lambda_{1}(\mu) z_2+ l z _2\left(z_1^2+z_2^2\right)+o\left(\abs{z}^{3}\right),
\end{cases}
	\end{align}
As a result, 
the dynamics of
the system \eqref{abstract-1} or \eqref{omega-2}
with $\mu$ in the vicinity of $\mu_c$ is now completely reduced 
to that of the system of ODEs \eqref{req-m1}. In other words, the bifurcation types of
the system \eqref{omega-2} at $\mu=\mu_c$ is determined by 
the bifurcation types of \eqref{req-m1} at $\mu=\mu_c$. We can also get the bifurcation solutions
by analyzing the lower-order system \eqref{req-m1}.
The parameter $l$ in the system \eqref{req-m1}
is called Lyapunov coefficient \cite{Luo1997}, whose sign gives the dynamics near the trivial solution
$(z_1,z_2)=(0,0)$, details is described in the following lemma:
\begin{lemma}
Based on the system  \eqref{req-m1},we have two conclusions:
\begin{enumerate}
\item[ \rm{(1) }] If $l<0$, the bifurcation
at $\mu=\mu_c$  is supercritical.  In this case, the system \eqref{req-m1} admits an infinite number of stable steady-state solutions $(z_1,z_2)=(s_1,s_2)$ for $\mu<\mu_c$. These solutions form a local ring attractor, where $\abs{\mathbf{s}}^2=\left(s_1^2+s_2^2\right)=-\lambda_1/l$, shown in \autoref{region}. 
\item[  \rm{(2) })] If $l>0$,  the bifurcation at $\mu=\mu_c$  is subcritical. Here, the system \eqref{req-m1} has an infinite number of unstable steady-state solutions $(z_1,z_2)=(s_1,s_2)$ for $\mu>\mu_c$.
These solutions constitute a local ring repeller, where $\abs{\mathbf{s}}^2=\left(s_1^2+s_2^2\right)=-\lambda_1/l$, shown in \autoref{region1}.
\end{enumerate}
\end{lemma}
\begin{proof}
The results can be abtianed by consider the following system 
	\begin{align}\label{req-m2}
\begin{cases}
	\frac{dz_1}{dt} = \lambda_{1}(\mu) z_1+ lz _1\left(z_1^2+z_2^2\right),\\
		\frac{dz_2}{dt} = \lambda_{1}(\mu) z_2+ l z _2\left(z_1^2+z_2^2\right).
\end{cases}
	\end{align}
The proof is standard, we omit it here.
\end{proof}

\begin{figure}[tbh]
\centering 
        \begin{tikzpicture}[>=stealth',xscale=1,yscale=1,every node/.style={scale=1.5}]
\draw [thick,->] (0,-2) -- (0,2) ;
\draw [thick,->] (-2,0) -- (2,0) ;
\draw [thick,->] (0,0) -- (0.707,0.707) ;
\draw [thick,->] (0,0) -- (-0.707,0.707) ;
\draw [thick,->] (0,0) -- (-0.707,-0.707) ;
\draw [thick,->] (0,0) -- (0.707,-0.707) ;
\draw [thick,->] (1.5,1.5) -- (0.707,0.707) ;
\draw [thick,->] (-1.5,1.5) -- (-0.707,0.707) ;
\draw [thick,->] (-1.5,-1.5) -- (-0.707,-0.707) ;
\draw [thick,->] (1.5,-1.5) -- (0.707,-0.707) ;
\draw [thick,->] (0,0) -- (0.5,0) ;
\draw [thick,->] (0,0) -- (-0.5,0) ;
\draw [thick,->] (0,0) -- (0,0.5) ;
\draw [thick,->] (0,0) -- (0,-0.5) ;
\node [below right] at (2.1,0) {$z_1$};
\node [above right] at (1,0) {$\abs{s}$};
\node [below right ] at (1,0) {$\mathbf{z}_1$};
\node [above right ] at (0,1) {$\mathbf{z}_2$};
\node [below left ] at (-1,0) {$\mathbf{z}_3$};
\node [below right] at (0,-1) {$\mathbf{z}_4$};
\node [right] at (0,2.1) {$z_2$};
\draw[domain = -2:360][draw=blue, ultra thick,samples = 200] plot({cos(\x)}, {sin(\x)});
\draw[fill,red] (1,0) circle [radius=2pt];
\draw[fill,red] (0,1) circle [radius=2pt];
\draw[fill,red] (-1,0) circle [radius=2pt];
\draw[fill,red] (0,-1) circle [radius=2pt];
\draw[fill,black] (0,0) circle [radius=2.5pt];
\end{tikzpicture}
\caption{For the system \eqref{req-m1}, each point $\mathbf{s}$ on the blue circle represents an equilibrium of \eqref{req-m1} as $\mu<\mu_c$, it is stable, which the equilibrium $(0,0)$ becomes unstable. Correspondingly,
$\mathbf{\mathbf{s}}\Psi_1e^{i \theta}+\overline{\mathbf{s}}\Psi_1e^{-i\theta}
+e^{2i\theta}G_{11} \mathbf{s}^2
+e^{-2i\theta}\overline{G_{11}}\overline{\mathbf{s}}^2+o(\abs{\mathbf{s}}^2)$
 represents a stable equilibrium of equations \eqref{omega-2}. The four points $\mathbf{z}_1-\mathbf{z}_4$ given by
$\mathbf{z}_1=\left(\sqrt{-\lambda_1/l},0\right),
\mathbf{z}_2=\left(0,\sqrt{-\lambda_1/l}\right),
\mathbf{z}_3=\left(-\sqrt{-\lambda_1/l},0\right),
\mathbf{z}_4=\left(0,-\sqrt{-\lambda_1/l}\right)$
correspond to four
stable equilibria 
$\mathbf{z}_j\Psi_1e^{i \theta}+\overline{\mathbf{z}_j}\Psi_1e^{-i\theta}
+e^{2i\theta}G_{11} \mathbf{z}_j^2
+e^{-2i\theta}\overline{G_{11}}\overline{\mathbf{z}_j}^2+o(\abs{\mathbf{z}_j}^2)$
 of the system of equations \eqref{omega-2} $(j=1,\cdots,4)$
 }
\label{region}
\end{figure}
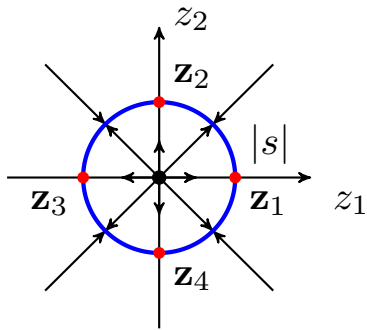

\begin{figure}[tbh]
\centering 
        \begin{tikzpicture}[>=stealth',xscale=1,yscale=1,every node/.style={scale=1}]
\draw [thick,->] (0,-1.5) -- (0,1.5) ;
\draw [thick,->] (-1.5,0) -- (1.5,0) ;
\draw [thick,->] (0.707,0.707)-- (0.3,0.3);
\draw [thick,->]  (-0.707,0.707) -- (-0.3,0.3);
\draw [thick,->] (-0.707,-0.707)-- (-0.3,-0.3) ;
\draw [thick,->] (0.707,-0.707)-- (0.3,-0.3) ;
\draw [thick,-] (0.3,0.3) -- (0,0) ;
\draw [thick,-] (-0.3,0.3) -- (0,0) ;
\draw [thick,-] (-0.3,-0.3) -- (0,0) ;
\draw [thick,-](0.3,-0.3) -- (0,0) ; 
\draw [thick,->] (1,0)  -- (0.5,0) ;
\draw [thick,->] (-1,0) --(-0.5,0) ;
\draw [thick,->] (0,1) -- (0,0.5);
\draw [thick,->] (0,-1) -- (0,-0.5);
\node [below right] at (1.6,0) {$z_1$};
\node [above right] at (1,0) {$\abs{s}$};
\node [right] at (0,1.6) {$z_2$};
\draw[domain = -2:360][draw=black, ultra thick,samples = 200] plot({cos(\x)}, {sin(\x)});
\draw[fill,blue] (0,0) circle [radius=2.5pt];
\end{tikzpicture}
\caption{Each point on circle represents an equilibrium of \eqref{req-m1} as $\mu>\mu_c$, it is unstable, while the trivial equilibrium $(0,0)$ is stable.
 Correspondingly,
$s\Psi_1(r)e^{i\theta}+\overline{s}\Psi_1(r)e^{-i\theta}
+e^{2i\theta}G_{11} \mathbf{s}^2
+e^{-2i\theta}\overline{G_{11}}\overline{\mathbf{s}}^2+o(s^2)$
 represents an unstable equilibrium of the system of equations \eqref{omega-2}.
}\label{region1}
\end{figure}
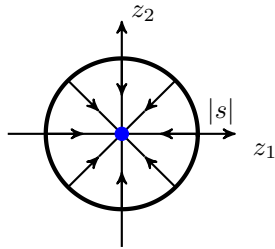
\begin{remark}
Consider any point  $(s_1,s_2)$ on the blue circle and define  $s=s_1+s_2i$,
based on \eqref{decomp} ,\eqref{stable-space} and \eqref{decomp-2}. The corresponding bifurcation solution $\psi_{s}(r,\theta)$ of the problem \eqref{omega-2},
approximated to the second order, takes the following form:
\begin{align}\label{flow-p}
\begin{aligned}
\psi_{s}(r,\theta)=&
s\Psi_1e^{i \theta}+\overline{s}\Psi_1e^{-i\theta}
+e^{2i\theta}G_{11} s^2
+e^{-2i\theta}\overline{G_{11}}\overline{s}^2+o(s^2),\quad \abs{s}=\sqrt{-\lambda_1/l}.
\end{aligned}
\end{align}
We select red points $\mathbf{z}_j$ $(j=1,\cdots,4)$ in \autoref{region}
to analyze bifurcation solutions.
\end{remark}

 \section{Numerical investigations}

The Lyapunov coefficient  $l$, defined in \eqref{Coefficient2}, determines both the type of bifurcation at $\mu=\mu_c$ and the stability or instability of the equilibria of the Navier–Stokes equations under Navier boundary conditions  \eqref{cond-1}-\eqref{cond-2}. To assess the stability of these equilibria bifurcated at $\mu=\mu_c$, we numerically estimate
$l$. For $a=1$ and each $b\in \{4.99,5,5.01,5.02\}$, $\mu_c$ is a function of $\alpha$. For $5\leq\alpha \leq 15$, we compute and plot $\mu_c$, as shown in \autoref{Critical_p}.
From \autoref{Critical_p}, we observe that the Lyapunov coefficient $l$ is negative for $(\alpha,b)$ at some parameter region and positive for others. Therefore, both supercritical and subcritical bifurcations occur in the Navier–Stokes system at $\mu=\mu_c$ under the mixed boundary conditions \eqref{cond-1}-\eqref{cond-2}. Hence, both
first and second conclusions of \autoref{criticalcase} occur in the Navier-Stokes equations 
\eqref{main=p1}-\eqref{average}.

\begin{figure}[h]
  \centering
  \includegraphics[width=.6\textwidth,height=.4\textwidth]{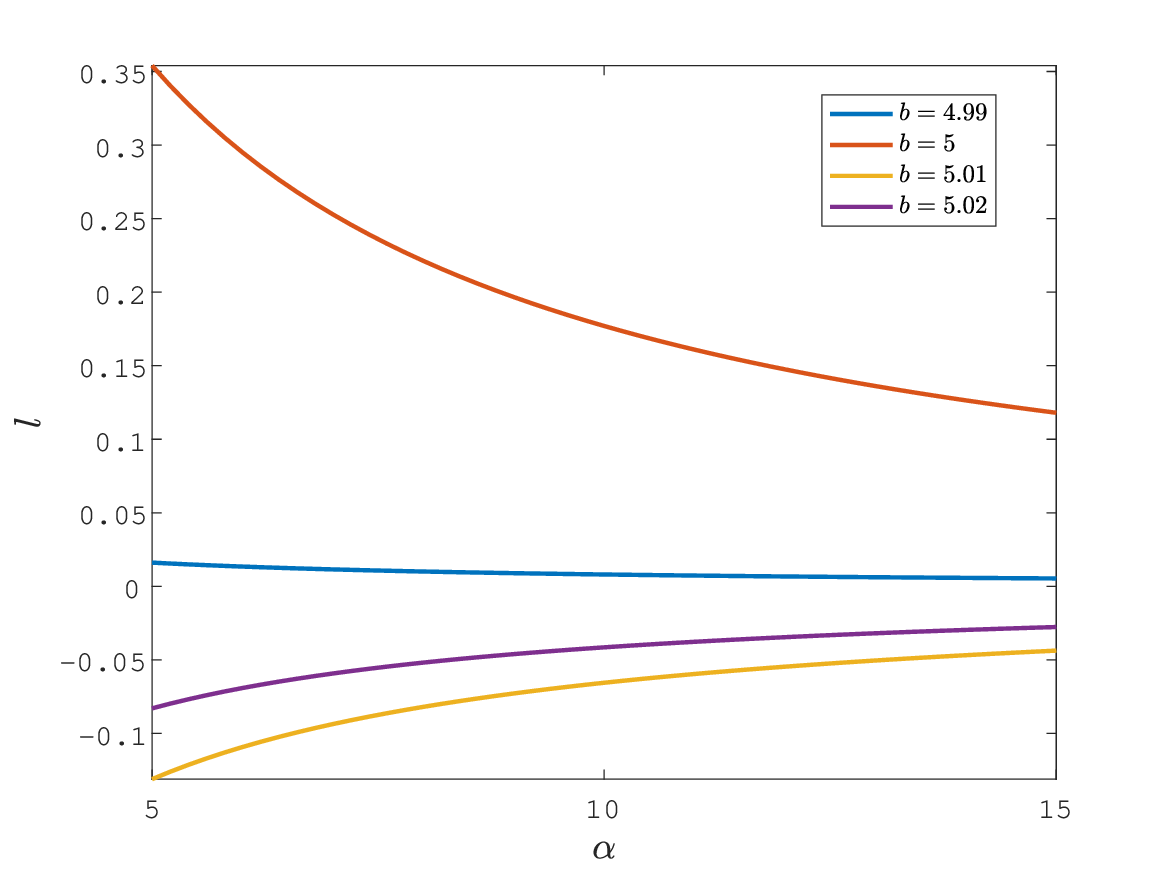}\\
  \caption{The Lyapunov coefficient $l$ as a function of $\alpha$
  with different $b\in [5,15]$ with fixed $a=1$.
  }\label{Critical_p}
\end{figure}

Given that both supercritical and subcritical bifurcations can occur in the Navier–Stokes equations
at $\mu=\mu_c$ under mixed boundary conditions \eqref{cond-1}-\eqref{cond-2},
 a fundamental follow-up question is to determine the regions on the parameter plane
 $(\alpha,b)$ that give rise to different types of bifurcation, which is corresponding to
 conclusions (1) and (2) of \autoref{criticalcase}. From a numerical perspective, once the parameters $(\alpha,b)$ are specified, evaluating the map  $(\alpha,b)\to l$ is relatively straightforward, albeit computationally intensive. Therefore, the task of identifying these regions can be effectively executed using a bisection method without major issues. This process yields a curve in the parameter space, corresponding approximately to $l$, that represents an effective boundary between the region where a supercritical/subcritical bifurcation occurs. The results are shown in \autoref{Critical_p22}. For parameter $(\alpha,b)$ in the yellow region of
 \autoref{Critical_p22}, the bifurcation of the system of equations
 \eqref{main=p1}-\eqref{average} or the equivalent problem \eqref{omega-2} at $\mu=\mu_c$ is supercritical.
 That is,  (2) of \autoref{criticalcase} holds. While $(\alpha,b)$ in the green region of
 \autoref{Critical_p22}, 
 the bifurcation is subcritical. Equivalently, (1) of \autoref{criticalcase} holds.
 
 \begin{figure}[h]
  \centering
  \includegraphics[width=.6\textwidth,height=.4\textwidth]{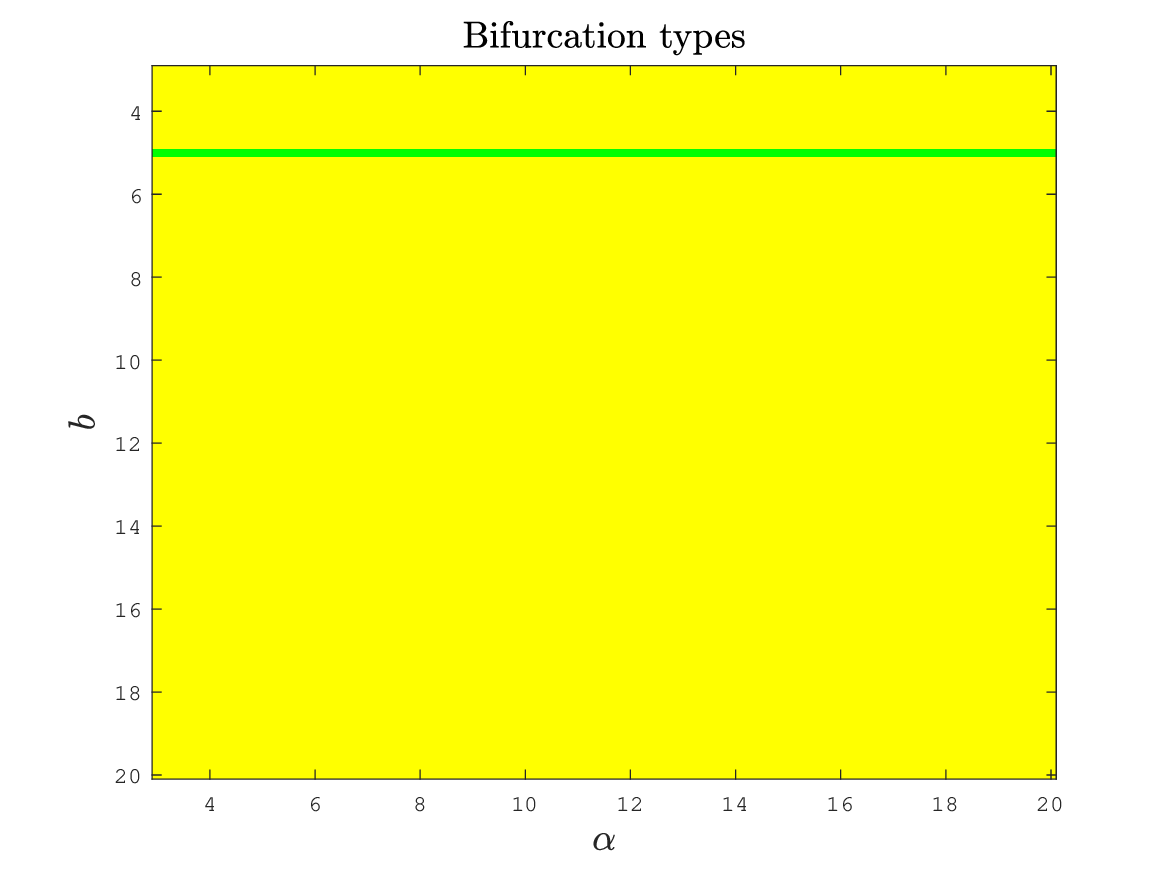}\\
  \caption{
  Yellow region (supercritical bifurcation), green region (subcritical bifurcation) where $a=1$.
  }\label{Critical_p22}
\end{figure}

Taking $a=1, b=3$ and $\alpha=5$, we have $\mu_c= 1.3404$. For $\mu=1.3403$, we have $l=-0.2783$, and 
$\sqrt{-\lambda_1/l}=  0.0578$. On the circle $\abs{\mathbf{s}}=
     0.0578$, 
for the four points $p_j$ with $j=1,2,3,4$, denote the corresponding 
bifurcation solutions as $\psi^{b,j}$, which are given by
\begin{align}\label{ooo}
\begin{aligned}
&\begin{aligned}
\psi^{b,1}\approx&
    0.0578\left(\Psi_1e^{i\theta}+\Psi_1e^{-i\theta}\right)
+  0.0033\left(e^{2i\theta}G_{11} 
+e^{-2i\theta}\overline{G_{11}}\right),
\end{aligned}
\\
&\begin{aligned}
\psi^{b,2}\approx&
    0.0578\left(i\Psi_1e^{i\theta}-i\Psi_1e^{-i\theta}\right)
-  0.0033\left(e^{2i\theta}G_{11}
+e^{-2i\theta}\overline{G_{11}}\right),
\end{aligned}\\
&\begin{aligned}
\psi^{b,3}\approx&-
    0.0578\left(\Psi_1e^{i\theta}+\Psi_1e^{-i\theta}\right)
+ 0.0033\left(e^{2i\theta}G_{11}
+e^{-2i\theta}\overline{G_{11}}\right),
\end{aligned}
\\
&\begin{aligned}
\psi^{b,4}\approx&-
    0.0578\left(i\Psi_1e^{i\theta}-i\Psi_1e^{-i\theta}\right)
- 0.0033\left(e^{2i\theta}G_{11}
+e^{-2i\theta}\overline{G_{11}}\right).
\end{aligned}\\
\end{aligned}
\end{align}
We plot $\psi^{b,j}$ with $j=1,2,3,4$, shown in \autoref{Critical_p11}.

\begin{figure}[h]
  \centering
  \includegraphics[width=.6\textwidth,height=.6\textwidth]{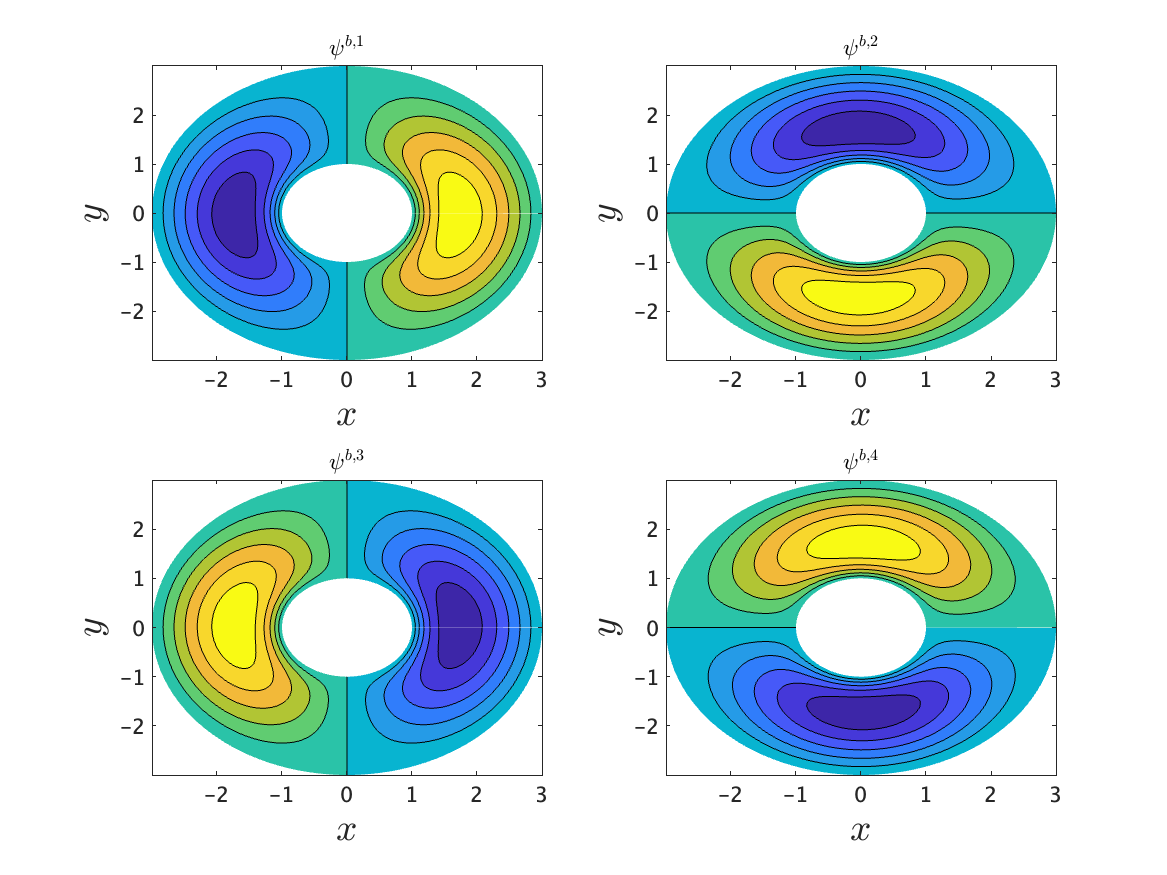}\\
  \caption{Four of bifurcation solutions $\psi_{s}(r,\theta)$
  given by \eqref{ooo}, where
$b=3,\alpha=5$, $\mu_c=1.1924$ and $s=\mathbf{z}_j$ with $j=1,2,3,4$.
  }\label{Critical_p11}
\end{figure}

 Mathematically,$\psi^{b,j}$ with $j=1,2,3,4$  represent four distinct bifurcation solutions. However, based on the analysis presented in \autoref{Critical_p11}, we observe that these solutions share a common flow pattern. Specifically, due to the periodicity of $\theta$, the flow patterns associated with $\psi^{b,j}~(j=1,2,3,4)$ essentially differ only by a translation. In fact, while different points $\mathbf{s}$ on the circle $\abs{\mathbf{s}}=\sqrt{-\lambda_1/l}$, the following expression
 \[
 \psi_{s}(r,\theta)=s\Psi_1(r)e^{i\theta}+\overline{s}\Psi_1(r)e^{-i\theta}
+e^{2i\theta}G_{11} s^2
+e^{-2i\theta}\overline{G_{11}}\overline{s}^2+o(s^2)
\]
 gives a different steady-state solution. However,  the underlying flow patterns of $\psi_{s}(r,\theta)$ for varying $s$ are fundamentally the same, because they only differ by a rotation.

For the Navier-Stokes equations 
\eqref{main=p1}-\eqref{average} or the equivalent problem \eqref{omega-2}, we demonstrate the existence of a viscosity threshold 
$\mu_c$. The exact expression of $\mu_c$ is derived by employing the 
method of separation of variables, as shown in \eqref{uuu}. Specifically, when  $\mu<\mu_c$, the system of equations \eqref{main=p1}-\eqref{average}   bifurcates into an infinite number of nontrivial steady-state solutions. For the majority of parameter values $\alpha$ and $b$, the bifurcation at $\mu=\mu_c$ is supercritical,
all bifurcation solutions consist of a local ring attractor. However, there exists a subcritical bifurcation for certain special parameter regions (green part in \autoref{Critical_p22}) of $(\alpha,b)$ with $a=1$. In that case,
all bifurcation solutions are unstable, which consist of a local ring repeller.

Numerically, we find that 
although these nontrivial steady-state solutions are mathematically different, 
the underlying flow patterns are fundamentally the same, which is essentially due to the periodicity in $\theta$ direction. These results can be used for understanding 
the formation of vortices of fluid
near the viscous boundary described by $\eqref{cond-1}_2$.

 \section{Appendix}
\begin{appendix}

   \begin{lemma}\label{lemma-grad}
   For $\mathbf{u}\in W^{1,p}(\Omega)$ satisfying $\mathbf{u}\cdot \mathbf{n}|_{\partial \Omega}=0$,  there
   exists two constants $ C_a^b$ and $ D_a^b$ such that
             \begin{align}\label{omega}
           \begin{aligned}
 C_a^b\norm{\mathbf{u}}_{L^p}^p\leq \int_{\Omega}\abs{\nabla \mathbf{u}}^p\,dxdy
 +\int_{\partial B(0,b)}\abs{u_{\tau}}^p\,ds,
   \end{aligned}
     \end{align}
     and
                 \begin{align}\label{omega}
           \begin{aligned}
 D_a^b\norm{\mathbf{u}}_{W^{1,p}}^p\leq
 \int_{\Omega}\abs{\nabla \mathbf{u}}^p\,dxdy
 +\int_{\partial \Omega}\abs{u_{\tau}}^p\,ds.
   \end{aligned}
     \end{align}
  \end{lemma}

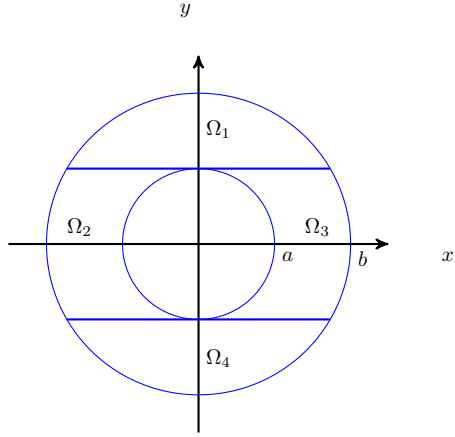
\begin{figure}[tbh]
\centering
        \begin{tikzpicture}[>=stealth',xscale=1,yscale=1,every node/.style={scale=0.8}]
\draw [thick,->] (0,-2.5) -- (0,2.5) ;
\draw [thick,->] (-2.5,0) -- (2.5,0) ;
%\draw [thick,->] (1.4,1.4) -- (0.5,0.5) ;
\node [below right] at (3.1,0) {$x$};
\draw [thick,-,blue] (-1.732,1) -- (1.732,1) ;
\draw [thick,-,blue] (-1.732,-1) -- (1.732,-1) ;
\node [below right] at (1,0) {$a$};
\node [below right] at (2,0) {$b$};
\node [above right] at (0,1.3) {$\Omega_1$};
\node [above left] at (-1.3,0) {$\Omega_2$};
\node [above right] at (1.3,0) {$\Omega_3$};
\node [below right] at (0,-1.3) {$\Omega_4$};
\node [left] at (0,3.1) {$y$};
\draw[domain = -2:360,blue][samples = 200] plot({cos(\x)}, {sin(\x)});
\draw[domain = -2:360,blue][samples = 200] plot({2*cos(\x)}, {2*sin(\x)});
\end{tikzpicture}
\caption{$\Omega_1-\Omega_4$.}
\label{domain14}
\end{figure}

  \begin{proof}
    Let $ \Omega=\cup_{j=1}^4\Omega_j$, shown in \autoref{domain14}.
   For $(x,y)\in \Omega_1\cup  \Omega_4$, we define $l^{j1}_{x,y}$ and $l^{j2}_{x,y}$ to be the two horizontal lines as follows.
   \begin{align}
    \begin{aligned}
   & l^{11}_{x,y}=l^{41}_{x,y}:=\{(s,y)|-\sqrt{b^2-y^2}\leq s\leq x\},\\
   &l^{12}_{x,y}=l^{42}_{x,y}:=\{(s,y)| -\sqrt{b^2-y^2}\leq s\leq \sqrt{b^2-y^2}\},
   \end{aligned}
      \end{align}
      For $(x,y)\in \Omega_2$, we define $l^1_{x,y}$ and $l^2_{x,y}$
    to be the two
   horizontal lines as follows,
   \begin{align}
   \begin{aligned}
   &l^{21}_{x,y}:=\{(s,y)|-\sqrt{b^2-y^2}\leq s\leq x\},\\
    &l^{22}_{x,y}:=\{(s,y)| -\sqrt{b^2-y^2}\leq s\leq -\sqrt{a^2-y^2}\}.
       \end{aligned}
   \end{align}
     For $(x,y)\in \Omega_3$, we define $l^1_{x,y}$ and $l^2_{x,y}$ to be the two horizontal lines as follows,
  \begin{align}
   \begin{aligned}
   &l^{31}_{x,y}:=\{(s,y)|x\leq s\leq \sqrt{b^2-y^2}\},\\
   & l^{32}_{x,y}:=\{(s,y)| \sqrt{a^2-y^2}\leq s\leq \sqrt{b^2-y^2}\}.
      \end{aligned}
   \end{align}
We take $j=1$ and $j=2$ as an example. The fundamental theorem of calculus yields 
   \begin{align}
 \mathbf{u}(x,y)= \mathbf{u}(-\sqrt{b^2-y^2},y)+
 \int_{l^{11}_{x,y}}\partial_x \mathbf{u}\,ds.
   \end{align}
   and 
\begin{align}
 \mathbf{u}(x,y)= \mathbf{u}(-\sqrt{b^2-y^2},y)+
 \int_{l^{21}_{x,y}}\partial_x \mathbf{u}\,ds.
   \end{align}
  which gives
        \begin{align}\label{omega-14}
           \begin{aligned}
&\int_{\Omega_1}\abs{\mathbf{u}(x,y)}^p\,dxdy
=\int_{a}^{b}
\left(\int_{-\sqrt{b^2-y^2}}^{\sqrt{b^2-y^2}}
\abs{\mathbf{u}(x,y)}^p\,dx\right)\,dy
\\&\leq 2^p\sqrt{b^2-a^2}\int_{\partial B(0,b)}\abs{u_{\tau}}^p\,ds
+ \int_{\Omega_1}\left(
 (2\sqrt{b^2-y^2})^{p-1}
 \int_{l^{12}_{x,y}}\abs{\partial_x \mathbf{u}}^p\,ds\right)\,dxdy
 \\&\leq   2^p\sqrt{b^2-a^2}\int_{\partial B(0,b)}\abs{u_{\tau}}^p\,ds+
 \int_{\Omega_1}\left((2\sqrt{b^2-y^2})^{p-1}
 \int_{l^{12}_{x,y}}\abs{\nabla \mathbf{u}}^p\,ds\right)\,dxdy
\\& \leq  2^p\sqrt{b^2-a^2}\int_{\partial B(0,b)}\abs{u_{\tau}}^p\,ds+2b(2\sqrt{b^2-a^2})^{p-1}
 \int_{\Omega_1}\abs{\nabla \mathbf{u}}^p\,dxdy.
   \end{aligned}
     \end{align}
 and 
         \begin{align}\label{omega-23}
           \begin{aligned}
&\int_{\Omega_2}\abs{\mathbf{u}(x,y)}^p\,dxdy
=\int_{-a}^{a}
\left(\int_{-\sqrt{b^2-y^2}}^{-\sqrt{a^2-y^2}}
\abs{\mathbf{u}(x,y)}^p\,dx\right)\,dy
\\&\leq 2^p\sqrt{b^2-a^2}\int_{\partial B(0,b)}\abs{u_{\tau}}^p\,ds
+ \int_{\Omega_2}\left(
 (\sqrt{b^2-a^2})^{p-1}
 \int_{l^{21}_{x,y}}\abs{\partial_x \mathbf{u}}^p\,ds\right)\,dxdy
 \\&\leq   2^p\sqrt{b^2-a^2}\int_{\partial B(0,b)}\abs{u_{\tau}}^p\,ds+
 \int_{\Omega_2}\left((\sqrt{b^2-a^2})^{p-1}
 \int_{l^{21}_{x,y}}\abs{\nabla \mathbf{u}}^p\,ds\right)\,dxdy
\\& \leq  2^p\sqrt{b^2-a^2}\int_{\partial B(0,b)}\abs{u_{\tau}}^p\,ds+2b(\sqrt{b^2-a^2})^{p-1}
 \int_{\Omega_2}\abs{\nabla \mathbf{u}}^p\,dxdy.
   \end{aligned}
     \end{align}

  Hence, in view of \eqref{omega-14} and \eqref{omega-23}, there exists a positive constant $C_a^b$
  such that
           \begin{align}\label{omega}
           \begin{aligned}
\int_{\Omega}\abs{\mathbf{u}(x,y)}^p\,dxdy\leq
C_a^b\left(
\int_{\partial B(0,b)}\abs{u_{\tau}}^p\,ds+
 \int_{\Omega}\abs{\nabla \mathbf{u}}^p\,dxdy
 \right).
   \end{aligned}
     \end{align}
     \end{proof}   

\begin{lemma}\label{lemma-grad-1-1}   
      For $\mathbf{v}, \mathbf{w}\in H^{2}(\Omega)$ satisfying boundary conditions \eqref{cond-1}-\eqref{cond-2} and $\nabla \cdot \mathbf{v}=0$,  we have
       \begin{align}
    \begin{aligned}
 -\int_{\Omega}\Delta \mathbf{v}\cdot\mathbf{w}\,dx\,dy
   =& \int_{\Omega}\nabla\mathbf{v}:\nabla\mathbf{w}\,dx\,dy
   +\int_{\partial B(0,a)}\left(a^{-1}-\frac{\alpha}{\mu}\right)v_{\mathbf{\tau}} w_{\mathbf{\tau}}\,ds\\
     &+\frac{1}{b}\int_{\partial B(0,b)}v_{\mathbf{\tau}} w_{\mathbf{\tau}} \,ds.
   \end{aligned}
  \end{align}
\end{lemma}
 \begin{proof}
 Let us denote $(\partial_1,\partial_2):=(\partial_x,\partial_y)$.
 Taking integration by parts and projecting $\mathbf{v}$ onto its
tangential and normal part on the boundary, i.e.,
\begin{align*}
&\mathbf{v}=v_{\mathbf{n}}\mathbf{n}+
v_{\tau}\mathbf{\tau},\quad 
v_{\mathbf{n}}=\mathbf{v}\cdot\mathbf{n},\quad
v_{\tau}=\mathbf{v}\cdot\mathbf{\tau},
\\ &\mathbf{w}=w_{\mathbf{n}}\mathbf{n}+
w_{\tau}\mathbf{\tau},\quad 
w_{\mathbf{n}}=\mathbf{w}\cdot\mathbf{n},\quad
w_{\tau}=\mathbf{w}\cdot\mathbf{\tau},
\end{align*}
we have
 \begin{align}
  \begin{aligned}
  -\int_{\Omega}\Delta \mathbf{v}\cdot\mathbf{w}\,dx\,dy&=
  -\int_{\partial\Omega}w_j n_i\partial_iv_j\,dx\,dy
  +\int_{\Omega}\nabla\mathbf{v}:\nabla\mathbf{w}\,dx\,dy\\
 & =-\int_{\partial\Omega}w_k\tau_k\tau_jn_i\partial_iv_j\,ds-
  \int_{\partial\Omega}w_kn_kn_jn_i\partial_iv_j\,ds
  +\int_{\Omega}\nabla\mathbf{v}:\nabla\mathbf{w}\,dx\,dy\\
   &=-\int_{\partial\Omega}w_k\tau_k\tau_jn_i\left(\partial_iv_j
   +\partial_jv_i\right)\,ds+\int_{\partial\Omega}w_k\tau_k\tau_jn_i\partial_jv_i\,ds\\
  &\quad +\int_{\Omega}\nabla\mathbf{v}:\nabla\mathbf{w}\,dx\,dy\\
  &=\int_{\Omega}\nabla\mathbf{v}:\nabla\mathbf{w}\,dx\,dy
   +\int_{\partial B(0,a)}\left(a^{-1}-\frac{\alpha}{\mu}\right)w_{\mathbf{\tau}}v_{\mathbf{\tau}} \,ds\\
     &\quad+\frac{1}{b}\int_{\partial B(0,b)}w_{\mathbf{\tau}}v_{\mathbf{\tau}} \,ds
 \end{aligned}
  \end{align}
  where we have used $\mathbf{v}\cdot\mathbf{n}|_{\partial\Omega}=0$, $\mu\left(\nabla\mathbf{v}+\left(\nabla\mathbf{v}\right)^{\rm{T_{r}}}\right)\cdot\mathbf{n}\cdot\mathbf{\tau}|_{x^2+y^2=a^2}=\alpha\mathbf{v}\cdot\mathbf{\tau}|_{x^2+y^2=a^2}$, $\nabla \times   \mathbf{v}|_{x^2+y^2=b^2}=0$,
  $\mathbf{w}\cdot\mathbf{n}|_{\partial\Omega}=0$, $\nabla \times   \mathbf{w}|_{x^2+y^2=b^2}=0$
   and
 \begin{align}\label{oooo}
  \begin{aligned}
 & \int_{\partial B(0,a)}w_k\tau_k\tau_jn_i\left(\partial_iv_j
   +\partial_jv_i\right)\,ds   =
  \frac{\alpha}{\mu} \int_{\partial B(0,a)}v_{\mathbf{\tau}}w_{\mathbf{\tau}}  \,ds\\
  &  \int_{\partial B(0,b)}w_k\tau_k\tau_jn_i\left(\partial_iv_j
   +\partial_jv_i\right)\,ds   =-
  \frac{2}{b} \int_{\partial B(0,b)}v_{\mathbf{\tau}}w_{\mathbf{\tau}} \,ds,\\
  &\int_{\partial B(0,a)}w_k\tau_k\tau_jn_i\partial_jv_i\,ds
  =\frac{1}{a}
   \int_{\partial B(0,a)}v_{\mathbf{\tau}}w_{\mathbf{\tau}} \,ds,\\
    &\int_{\partial B(0,b)}w_k\tau_k\tau_jn_i\partial_jv_i\,ds
  =-\frac{1}{b}
   \int_{\partial B(0,b)}v_{\mathbf{\tau}}w_{\mathbf{\tau}} \,ds.
 \end{aligned}
  \end{align}
In the process deriving the \eqref{oooo}, we have utilized
  \[
\tau_jn_i\partial_jv_i|_{\partial B(0,a)}=
 \mathbf{n}\cdot(\mathbf{\tau}\cdot\nabla) (v_{\mathbf{\tau}}\mathbf{\tau})
|_{\partial B(0,a)}
=\left(\mathbf{n}\cdot(\mathbf{\tau}\cdot\nabla) \mathbf{\tau}
|_{\partial B(0,a)}\right)
v_{\mathbf{\tau}}=\frac{1}{a}v_{\mathbf{\tau}},
  \]
  and
      \[
\tau_jn_i\partial_jv_i|_{\partial B(0,b)}=
 \mathbf{n}\cdot(\mathbf{\tau}\cdot\nabla) (v_{\mathbf{\tau}}\mathbf{\tau})
|_{\partial B(0,b)}
=\left(\mathbf{n}\cdot(\mathbf{\tau}\cdot\nabla) \mathbf{\tau}
|_{\partial B(0,b)}\right)
v_{\mathbf{\tau}}=-\frac{1}{b}v_{\mathbf{\tau}}.
  \]
  
 \end{proof}
      \begin{lemma}\label{lemma-grad-1}
   For $\mathbf{v}\in W^{k,p}(\Omega)$ satisfying
   $\mathbf{v} \cdot   \mathbf{n}|_{x^2+y^2=a^2, b^2}=0$
   and $\nabla \cdot \mathbf{v}=0$, we have
            \begin{align}
           \begin{aligned}
\norm{\nabla \mathbf{v}}_{W^{k,p}}\leq C\norm{\omega}_{W^{k,p}}
,\quad 1<p<\infty.
   \end{aligned}
     \end{align}
  \end{lemma}
   \begin{proof}
   Let $v_1=-\partial_y\psi$, $v_2=\partial_x\psi$ and $\omega=( -\partial_y,\partial_x)\cdot\mathbf{v}$, then we have
               \begin{align}
           \begin{cases}
\Delta \psi=\omega, \quad (x,y)\in \Omega,\\
 \psi=\beta_b,\quad \quad x^2+y^2=b^2,\\
  \psi=\beta_a,\quad \quad x^2+y^2=a^2,
   \end{cases}
     \end{align}
     where $\beta_a$ and $\beta_b$ are two constants. Replacing $\psi$ by
     \[
     \psi+\beta_b- \frac{(\beta_b-\beta_a)\ln b}{\ln\left(\frac{b}{a}\right)}
     +
     \frac{\beta_b-\beta_a}{\ln \left(\frac{b}{a}\right)}\ln r ,
     \]
    we get that
               \begin{align}
           \begin{cases}
\Delta \psi=\omega, \quad (x,y)\in \Omega,\\
 \psi=0,\quad \quad x^2+y^2=b^2,\\
  \psi=0,\quad \quad x^2+y^2=a^2.
   \end{cases}
     \end{align}
     By the elliptic regularity theory, we conclude
                 \begin{align}
           \begin{aligned}
\norm{\nabla \mathbf{v}}_{W^{k,p}}\leq \norm{\psi}_{W^{k+2,p}} \leq C\norm{\omega}_{W^{k,p}}.
   \end{aligned}
     \end{align}
   \end{proof}
 \begin{lemma}\label{epsilon-1}
 For any $\epsilon>0$ and $a_1, a_2\in \R$, there exists $C_{\epsilon}$ depending on
 $a_1, a_2, a$ and $b$
 such that
\begin{align}\label{epsilon}
\begin{aligned}
&\abs{a_1\int_{\partial B(0,b)}\left(\tau \cdot \mathbf{v}\right)^2 \,ds
+a_2\int_{\partial B(0,a)}\left(\tau \cdot \mathbf{v}\right)^2 \,ds}
\\&\leq \epsilon \int_{\Omega}\abs{\nabla \mathbf{v}}^2\,dx\,dy
+C_{\epsilon} \int_{\Omega}\abs{\mathbf{v}}^2\,dx\,dy
\end{aligned}
\end{align}
\end{lemma}
\begin{proof}
Let us define $\alpha_1$ and $\alpha_2$ as follows
\[
\alpha_1=\frac{a_1+a_2}{b-a},\quad
\alpha_2=\frac{-aa_1-ba_2}{b-a}.
\]
Then, we have
\[
\alpha_1b+\alpha_2=a_1,\quad
\alpha_1a+\alpha_2=-a_2
\]
by which we have
\[
\begin{aligned}
&\abs{a_1\int_{\partial B(0,b)}\left(\tau \cdot \mathbf{v}\right)^2 \,ds
+a_2\int_{\partial B(0,a)}\left(\tau \cdot \mathbf{v}\right)^2 \,ds}
\\&=\abs{\int_0^{2\pi}\left(
\int_a^b \frac{\partial}{\partial r}
\left(
\left(\tau \cdot \mathbf{v}\right)^2\left(\alpha_1r+\alpha_2\right)
\right)\,dr
\right)\,d\theta}
\\&\leq \int_0^{2\pi}
\int_a^b \abs{\frac{\partial}{\partial r}
\left(
\left(\tau \cdot \mathbf{v}\right)^2\left(\alpha_1r+\alpha_2\right)
\right)}\,dr\,d\theta
\\&\leq \epsilon \int_{\Omega}\abs{\nabla \mathbf{v}}^2\,dx\,dy
+C_{\epsilon} \int_{\Omega}\abs{\mathbf{v}}^2\,dx\,dy.
\end{aligned}
\]
where we have used Young's inequalities.
\end{proof}

\begin{lemma}\label{Stokes0725}\textbf{Stokes' estimates}. Consider the equation 
\begin{align}\label{biaozhunde0725}
\begin{aligned}
&\mu \Delta \mathbf{v} - \nabla q = \mathbf{f}, \\
&\nabla \cdot \mathbf{v} = 0,
\end{aligned}
\end{align}
subject to the boundary conditions \eqref{cond-3}, where $\mathbf{f}=\left(f_{1},f_{2}\right)\in \left[L^{p}\left(\Omega\right)\right]^2$. Assume that 
$\left(\mathbf{v},q\right)\in \left[W^{1,p}\left(\Omega\right)\right]^2\times L^{p}\left(\Omega\right)$ is the weak solution to the above equation, then we have the following estimates 
\begin{align}\label{manzu0725}
    \left\|\nabla^2\mathbf{v}\right\|_{L^{p}\left(\Omega\right)}
    +\left\|\nabla q\right\|_{L^{p}\left(\Omega\right)}\leq C\left(\Omega,\alpha,\mu\right)
    \left(\left\|\mathbf{f}\right\|_{L^{p}\left(\Omega\right)}+\left\|\mathbf{v}\right\|_{W^{1,p\left(\Omega\right)}}\right).
    \end{align}
\end{lemma}
\begin{proof}
     In what follows, we follow the proof in \cite{Li2025} to obtain the conclusion desired.
We will construct a solution to the following equation 
\begin{align}\label{wufadingyi1020}
\begin{cases}
-\mu\Delta \mathbf{w}+\nabla\Theta=\mathbf{f},
\\
\nabla\cdot \mathbf{w}=0,
\\
\left(\partial_{1}w_{2}-\partial_{2}w_{1}\right)|_{r=b}=\mathbf{w}\cdot\mathbf{n}|_{r=a,b}=0,~
\left(\partial_{1}w_{2}-\partial_{2}w_{1}\right)|_{r=a}
=\left(\frac{2}{r}-\frac{\alpha}{\mu}\right)v_{\theta}|_{r=a},
\\
\left(\partial_{1}w_{2}-\partial_{2}w_{1}\right)|_{r=b}
=\left(\partial_{1}v_{2}-\partial_{2}v_{1}\right)|_{r=b},
\end{cases}
\end{align}
where $\mathbf{w}=\left(w_{1},w_{2}\right)$ and $v_{\theta}=v_{2}\cos{\theta}-v_{1}\sin{\theta}$. 
Obviously, \(\left(\mathbf{v},q\right)\) is a \(p-\)generalized solution of \eqref{wufadingyi1020}. The definition of \(p-\)generalized solution can be found in Definition 2.1 of \cite{Li2025}.  

First, we consider the following auxiliary problem 
\begin{align}\label{fuzhu11020}
\begin{cases}
-\mu\Delta G=\partial_{1}f_{2}-\partial_{2}f_{1},
\\
G|_{\partial\Omega}=
\frac{b-r}{b-a}\left(\frac{2}{r}-\frac{\alpha}{\mu}\right)v_{\theta}
:=U\in W^{1,p}\left(\Omega\right).
\end{cases}
\end{align}
From the elliptic theory, one can conclude that there exists a weak solution 
\(G\) satisfying 
\[
\left\|\nabla G\right\|_{L^{p}\left(\Omega\right)}
\leq C\left(\left\|\mathbf{f}\right\|_{L^{p}\left(\Omega\right)}+\left\|\mathbf{v}\right\|_{W^{1,p}\left(\Omega\right)}\right).
\]
Furthermore, we consider another auxiliary problem 
\begin{align}\label{fuzhu21020}
\begin{cases}
\Delta\Phi=G,
\\
\Phi|_{\partial\Omega}=0.
\end{cases}
\end{align}
Similarly, from the elliptic theory, there exists a \(\Phi\) satisfying the above equation and 
\[
\left\|\nabla^3\Phi\right\|_{L^{p}\left(\Omega\right)}\leq C\left\|\nabla G\right\|_{L^{p}\left(\Omega\right)}
\leq C\left(\left\|\mathbf{f}\right\|_{L^p\left(\Omega\right)}+\left\|\mathbf{v}\right\|_{W^{1,p}\left(\Omega\right)}\right).
\]
Let 
\[
w_{1}=-\partial_{2}\Phi,~w_{2}=\partial_{1}\Phi,
\]
then 
\(\nabla\cdot\mathbf{w}=0\), \(\mathbf{w}\cdot\mathbf{n}|_{r=a,b}=0\) and $\left(\partial_{1}w_{2}-\partial_{2}w_{1}\right)|_{r=b}=0$. Furthermore, 
\(\Delta\Phi=\partial_{1}w_{2}-\partial_{2}w_{1}=G\). Thus, 
\(-\mu\Delta G=-\mu\Delta\left(\partial_{1}w_{2}-\partial_{2}w_{1}\right)=\partial_{1}f_{2}-\partial_{2}f_{1}\). Then, from the field theory(an irrotational field must be  potential), there exists a \(\nabla\Theta\) such that
\[
-\mu\Delta \mathbf{w}+\nabla\Theta=\mathbf{f}.
\]
From the above construction, one can easily verify that \(\left(\mathbf{w},\Theta\right)\) satisfies the equation \eqref{wufadingyi1020}. And 
\begin{align*}
\left\|\nabla^2\mathbf{w}\right\|_{L^{p}\left(\Omega\right)}+\left\|\nabla\Theta\right\|_{L^{p}\left(\Omega\right)}\leq
C\left(\left\|\mathbf{f}\right\|_{L^p\left(\Omega\right)}+\left\|\mathbf{v}\right\|_{W^{1,p}\left(\Omega\right)}\right).
\end{align*}
If we can show that \(\mathbf{v}=\mathbf{w}\) almost everywhere, then the proof is finished. Thus, we will show the uniqueness of solution to the following equation 
\begin{align}\label{weiyixing1020}
\begin{cases}
-\mu\Delta \mathbf{H}+\nabla \Theta_{1}=\mathbf{0},
\\
\nabla\cdot \mathbf{H}=0,
\\
\mathbf{H}\cdot\mathbf{n}|_{r=a,b}=0,~\left(\partial_{1}H_{2}-\partial_{2}H_{1}\right)|_{r=a,b}=0.
\end{cases}
\end{align}
To this end, we consider the following 
auxiliary problem 
\begin{align}\label{renyi1020}
\begin{cases}
-\mu\Delta \mathbf{Q}+\nabla \widetilde{\Theta}=\widetilde{\mathbf{F}},
\\
\nabla\cdot \mathbf{Q}=0,~
\mathbf{Q}\cdot\mathbf{n}|_{r=a,b}=0,~\left(\partial_{1}Q_{2}-\partial_{2}Q_{1}\right)|_{r=a,b}=0,
\end{cases}
\end{align}
where \(\widetilde{\mathbf{F}}\in \left[L^{p'}\left(\Omega\right)\right]^2\) is arbitrary. Following the same steps for the auxiliary problems 
\eqref{fuzhu11020} and \eqref{fuzhu21020} where 
\(U\equiv 0\) and \(\mathbf{f}\) is replaced by \(\widetilde{\mathbf{F}}\), one can obtain the strong solution \(\left(\mathbf{Q},\widetilde{\Theta}\right)\in \left[W^{2,p'}\left(\Omega\right)\right]^2\times \left[L^{p'}\left(\Omega\right)\right]^2\) to the problem \eqref{renyi1020}. Multiplying \(\eqref{weiyixing1020}_{1}\) by \(\mathbf{Q}\) and 
\(\eqref{renyi1020}_{1}\) by \(\mathbf{H}\), then integrating the results over \(\Omega\) by parts, respectively, one can find that 
\[
0=\int_{\Omega}\widetilde{\mathbf{F}}\cdot \mathbf{H}dx,
\]
which implies \(\mathbf{H}=\mathbf{0}\) almost everywhere since 
\(\widetilde{\mathbf{F}}\) is arbitrary. Thus, the uniqueness for the problem \eqref{weiyixing1020} holds. Therefore, we can conclude that \(\mathbf{v}=\mathbf{w}\) almost everywhere. Thus, \(\left(\mathbf{v},q\right)\) satisfies the estimate \eqref{manzu0725}. Finally, from the definition of \(p-\)generalized solution, one can show that \(\left(\mathbf{v},q\right)\) satisfies the boundary condition \eqref{cond-3}.
%     Due to the shape of $\Omega$, it is convenient to consider the equation \eqref{biaozhunde0725} in polar coordinate. Thus, we have 
%     \begin{align*}
%     \begin{cases}
%     \mu\left(\Delta_{r}v_{r}-\frac{v_{r}}{r^2}-\frac{2}{r^2}
%     \frac{\partial v_{\theta}}{\partial \theta}\right)-\frac{\partial q}{\partial r}=f_{1},
%     \\
% \mu\left(\Delta_{r}v_{\theta}-\frac{v_{\theta}}{r^2}+\frac{2}{r^2}
%     \frac{\partial v_{r}}{\partial \theta}\right)-\frac{1}{r}\frac{\partial q}{\partial \theta}=f_{2},
%     \\
% v_{r}|_{r=a,b}=\frac{\partial(rv_{r})}{\partial r}+\frac{\partial v_{\theta}}{\partial\theta}=0,
%     \\
%     \left(\frac{\partial v_{\theta}}{\partial r}-\left(\frac{1}{r}-\frac{\alpha}{\mu}\right)v_{\theta}\right)|_{r=a}=\left(\frac{\partial v_{\theta}}{\partial r}+\frac{v_{\theta}}{r}\right)|_{r=b}=0.
%     \end{cases}
%     \end{align*}
\end{proof}
\end{appendix}

%\bibliography{main-reference}
%\bibliographystyle{unsrt}
 
 \providecommand{\noopsort}[1]{}\providecommand{\singleletter}[1]{#1}%

\end{document}